\documentclass[12pt]{amsart}
\usepackage{amssymb}
\newcommand{\C}{\mathbb{C}}
\newcommand{\R}{\mathbb{R}}
\newcommand{\N}{\mathbb{N}}

\def\be{\begin{equation}}
\def\ee{\end{equation}}
\def\bequnan{\begin{eqnarray*}}
\def\eequnan{\end{eqnarray*}}
\def\la{\label}
\def\d{\mathrm{d}}
\def\von{\varepsilon}

\def\cD{{\mathcal D}}

\def\P1{{\mathbf P}}

\def\cM{{\mathcal M}}
\def\cH{{\sf H}}

\def\HE{{\mathcal H} (E) }
\def\HW{{\mathcal H} (W) }
\def\HG{{\mathcal H} (G) }

\def\cU{{\mathcal U}}
\def\tq{\tilde{q}}

\def\th{\tilde{h}}

\def\tf{\tilde{f}}

\def\wt{\widetilde}

\def\({\left(}
\def\Ran{\operatorname{Ran}}
\def\){\right)}
\def\s2{ \wt{S_2} }

\def\llangle{\left\langle}
\def\rrangle{\right\rangle}

\def\len{\left\|}
\def\rin{\right\|}
\def\yOonequ{{\stackrel{\lower.5mm\hbox{$ \scriptstyle\circ $}}y}_1^{\lower3pt\hbox{$ \scriptstyle 2 $}}}
\def\yOtwoqu{{\stackrel{\lower.5mm\hbox{$ \scriptstyle\circ $}}y}_2^{\lower3pt\hbox{$ \scriptstyle 2 $}}}
\def\yOonequti{\stackrel{\lower1mm\hbox{$ \scriptstyle\circ $}}{\widetilde y}_1^{\lower4pt\hbox{$ \scriptstyle 2 $}}}
\def\yOtwoquti{\stackrel{\lower1mm\hbox{$ \scriptstyle\circ $}}{ \widetilde y  }_2^{\lower4pt\hbox{$ \scriptstyle 2 $}} }
\def\yOone{{\stackrel{\lower.5mm\hbox{$ \scriptstyle\circ $}}y}_1}
\def\yOtwo{{\stackrel{\lower.5mm\hbox{$ \scriptstyle\circ $}}y}_2}
\def\yOoneti{\stackrel{\lower1mm\hbox{$ \scriptstyle\circ $}}{\widetilde y}_1}
\def\yOtwoti{\stackrel{\lower1mm\hbox{$ \scriptstyle\circ $}}{ \widetilde y  }_2 }

\begin{document}

\newcounter{prop}
\newtheorem{theorem}{Theorem}
\newtheorem{proposition}[prop]{Proposition}
\newtheorem{corollary}[prop]{Corollary}
\newtheorem{lemma}[prop]{Lemma}
\newcounter{rem}
\newtheorem{remark}[rem]{Remark}
\newtheorem*{remarknonumb}{Remark}
\newcounter{exerc}
\newtheorem{exercise}[exerc]{Exercise}
\def\theremark{\unskip}
\newtheorem{deffinition}{Definition}
\newtheorem{definition}{Definition}

\def\thedefinition{\unskip}
%\newnumbered{remark}{Remark}

%%%%%%%%%%%%%%%%%%%%%%%%%%%%%%%%%%%%%%%%

\centerline{CANONICAL SYSTEMS AND DE BRANGES SPACES}

\bigskip
\centerline{R. Romanov\footnote{Department of Mathematical Physics, Faculty of Physics, St Petersburg State University, e-mail: morovom@gmail.com }}

\bigskip
\bigskip

This is an exposition of direct and inverse spectral theory of canonical systems based on de Branges spaces. The main source is chapter 2 in \cite{deBr} which will be referred to throughout as the book. The theory can be considered a far-reaching generalization of the Stiltjes algorithm in the inverse spectral theory of Jacobi matrices. It has two important features,

\begin{itemize}

\item Unlike many other 1D inverse spectral theories (transformators, Krein systems), it is not perturbative  -- there is no underlying problem with well-understood eigenfunctions to be compared with.

\item It deals with the most general form of a second order differential operator -- a canonical system, which includes Schr\"odinger and Dirac operators, Jacobi matrices and strings as partial cases, and thus allows to "interpolate" between them. 

\end{itemize}

The main result in the theory establishes correspondence between classes of entire functions or measures on the real line and canonical systems -- first order matrix differential operators parametrized by properly normalized $ 2 \times 2 $ matrix functions. It belongs mostly to L. de Branges. Let us give its informal description (Theorems \ref{inversolution} and \ref{uniqueness} in the present text). 

The two basic objects are an Hermite -- Biehler (HB) function, $ E ( z ) $, that is, an entire function satisfying $ | E ( z ) |> | E ( \overline z ) | $ for $ z \in \C_+ $, and a Hamiltonian, $ \cH ( x ) $,  -- a  locally summable $ 2 \times 2 $ matrix-valued function on interval $ [ 0 , L ) $, $ L \le \infty $, such that $ \cH ( x ) \ge 0 $ a.e. They respectively represent data in inverse and direct problems to be described below. The link between them is the fact (see Proposition \ref{HB}) that if $ Y ( x , \lambda ) $ is the solution of the Cauchy problem
\be\la{c} \begin{pmatrix} 0 & -1 \cr 1 & 0 \end{pmatrix} \frac{\d Y}{\d x} = \lambda \cH Y , \ee
\[ Y ( 0 ) = C ,\] 
 $ C \ne 0 $ being a constant vector with real components, then $ E_x ( \lambda ) = Y_+ ( x , \lambda ) + i Y_- ( x , \lambda ) $ is an HB function of $ \lambda $ for each $ x \in ( 0 , L ) $. When $ L $ is finite and $ \cH $ is summable on $ ( 0 , L ) $ this is also true for $ x = L $. The  function $ E_L $ is then a natural data for the inverse problem related to the equation (\ref{c}) which is called the canonical system. Let us elaborate on that. For a given Hamiltonian $ \cH $ it is possible to construct a Hilbert space and a selfadjoint operator in it in such a way that (\ref{c}) supplemented by a selfadjoint boundary condition at $ x = L $  would become the eigenproblem for this operator. For instance, the zeroes of $ Y_\pm ( L , \cdot ) $ would form the spectra of operators corresponding to the boundary conditions $ F_\pm ( L ) = 0 $ respectively. Obviously, $ Y_\pm ( L , \cdot ) $ are just the real and imaginary parts of the function $ E_L $, and thus, knowing $ E_L $ we know the spectra of two selfadjoint extensions corresponding to two different boundary conditions at the right end (and the same condition at the left end). On the other hand, from Borg theorems in the classical theory of second order differential operators \cite{Marchenko} we know that two spectra is precisely the amount of information necessary to guarantee uniqueness of solution of the inverse problem of restoring the potential from the spectra, hence $ E_L $ as the choice of the spectral data. Now, the main result says that given an HB function, $ E $, satisfying a regularity condition (this is an easy necessary condition encoding that $ L < \infty $) and some normalizations, there exist a finite $ L $ and an essentially unique Hamiltonian $ \cH $ defined on $ ( 0, L ) $ such that $ E = E_L $. 

The exposition of this powerful and final theorem in the book is hard to understand and appears to be superfluously complicated by total omission of motivations and heuristic arguments. A typical example is Theorem 27 in \cite{deBr}: its meaning is perfectly clear if we consider the corresponding statement - a formula for scalar products of special solutions, see Appendix I, theorem \ref{restore} and an explanation in sect \ref{diprobl}, -  in the \textit{direct} spectral theory. The formulation in the book however just states a lengthy identity without any hints as to where it comes from. We aim to remedy these drawbacks and simplify the arguments whenever possible. 

Another set of results exposed here refers to the case $ L = \infty $. In this case the basic object is the Weyl-Titchmarsh $m$-function of the system. The direct spectral theory (Theorem \ref{sptheoremsing}) says that the generalized Fourier transform with respect to the functions $ Y ( x , \lambda ) $ is an isometry from the Hilbert space of the canonical system into $ L^2 ( \R, \d \mu ) $ where $ \mu $ is the measure in the Herglotz integral representation of the $ m $-function. The inverse problem is to find which $ \mu $'s correspond to canonical systems on the semiaxis in the sense above. The answer (Theorem \ref{invsingular}) is -- any measure $ \mu $ such that $ \int_\R \( 1 + t^2 \)^{ -1 } \d \mu ( t ) $ is finite and a simple necessary condition is satisfied. Its proof is based on nesting circles analysis and feels much easier from the viewpoint of spectral theory than the above result for finite $ L $, for it only uses finitedimensional de Branges spaces, and even that can be eliminated. 

The last result to be mentioned is a formula for the type of a canonical system in the case $ L< \infty $ in terms of the Hamiltonian (Theorem \ref{exptypeM}) obtained independently by de Branges and Krein. 

We assume that the readership is made of complex analysts and specialists in spectral theory of differential operators. Readers are expected to be familiar with the extension theory for symmetric operators and with basic growth theory for entire functions. The only assertion to be used beyond the latter is the Krein theorem on entire functions of bounded type (the formulation is given in the Notation and Basics section). Some arguments might have been shortened if we had used the theory of the Cartwright class but we have opted not to do so. One point in the proof of the lattice property of de Branges spaces requires the notion and elementary properties of subharmonic functions which can be found for instance in \cite[Chapter 17]{Rudin}. A point in the proof of isometric embeddings of  de Branges spaces (theorem  \ref{isometry}) uses basic facts about factorization of analytic functions in the half plane. All of them can be found in cf. \cite{Gofman}.  

The exposition is closed, the proofs given are complete, save for one exception. We use but do not prove the axiomatic characterization of de Branges spaces (theorem \ref{axiomatic} in the current text), mainly because its expositions in the book and in \cite{DymMcKean} do not require improvement.

In sections devoted to solving the inverse problems the exposition sometimes oscillates between direct problems and the inverse ones. We often consider which properties should be enjoyed by objects arising in solutions for an inverse problem if we already knew that the solution exists, and try to postulate or establish them in terms of the data for the inverse problem. 

The structure of the text is as follows. We begin by defining the canonical system, see (\ref{can}), and establishing a convenient normalization of it. The convenience of the normalization we use is that the distinction between a regular problem on a finite interval and a singular problem on the semiaxis in the limit circle case disappears when passing to the systems -- the systems on the semiaxis are exactly those corresponding to operators in the limit point case at infinity. Section 1 explains how to put Schr\"odinger and Dirac equations and the string equation in the form of a canonical system. Section 2 gives a rigorous construction of a selfadjoint operator corresponding to the system in cases of a system on a finite interval and on a semiaxis. Section 3 describes the spectral theorem for this operator in terms of de Branges spaces. Section 4 solves the inverse problems for canonical systems with operators of finite rank. Section 5 is devoted to solution of the inverse problem on a finite interval from the first column of the monodromy matrix. Section 6 contains a derivation of the formula for type of the system in terms of the Hamiltonian. Section 7 establishes the uniqueness of a canonical system corresponding to a given regular HB function and contains a formulation of constructive algorithm for solving the inverse problem in the regular case. As an example we derive the Borg theorem for the Schr\"odinger operator on an interval from this uniqueness result. Sections 8 and 9 study the direct and inverse problem on the semiaxis. In Section 8 we study the Fourier transform with respect to generalized eigenfunctions and define its spectral measure. In Section 9 we solve the inverse problem of reconstructing the system from the spectral measure.

There are three appendices. We put in there the results used in the solution of the inverse problems which are proved entirely on "spectral" side and can thus be considered facts of function theory, as opposed to spectral theory of differential operators. These facts are a theorem on restoring the monodromy matrix from its first column (Appendix I), a parametrization of measures $ \mu $ for which a de Branges space is isometrically embedded in $ L^2 ( \R , \mu ) $ and a theorem saying that subspaces of a regular de Branges space are regular (Appendix II), a proof of the lattice property of de Branges spaces (Appendix III). These are followed by a short section on basic notions and facts of theory of de Branges spaces. In the end we give comments on the literature. 

\textbf{Acknowledgements.} These are notes of lectures delivered by the author in 2011--2013 in St Petersburg. The author had many helpful discussions on de Branges spaces with Anton Baranov and Yuri Belov. Roman Bessonov read the whole text and suggested numerous improvements. I am grateful to all of them.  %\textbf{Disclaimer}: in some cases the author is not able to say whether his proof of a fact is actually the same as the one given by de Branges or not. 

\medskip

\textbf{Notation and Basics:} 

$ \langle \cdot , \cdot \rangle $ -- scalar product in a Hilbert space. A  subscript, if present, specifies the space. 

$ H^\infty $ -- the class of bounded analytic functions in  the upper half plane $ \C_+ $; $ \len \cdot \rin_\infty $ stands for the norm in it; $  H^2 $ -- the Hardy class in $ \C_+ $. 

$ L^2 ( \R , \d \mu ) $ -- the space of functions on the real line square summable with respect to a measure $ \mu $. If $ \mu $ is a. c., $ \d \mu = \rho \,	 \d t $ with a locally summable $ \rho $, the $ \d t $ is skipped in the notation for the space.

$ K^X_w $ -- the reproducing kernel at the point $ w $ for a reproducing kernel functional Hilbert space, $ X $, $ f ( w ) = \llangle f , K^X_w \rrangle_X $ for all $ f \in X $.

$ f^* ( z ) \colon = \overline{ f ( \overline z ) } $.

For an entire function, $ g $, $ g_+ \colon = ( g + g^* )/2 $, $ g_- \colon = ( g - g^* )/(2i) $.  

$ \cH ( x ) $ -- a function $ ( 0 , L ) \to {\textrm{Mat}}_2 ( \R ) $, $ 0 < L \le \infty $, such that $ \cH ( x ) \ge 0 $ a. e., and $ \cH \in L^1 ( 0 , L^\prime ) $ for all $ L^\prime < L $.

$ L^2 ( I; \rho ) $ -- the Hilbert space of functions (with values in $ \C^n $, $ n < \infty $) on an interval $ I \subset \R_+ $ defined by the scalar product $ \langle f , g \rangle = \int_I \langle f ( x ) , \rho ( x ) g ( x ) \rangle_{ \C^n } \d x $ for any $ n \times n $-matrix valued function $ \rho ( x ) $, $\rho ( x ) \ge 0 $ a. e., whose elements are $ L^\infty ( I ) $. 
 
$ J= \left( \begin{array}{cc} 0 & -1 \cr 1 & 0 \end{array} \right) $.

For a $ Y \in \C^2 $, $ Y_\pm $ stand for the components of $ Y $, $ Y = \Bigl( \! \begin{array}{c} Y_+ \cr Y_- \end{array} \! \Bigr) $. The indices $ \pm $ denoting the components can be sub- or superscripts.

$ W \{ u , v \} = u^\prime v - v^\prime u  $ -- the Wronskian.

This being lecture notes, indications of spaces in scalar products and differentials in integrals are suppressed unless we feel they are  necessary. (\ref{byparts}) is representative for our attitude.

\begin{definition}
A canonical system is the differential equation of the form
\be\la{can} J \frac{d Y}{d x } = z \cH Y ; \; z \in \C . \ee
\end{definition}

\begin{remarknonumb}
WLog, one can assume that $ \operatorname{tr} \cH ( x ) = 1 $ a. e. 
\end{remarknonumb}

Indeed, let 
\[ \xi ( x ) = \int_0^x \operatorname{tr} \cH ( t ) \, \d t . \] 
The set $ \{ \xi ( x ) \colon \xi ( x ) = \xi ( x^\prime ) \textrm{ for some } x^\prime \ne x \} $ is countable, hence the equality $ \wt{\cH}( \xi (x) ) = \cH ( x ) $ correctly defines a function $ \wt{\cH} $ a. e.  on the interval $ ( 0, L^\prime ) $, $ L^\prime =  \xi ( L ) $, and the set $ \{ s \in ( 0 , L^\prime )\colon \wt{\cH} ( s ) = 0 \} $ has the Lebesgue measure 0. Let $ Y $ be a solution to (\ref{can}), and let $ \wt Y ( \xi ( x ) ) = Y ( x ) $. This definition is correct, and $\frac 1{ \operatorname{tr} \wt{H} ( t ) } \wt{\cH} ( t ) \wt Y ( t ) $ is a locally bounded function. Then 
\bequnan \int_0^s \frac 1{ \operatorname{tr} \wt{H} ( t ) } \wt{\cH} ( t ) \wt Y ( t ) \d t = \int_0^{ x_s }  \wt{\cH} ( \xi ( x ) ) \wt Y ( \xi ( x ) ) \d x =  \int_0^{x_s}  \cH (  x ) Y ( x ) \d x = \\ \frac 1z \int_0^{x_s}  J Y^\prime ( x ) \d x = \frac 1z ( J Y ( x_s ) - J Y ( 0 ) )= \frac 1z (J \wt Y ( s ) - J \wt Y ( 0 )). \eequnan 
Here $ x_s $ is defined by $ \xi ( x_s ) =  s $. Thus $ \wt Y $ is a. c. and differentiating the equality one obtains

\begin{proposition} If $ Y $ is a solution to (\ref{can})  then $ \wt Y $ is a solution to the system \bequnan \la{normcan} J {\wt{Y}}^\prime &  = & z \cH^\circ \wt Y  ,\\
H^\circ ( s ) \colon & = & \frac 1{ \operatorname{tr} \wt{H} ( s ) } \wt{\cH} ( s ) , s \in ( 0 , L^\prime ) \eequnan and vice versa - if $ U $ is a solution to $ J U^\prime = z \cH^\circ U\ $, then $ Y ( x ) = U ( \xi ( x ) ) $ is a solution to (\ref{can}).
\end{proposition}

The parameter $ z $ in (\ref{can}) is referred to as the \textit{spectral} parameter. The following Harnack-type estimate will be used. Let $ \operatorname{tr} \cH ( x ) = 1 $ a. e.  and let $ M ( x , z ) $ be a matrix solution to (\ref{can}). Then 
\be\la{estM} \left\| M ( x , z ) \right\| \le e^ { x |z| } \len M ( 0 , z ) \rin , \ee the norm being the one of operators in $ \C^2 $. The solution with $ M ( 0, z ) = I $ is called \textit{fundamental}.  When $ L < \infty $ the fundamental solution at $ x = L $ is called the \textit{monodromy matrix}. The determinant of a matrix solution to (\ref{can}) does not depend on $ x $, in particular, $ \det M ( x , z ) \equiv 1 $ for the fundamental solution.

The fundamental solution is real entire in $ z$ for all finite $ x \le L $ and the corresponding estimates hold locally uniformly in $ x$, that is, the fundamental solution is entire as a function with values in the Banach space $ C[ 0 , x ] $.    

If $ 0 < a < b $ then the fundamental solutions satisfy $ M ( b , z ) = N ( b , z ) M ( a , z ) $ where $ N ( x , z ) $ is a solution of (\ref{can}) on $ ( a, b ) $  with $ N ( a , z  ) = I $ (the chain rule).

The following fact is often used throughout,

\medskip

\textbf{Krein's theorem.} \textit{If restrictions of an entire function to halfplanes $ \C_\pm $ are of bounded type, then this function has finite exponential type.}

\section{Examples}\la{examples}

1. Dirac equation.

Let $ Q ( x ) $ be a locally summable $ 2\times 2 $ - matrix function on $ ( 0, L ) $. Consider the equation 
\[  J X^\prime + QX = \lambda X, \;\; \lambda \in \C . \] 
Let the matrix function $ X^\circ $ be the solution of this equation for $ \lambda = 0 $ satisfying $ \det X^\circ = 1 $. The substitution $ X =  X^\circ Y $ gives \[ J X^\circ Y^\prime = \lambda X^\circ Y  \] or, on multiplying by $ \( X^\circ \)^T $,  \[ J Y^\prime = \lambda \( X^\circ\)^T X^\circ Y  \] 
on account of the fact that $ A^T J A = J $ for any matrix $ A$, $ \det A = 1 $. We have obtained a canonical system with the Hamiltonian $ \cH =  \( X^\circ\)^T X^\circ $. Notice that this Hamiltonian has rank 2 at each $ x $.

2. Schr\"odinger equation.

Let $ q $ be a locally summable function on $  ( 0 , L ) $. Consider the equation \[ - y^{\prime\prime} + qy = \lambda y . \] Let $ \yOone $, $ \yOtwo $ be solutions of this equation  for $ \lambda = 0 $ satisfying $ W \{ \yOone , \yOtwo \} = -1 $. Let $ y = y_1^{} \lower1pt\hbox{$\yOone $} + y_2^{} %  ^{} used for alignment of subscripts 
\lower1pt\hbox{$\yOtwo $} $. On substituting this into the equation and letting $ y_1^\prime \lower1pt\hbox{$\yOone $} + y_2^\prime \lower1pt\hbox{$\yOtwo $} = 0 $, one gets 
\[ J \( \! \begin{array}{c}  y_1^\prime \cr y_2^\prime \end{array} \! \) = \lambda  \left( \! \begin{array}{cc} \yOonequ & \yOone \yOtwo \cr  \yOone \yOtwo & \yOtwoqu \end{array} \! \right) \( \! \begin{array}{c}  y_1 \cr y_2 \end{array} \! \) , \] which is a canonical system with Hamiltonian $ \cH ( x ) =  \left( \! \begin{array}{cc} \yOonequ & \yOone \yOtwo \cr  \yOone \yOtwo & \yOtwoqu \end{array} \! \right) $. Notice that this Hamiltonian has rank 1 at each $ x $.

3. String equation.

Let $ \rho ( x ) $, $ \rho ( x ) > 0 $, be a locally summable scalar function. Consider the equation \[ - y^{\prime\prime} = \lambda \rho y . \]  Define $ Y = \left( \begin{array}{c}  \sqrt \lambda y \cr  y^\prime \end{array} \right) $. Then \[ J Y^\prime = \sqrt \lambda \left( \begin{array}{cc} \rho & 0 \cr  0 & 1 \end{array} \right) Y. \]

\section{Operator}

The differential expression corresponding to a canonical system, \[ \cH^{ -1 } ( x ) J \frac \d{\d x}, \] is formally symmetric in the space $ L^2 ( ( 0, L) ; \cH ) $. How to define a selfadjoint operator, $ D $, corresponding to it? $ \cH ( x ) $ may be non-invertible everywhere - see the Schr\"odinger operator above, hence the question - what is the space where $ D $ acts? To answer it, suppose that there exists an $ I \subset ( 0, L ) $ and an $ e \in {\R}^2 $ such that $ \cH ( x ) = \langle \cdot , e \rangle e $ for a. e. $ x \in I $. Consider the equation $ J f^\prime = \cH g $ on the interval $ I $.  On projecting both sides on $ e^\perp \colon = Je $ and taking into account that $ J e^\perp = - e $ we get that $ \langle f^\prime , e \rangle = 0 $. Thus $ \langle f (x) , e \rangle $ is a constant on $ I $.  This means that no function of the form $ \xi ( x ) e $ with $ \xi ( x ) \in L^2 ( I ) $ vanishing outside $ I $ and such that $ \int_I \xi = 0 $ can be approximated in the weighted $ L^2 $-norm by elements from the domain of $ D $. Hence we have to remove such elements from $ L^2 ( ( 0, L) ; \cH ) $ to obtain the Hilbert space for $ D $.

\begin{definition}
An open interval (possibly, semiaxis) $ I \subset \R_+ $ is called $ \cH  $-indivisible (or singular) if there exists an $ e \in {\R}^2 $ such that 
\be\la{singint} \cH ( x ) = \langle \cdot , e \rangle e \end{equation} 
for a. e. $ x \in I $ and there is no larger open interval $ I^\prime $ such that (\ref{singint}) holds for a. e. $ x \in I^\prime $.
\end{definition}  

\begin{definition}
\[ H = \{ f \in L^2 ( 0, L ; \cH ) \colon f(x) \simeq \textrm{const}_I \textrm{ for any } \cH-\textrm{indivisible interval } I \} . \]
\end{definition}

$ \simeq $ here is, hopefully, a self-explaining notation. 

Consider first the case $ L < \infty $. Let us define a natural domain of the operator corresponding to a selfadjoint boundary condition, say $ f_- ( 0 ) = f_- ( L ) = 0 $. 

\begin{deffinition}\label{defdom} Let 
\[ \mathcal D = \left\{ f \in H \colon \begin{array}{cc}  (\textrm{i}) & f \textrm{ is absolutely continuous on } (0, L) \cr (\textrm{ii}) & \exists g \in H \colon J f^\prime = \cH g \cr (\textrm{iii}) & f_- ( 0 ) = 0 \cr (\textrm{iv}) &  f_- ( L ) = 0 \end{array} \right\} . \]
\end{deffinition}

\begin{remarknonumb} (i) means that the equivalence class of an element $ f \in \mathcal D $ contains an a. c. representative, (iii), (iv) that this representative satisfies the stated identities. \end{remarknonumb}

In general the set $ \mathcal D $ is not dense in $ H $. Indeed, suppose that $ ( 0, \von ) $ is a singular interval for some $ \von > 0 $, and the Hamiltonian $ \cH ( x ) = \left( \begin{array}{cc}  0 & 0 \cr 0 & 1 \end{array} \right) $ for a. e. $ x \in ( 0, \von ) $. Let $ f \in \mathcal D $. Then $ f_- ( x ) = \langle f ( x ) , ( 0, 1 )^T \rangle $ is constant on $ ( 0, \von ) $ (see the discussion at the beginning of the section), hence zero by (iii). Thus, the element $ h \in H $ given by \[ h = \left\{ \begin{array}{cc} \Bigl( \! \begin{array}{c}  0 \cr 1 \end{array} \!\Bigr) , & x < \von, \cr 0, & x > \von \end{array} \right. \] is orthogonal to $ \mathcal D $. According to the following theorem, this is the only obstacle to density of $ \mathcal D $ in $ \cH $. 

Let $ e = \Bigl( \! \begin{array}{c}  0 \cr 1 \end{array} \! \Bigr) $. The element of $ H $ defined by equivalence class of the function $ u ( x) \equiv J e $ is  denoted by $ e^\perp $. 

\begin{theorem}\label{domain}
Let 
\[  \varphi_1 =  \left\{ \begin{array}{cc} e , & x < \von \cr 0, & x > \von \end{array} \right. , \;  \varphi_2 =  \left\{ \begin{array}{cc} e , & L - \von^\prime < x < L \cr 0, & x < L - \von^\prime \end{array} \right. . \]
Then $ {\mathcal D}^\perp $ is spanned by those of $ \varphi_1 , \varphi_2 $ which belong to $ H $ for some $ \von, \von^\prime > 0 $. \end{theorem}

\begin{proof} As we have seen above, $ \varphi_1 \perp {\mathcal D} $ whenever $ \varphi_1 \in H $, and the same is true of $ \varphi_2 $. Let $ \varphi \in {\mathcal D}^\perp $, and let $ g \in H $ be such that $ \llangle g , e^\perp \rrangle_H = 0 $. Define $ f ( x )  = J \int_0^x \cH g $. Then $ f \in \mathcal D $ and thus
\[ 0 = \llangle \varphi , f \rrangle_H = \int_0^L \llangle \cH \varphi , J \int_0^x \cH g \rrangle \d x = - \int_0^L \llangle J \int_t^L\cH \varphi , \cH g  \rrangle \d t  . \]
Notice that $ \Phi ( t ) = J \int_t^L\cH \varphi $ belongs to $ H $, hence the last equality says that $ \llangle \Phi , g \rrangle_H = 0 $ for all $ g \in H $ orthogonal to $ e^\perp $. It follows that $ \Phi ( t ) $ belongs to the equivalence class of a multiple of $ e^\perp $, $ \Phi ( t ) = c Je + \Omega ( t ) $, $ \Omega ( t ) \in \ker \cH ( t ) $ for a. e. $ t $, $ c \in \C $. The $ \Omega ( t ) $ is an a. c. function since $ \Phi ( t ) $ is. Let $ \cM = \{ t\in ( 0, L ) \colon \Omega ( t ) \ne 0 \} $. There exists a function $ \xi \colon \cM \to \R^2 $, $ \| \xi ( x ) \| = 1 $, such that $ \cH ( x) = \langle \cdot , \xi ( x ) \rangle  \xi( x ) $ for a.e. $ x \in \cM $. Moreover, the function $ \xi ( x ) $ can be chosen so that for $ x \in \cM $ we have, $ \Omega ( x ) =  \omega ( x )  J \xi ( x ) $, $ \omega ( x ) = \| \Omega ( x ) \| $. The functions $ \omega ( x ) $, $ \xi ( x ) $ are a. c. on $ \cM $, since $ \Omega ( x ) $ does not vanish there. On differentiating, we get for a. e. $ x \in \cM $,
\[ \omega^\prime ( x ) J \xi ( x ) + \omega ( x ) J \xi^\prime ( x ) = - \langle \varphi ( x ) , \xi ( x ) \rangle J \xi ( x ) . \] Canceling $ J $ out and taking into account that $ \xi^\prime ( x ) $ is orthogonal to $ \xi ( x ) $ we find $ \omega  ( x ) \xi^\prime ( x ) = 0 $ from whence $ \xi ( x ) $ is constant on intervals of the set $ \cM $, that is, each interval of $ \cM $ is a sub-interval of a singular one. Since $ \varphi \in H $ this implies that $ \omega^\prime $ is constant on the intervals of $ \cM $. Thus, for each interval, $ I $, of the set $ \cM $ there exist an $ e_I \in \R^2 $ such that $ \Omega ( x ) = (\textrm{linear function of } x) J e_I $, $ x \in I $. Let $ I \subset ( 0 , L ) $ be an interval of the set $ \cM $ such that neither $ 0 $ nor $ L $ are its ends. Since $ \Omega ( x ) =0 $ at the ends of $ I $, we infer that $ \Omega ( x ) \equiv 0 $ on $ I $, a contradiction, hence $ \cM $ consists of at most two intervals, $ I_0 = ( 0 , \von ) $ and $ I_L = ( L - \von^\prime , L ) $. It follows that \[ \cH ( x ) \varphi ( x ) =  \left\{ \begin{array}{cc} {\textrm{const}}_0 e_{ I_0 } , & x \in I_0 , \cr    {\textrm{const}}_L e_{ I_L } , & x \in I_L , \cr 0, & \textrm{otherwise.} \end{array} \right. \] It remains to notice that $ \Phi ( L ) = 0 $, hence $ \Omega ( L )  = -Je $, that is, $ e_{I_L} =  \pm e $. A similar consideration starting with $ f ( x ) = J \int_x^L \cH g $ shows that $ e_{I_0} = \pm e $.
\end{proof}

\begin{corollary}
$ \overline{\mathcal D} = H $ if the following two conditions are satisfied,  

(L) there is no $ \von > 0 $ such that $ \cH ( x ) = \langle \cdot , e  \rangle e $, $ e = ( 0, 1 )^T $, for a.e. $ x \in ( 0, \von ) $;

(R) there is no $ \von > 0 $ such that $ \cH ( x ) = \langle \cdot , e  \rangle e $, $ e = ( 0, 1 )^T $, for a.e. $ x \in ( L - \von, L ) $. 
\end{corollary}

In what follows (L) and (R) are referred to as compatibility conditions at left and right endpoint, respectively. 

\begin{theorem} \label{operator} Let (L) and (R) be satisfied. Then the mapping \[ D \colon  f \mapsto g \] on the domain $ \mathcal D $ is a correctly defined selfadjoint operator in $ H $, the notation being that of definition \ref{defdom}. \end{theorem}

\begin{proof} \textit{Correctness.} Let $ f_1 $, $ f_2 $ be two a. c. representatives of the same element $ f \in \mathcal D $ such that $ J f^\prime_1 = \cH g_1 $, $ J f^\prime_2 = \cH g_2 $ for some $ g_1, g_2 \in H $, and $ f^{1,2}_- ( 0 ) = f^{1,2}_- ( L ) = 0 $. We have to show that in this case $ g_1 = g_2 $. Let $ h = f_1 - f_2 $ and $ \cM  = \{ t \in ( 0 , L ) \colon h ( t ) \ne 0 \} $. For a. e. $ x \in \cM $ the rank of $ \cH ( x ) $ equals 1, and $ h ( x ) \in \ker \cH ( x ) $, hence one can define $ e ( x) \in \textrm{Ran} \, \cH ( x ) $ so that  $ h ( x ) = \rho ( x ) J e ( x ) $ where $ \rho ( x ) = \| h ( x ) \| $. The function $ \rho ( x ) $ is a. c. and does not vanish on $ \cM $, hence $ e ( x) $ is also a. c. on $ \cM $, and \[ h^\prime ( x ) = \rho^\prime ( x ) J e ( x ) + \rho ( x ) J e^\prime ( x ) . \] Projecting on $ e ( x ) $ and taking into account that the l. h. s. vanishes because $ J h^\prime \in \textrm{Ran } \! \cH $, we find $ \langle J e^\prime ( x ) , e ( x ) \rangle = 0 $ a. e. on $ \cM $. Since $ \| e ( x ) \| = 1 $ this implies that $ e^\prime ( x ) $ vanishes a.e. on $ \cM $. It follows that each interval of $ \cM $ is a sub-interval of a singular one. Reasoning as in the proof of theorem \ref{domain} we infer that $ \cM $ consists of at most two intervals, $ I_0 = ( 0 , \von ) $ and $ I_L = ( L - \von^\prime , L ) $, and $ h ( x ) = c ( x - \von ) J e_0 $, $ e_0 \in \textrm{Ran } \cH ( x ) $, for $ x \in I_0 $, $ h ( x ) = d ( L - x + \von^\prime ) J e_L $, $ e_L \in \textrm{Ran } \cH ( x ) $, for $ x \in I_L $. On account of the boundary conditions, $ h_- ( 0 ) = 0 $ which means that either $ ( J e_0 )_-  = 0 $, which contradicts (L), or $ c = 0 $. The second interval is treated similarly via condition (R), and we find that $ h = 0 $. 

\textit{Selfadjointness}.
The conditions (L) and (R) ensure that the domain of $ D $ is dense, hence the adjoint operator $ D^* $ exists. Let us find its domain, $ {\mathcal D}^* $. Given an $ f \in \mathcal D $, pick up a representative of it, satisfying $ f_- (0) = 0 $. We then have, 
\begin{eqnarray} \langle D f , f \rangle = \int_0^L \langle \cH D f , f \rangle = \int_0^L \langle J f^\prime , f \rangle = 2i \left. \Im ( f_+ \overline{f_-} ) \right|_0^L - \int_0^L \langle J f , f^\prime \rangle = \nonumber \\ \int_0^L \langle f , J f^\prime \rangle = \langle f , D f \rangle . \la{byparts} \end{eqnarray}
This means that $ D \subset D^* $. Let $ f \in {\mathcal D}^* $, $ g = D^* f $. Define $ \tf $ to be the solution to $ J {\tf}^\prime = \cH g $ of the form $ \tf = -J \int_0^x \cH g $. It is easy to see that $ \tf \in \mathcal D $ since $ e^\perp \in \mathcal D $ and thus $ \langle g , e^\perp \rangle_{ \cH }  = \langle f ,  D e^\perp \rangle_{ \cH } = 0 $ which implies $ \tf_- ( L ) = 0 $. For any $ u \in \mathcal D $ picking up a representative of its equivalence class from the definition of $ \mathcal D $ we have 
\be\la{anotherbyparts} \langle D u , f \rangle = \langle u , g \rangle = \int_0^L \langle u ,J{\tf}^\prime \rangle = \left. \langle u , J\tf \rangle  \right|_0^L + \int_0^L \langle J u^\prime , \tf \rangle = \langle D u , \tf \rangle . \ee 
Choose $ u \in \mathcal D $ in the form $ u ( x )  = - J \int_0^x \cH h $, where $ h \in H $ satisfies $ \llangle h , e^\perp \rrangle_H = 0 $ and is otherwise arbitrary. Then $ h = D u $, and (\ref{anotherbyparts}) shows that $ \tf - f \perp h $.   We obtain $ \tf - f = (\textrm{const}) \, e^\perp $, that is, $ \tf - (\textrm{const}) \, e^\perp $ is a representative of the equivalence class of $ f $. This representative by construction satisfies all the conditions in the definition of $ \mathcal D $. Thus, $ f \in \mathcal D $ and $ D = D^* $. \end{proof}

\begin{remarknonumb} It is possible that the conditions of the theorem are satisfied but there are two a. c. representatives of the same $ f \in \mathcal D $, one of them satisfies (iii), (iv) and the other does not. \end{remarknonumb}

Consider now the case $ L = \infty $. WLog, one can assume that $ ( b , \infty ) $ is not a singular interval for all $ b > 0 $, for if it is, then the corresponding space $ H $ is the same as the space of the canonical system on the finite interval $ ( 0 , b ) $ with the same Hamiltonian. We define the domain $ \mathcal D $ to be the set of  $ f \in H $ for which conditions (\textit{i}) - (\textit{iii}) of definition \ref{defdom} are satisfied.  

\begin{theorem}\label{operatorLinf}
$ {\mathcal D}^\perp $ is spanned by the vector $ \left\{ \begin{array}{cc} e , & x < \von \cr 0, & x > \von \end{array} \right. $, $ e = ( 0, 1 )^T $, whenever this vector belongs to $ H $ for some $ \von > 0 $, and is zero otherwise. If the condition (L) is obeyed then the mapping \[ D \colon  f \mapsto g \] on the domain $ \mathcal D $ is a correctly defined selfadjoint operator in $ H $. \end{theorem}

Throughout the proof, a number $ N > 0 $ is referred to as admissible if it is not an interior point of a singular interval. For an admissible $ N $ the functions supported on $ ( 0, N ) $ form a subspace in $ H $ denoted by $ H_N $. Recall also the notation $ e^\perp = J e $.

\begin{proof} \textit{Domain}. Let $ \varphi \in {\mathcal D}^\perp $, $ N \in ( 0 , \infty ) $ be an arbitrary point which is not interior to a singular interval, and let $ g \in H $ be supported on $ ( 0 , N ) $ and such that $ \llangle g , e^\perp \rrangle_H = 0 $. Define $ f ( x )  = 0 $ for $ x  > N $, and $ f(x) = J \int_x^N \cH g $ otherwise. Then $ f \in \mathcal D $. We have
\[ 0 = \llangle \varphi , f \rrangle_{ \cH } = \int_0^\infty \llangle \cH \varphi , J \int_x^N \cH g \rrangle \d x = - \int_0^N \llangle J \int_0^t \cH \varphi , \cH g  \rrangle \d t  . \]
Notice that $ \Phi ( t ) = J \int_0^t \cH \varphi $ belongs to $ H_A $ for any admissible $ A$, and hence the last equality says that $ \llangle \Phi , g \rrangle_{ H_N } = 0 $. It follows that $ \Phi ( t ) $ belongs to the equivalence class of a constant vector, a multiple of $ e^\perp $, in $ H_N $, $ \Phi ( t ) = c e^\perp + \Omega ( t ) $, $ \Omega ( t ) \in \ker \cH ( t ) $ for a. e. $ t \in ( 0 , N ) $, $  c\in \mathbb{C} $. It is obvious that the constant $ c $ does not depend on $ N $. Reproducing verbatim the argument from the proof of theorem \ref{domain} with $ N $ in the place of $ L $ we find that the set $ \cM = \{ t\in ( 0, N ) \colon \Omega ( t ) \ne 0 \} $ consists of at most two intervals, $ I_0 = ( 0 , \von ) $ and $ I_N = ( N - \von^\prime , N ) $, $ \von , \von^\prime \ge 0 $. Since the set of admissible $ N $ accumulates at infinity, we conclude that no interval $ I_N $ is in $ \mathcal M $, hence \[ \cH ( x ) \varphi ( x ) =  \textrm{const} \left\{ \begin{array}{cc}  e_0, & x \in I_0 , \cr 0, & \textrm{otherwise.} \end{array} \right. \] for some $ e_0 \in \R^2 $. From $ \Phi ( 0 ) = 0 $ we infer that $ \Omega ( 0 ) = - c e^\perp $, that is, $ e_0 = \pm e $.

The correctness of the definition of $ D $ under condition (L) is proved in the same way as in the case $ L < \infty $. The argument is shortened by absence of the right end of the interval which means that no extra condition is required.

\textit{Symmetricity}.
The condition (L) ensures that the domain of $ D $ is dense. Let us establish that $ D $ is symmetric. Given an $ f \in \mathcal D $, pick up a representative of it, satisfying $ f_- (0) = 0 $. Then, for any $ L \in ( 0, \infty ) $ we have \[ \int_0^L \langle J f^\prime , f \rangle = 2i \left. \Im ( f_+ \overline{f_-} ) \right|_0^L - \int_0^L \langle J f , f^\prime \rangle =  \langle J f ( L )  , f ( L ) \rangle  + \int_0^L \langle f , J f^\prime \rangle . \] As $ L \to \infty $ the l. h. s. tends to  $ \langle D f , f \rangle $, the second term in the rhs - to $ \langle f , D f \rangle $. This means that $ D $ is symmetric if $ \langle J f (L ), f ( L )  \rangle \to 0 $ as $ L \to \infty $. Indeed, let \[ a = \lim_{L \to \infty } \langle J f (L ), f ( L )  \rangle . \] In the basis of eigenvectors of $ \cH ( L ) $, we have \[ 
\langle J f (L ), f ( L )  \rangle = 2i \, \Im \( f_1 (L ) \overline{f_2 ( L )}  \) , \] $ f_1 ( x ) $ ($ f_2  ( x ) $) being the projections of $ f ( x ) $ on normalized eigenvectors corresponding to the maximal (minimal) eigenvalue of $ \cH ( x ) $, respectively. Here we took into account that $ U^* J U = J $ for any real orthogonal matrix $ U $. For any $ x $, $ y $ large enough, $ x > y $, one writes,
\[ \left| \frac a2 ( x - y ) \right| \le \left| \int_y^x \Im \( f_1 \overline{f_2 } \) \right| \le \( \int_y^\infty \left| f_1 \right|^2 \)^{1/2 } \( \int_y^x \left| f_2 \right|^2 \)^{1/2 } \le \dots \]
Notice that $ f \in H $ means that $ f_1 \in L^2 ( 0, \infty ) $ - the normalization of the Hamiltonian implies $ \int_0^\infty | f_1 |^2 \le 2 \left\| f \right\|^2_H $, and $ f \in \mathcal D $ means that $ f^\prime \in L^2 ( \R ; {\C}^2 ) $, hence $ \left\| f ( s )  \right\|^2 = O ( s ) $. One can proceed estimating,
\[  \dots \le o ( y ) \( x^2 - y^2 \)^{ 1/2 } . \] Taking $ y $ large enough and considering the limit $ x \to \infty $, we get a contradiction, unless $ a = 0 $. The symmetricity of $ D $ is proved. 

\textit{Selfdjointness}. Let $ {\mathcal D}^* $ be the domain of $ D^* $, $ f \in {\mathcal D}^* $, $ g = D^* f $. Define $ \tf $ to be the solution to $ J {\tf}^\prime = \cH g $ of the form $ \tf = -J \int_0^x \cH g $. For any compactly supported $ u \in \mathcal D $ picking up a representative of its equivalence class from the definition of $ \mathcal D $ we have for any $ L $ to the right of the support of $ u $, \[ \langle D u , f \rangle = \langle u , g \rangle = \int_0^L \langle u ,J{\tf}^\prime \rangle = \left. \langle u , J\tf \rangle  \right|_0^L + \int_0^L \langle J u^\prime , \tf \rangle = \llangle D u , \tf \rrangle_{ H_L } . \]  Choose $ u $ in the form \[ u ( x ) = \left\{ \begin{array}{cc} - J \int_x^L \cH h , & x < L, \cr 0, & \textrm{otherwise} \end{array} \right. \] where $ L $ is not an interior point of a singular interval, and $ h $ is an arbitrary element of $ H $ supported on $ ( 0, L ) $ and such that $ \llangle h , e^\perp \rrangle_{ H_L } = 0 $. Then, $ D u = h $ and the calculation above says $ \tf - f $ is orthogonal to $ h $ in $ H_L $. It follows that the restriction of $ \tf - f $ on $ ( 0 , L ) $ is an element of $ H_L $ which belongs to the equivalence class of a constant scalar multiple of $ e^\perp $. It is obvious that this constant, say $ b $, does not depend on $ L $. Since $ L $ is arbitrary, we find that the class of equivalence of $ f $ in $ H $ has an a. c. representative, $ \hat f = \tf - b e^\perp $. This representative satisfies $ J {\hat f}^\prime = \cH g $, and $ {\hat f}_- ( 0 ) = 0 $ by the choice of $ h $. These mean that $ f \in \mathcal D $, that is, $ D = D^* $. 
\end{proof}

\begin{remark}\la{essentialself} In the last portion of the proof we have actually shown that the operator $ D $ is essentially selfadjoint on the linear set $ \wt D $ of compactly supported vectors $ f \in \cD $ by establishing that $ \( \left. D \right|_{ \wt D } \)^* = D $. 
\end{remark}

\medskip
\textbf{Examples.} 4. Jacobi matrices. Let $ b_j $, $ j \ge 0 $, $ b_0 = 0 $, be a monotone increasing sequence, $ \Delta_1 = ( 0 , b_1 ) $, $ \Delta_2 = ( b_1 , b_2 ) $ etc., $ e_j \in {\mathbb{R}}^2 $, $ j \ge 1 $, a sequence of vectors of unit norm, $ e_j \ne \pm e_{j-1} $. Define 
\be\la{HamJacobi} \cH ( x ) = \langle \cdot , e_j \rangle e_j , \; \; x \in \Delta_ j . \ee
Compatibility condition (R) is trivially satisfied for this Hamiltonian. Let us calculate the operator $ D $ assuming that the condition (L) is also satisfied. The space $ H $ is naturally identified with $ l^2 ( \N ; \nu ) $, $ \nu = \{ l_j \} $, $ l_j = |\Delta_j | $. For $ g = D f $, $ f \in \mathcal D $, let $ \{ g_j \} $ be the sequence in $ l^2 ( \N ; \nu ) $ identified with $ g $, $ e_j^\perp = J e_j $, \[ f ( x ) = f_j e_j + \xi_j ( x ) e_j^\perp , \; 
\; x \in \Delta_j . \]
Our goal is to express $ g_j $'s via $ f_j $'s. The equation $ J f^\prime = \cH g $ is written as $ \xi_j^\prime ( x ) = - g_j $, $ x \in \Delta_j $. This gives \[ \xi_j = s_j + g_j ( b_j - x ) , \; \; x \in \Delta_j . \] The continuity of $ f $ at $ x = b_{ j-1} $ means that for $ j \ge 2 $
\be\label{continuity} f_{ j-1 } e_{ j-1 } + s_{ j-1 } e_{ j-1 }^\perp = f_j e_j + ( s_j + g_j l_j ) e_j^\perp . \ee
On taking scalar product with $ e_j $, we obtain
\[ f_{ j-1 } \langle e_{ j-1 } , e_j \rangle + s_{ j-1 } \langle e_{ j-1 }^\perp , e_j \rangle = f_j , \] from whence
\[ s_j = \frac 1{\langle e_j^\perp , e_{ j+1 } \rangle } \left[ f_{ j+1 } - f_j \langle e_j , e_{j+1} \rangle \right],  \;\;  j \ge 1 . \] 
Plugging this back to (\ref{continuity}) and taking scalar product with $ e_{ j-1 } $ gives
\[ f_{ j-1 } = f_j \langle e_j , e_{ j-1 } \rangle +  \left[ \frac 1{\langle e_j^\perp , e_{ j+1 } \rangle } ( f_{ j+1 } - f_j \langle e_j , e_{j+1} \rangle ) + g_j l_j \right] \langle e_j^\perp , e_{ j-1 } \rangle \] or 
\[ l_j g_j = \frac 1{\langle e_j^\perp , e_{ j-1 } \rangle } f_{ j-1 } + \( \frac{ \langle e_j , e_{ j+1 } \rangle }{ \langle e_j^\perp , e_{ j+1 } \rangle } - \frac{ \langle e_j , e_{ j-1 } \rangle }{ \langle e_j^\perp , e_{ j-1 } \rangle } \) f_j - \frac 1{\langle e_j^\perp , e_{ j+1 } \rangle } f_{ j+1 } . \]
Let $ u_j = g_j \sqrt{ l_j } $, $ v_j = f_j \sqrt{ l_j } $. Then 
\[ u_j = \frac 1{\langle e_j^\perp , e_{ j-1 } \rangle \sqrt{ l_{ j-1 } l_j }} v_{ j-1 } + \frac 1{ l_j } \( \frac{ \langle e_j , e_{ j+1 } \rangle }{ \langle e_j^\perp , e_{ j+1 } \rangle } - \frac{ \langle e_j , e_{ j-1 } \rangle }{ \langle e_j^\perp , e_{ j-1 } \rangle } \) v_j - \frac 1{\langle e_j^\perp , e_{ j+1 } \rangle \sqrt{ l_{ j+1 } l_j }} v_{ j+1 } . \]
Notice that 
\[ \langle e_j^\perp , e_{ j-1 } \rangle = \langle J e_j , e_{ j-1 } \rangle = - \langle e_j, e_{ j-1 }^\perp \rangle = - \langle e_{ j-1 }^\perp , e_j \rangle  . \]
It follows that if we define 
\begin{eqnarray} \la{rho} \rho_j & = & - \frac 1{\langle e_j^\perp , e_{ j+1 } \rangle \sqrt{ l_{ j+1 } l_j }}, \; j \ge 1 , \\ \la{q} q_j & = & \frac 1{ l_j } \( \frac{ \langle e_j , e_{ j+1 } \rangle }{ \langle e_j^\perp , e_{ j+1 } \rangle } - \frac{ \langle e_j , e_{ j-1 } \rangle }{ \langle e_j^\perp , e_{ j-1 } \rangle } \) , \; j \ge 2 , \end{eqnarray} 
then 
\[ u_j = \rho_{ j-1 } v_{ j-1 } + q_j v_j + \rho_j v_{ j+1 } , \;  j \ge 2 . \]
For $ j = 1 $ the boundary condition $ f_- ( 0 ) = 0 $ can be interpreted as a condition of the form (\ref{continuity}) with $ e_0 = \( \begin{array}{c} 0 \cr 1 \end{array} \) $, $ f_0 = 0 $.  This gives $ u_1 = q_1 v_1 + \rho_1 v_2 $ with \be\la{q1} q_1 = \frac 1{ l_1 } \( \frac{ \langle e_1 , e_2 \rangle }{ \langle e_1^\perp , e_2 \rangle } - \frac{ e_1^- }{ \ e_1^+ } \) . \ee Notice that $ e_1^+ \ne 0 $ - this  is the condition (L). Let $ U $ stand for the natural isomorphism between $ l^2 ( \mathbb{N}; \nu ) $ and $ l^2 ( \N ) $,  $\( U h \)_j = \sqrt{ l_j } h_j $. The calculations just done show that $ U D U^* $ is a selfadjoint operator corresponding to the following Jacobi matrix 
\[ \( \begin{array}{cccccc} q_1 & \rho_1 & 0 & \dots & & \cr 
\rho_1 & q_2 & \rho_2 & 0 & \dots & \cr
0 & \rho_2 & q_3 & \rho_3 & 0 & \dots \cr  & 0 & \ddots & \ddots & \ddots & \ddots \end{array} \). \]
For those familiar with the Weyl terminology, we notice that this matrix is in the limit-circle case if and only if $ L < \infty $.

\begin{proposition} To any Hamiltonian of the form (\ref{HamJacobi}) with $ e_1^+ \ne 0 $ there corresponds a semi-infinite Jacobi matrix defined by (\ref{rho}), (\ref{q}) and (\ref{q1}). For any symmetric Jacobi matrix with $ \rho_j > 0 $ there exists a Hamiltonian of the form (\ref{HamJacobi}) such that (\ref{rho}), (\ref{q}) and (\ref{q1}) hold and $ e_1^+ \ne 0 $. Given an $ e \in \R^2 $, $ e \ne ( 0, 1 )^T $, $ \| e \| = 1 $, and a $ \Delta > 0 $, this Hamiltonian can be chosen so that $ e_1 = e $, $ \Delta_1 = \Delta $, and this choice is unique. \end{proposition}

That any Hamiltonian of the form (\ref{HamJacobi}) with $ e_1^+ \ne 0 $ correspond to a Jacobi matrix is proved by the calculations above. The reverse assertion is easily established by induction on taking into account that $ \frac{ \langle e_j , e_{ j+1 } \rangle }{ \langle e_j^\perp , e_{ j+1 } \rangle } $ is the cotangent of the angle between $ e_j $ and $ e_{j+1} $, and hence defines the Hamiltonian on the interval $ \Delta_{ j+1 } $ provided that it is known on $ \Delta_j $.

An assertion analogous to theorem \ref{operator} holds for other boundary conditions at the right end of the interval. We formulate the result leaving the details to the reader.

\begin{theorem} \label{operatorbc}  Let $ e = \( \! \begin{array}{c} \cos \alpha \cr \sin \alpha \end{array} \! \) $, $ 0 \le \alpha \le \pi $, $ L < \infty $.  Let the conditions (L) and 
\[ (R^\prime) \;\;\; \textrm{There is no }  \von > 0 \colon \cH ( x ) = \langle \cdot , e  \rangle e \textrm{ for a.e. } x \in ( L - \von, L ) .\] be satisfied. Let $ {\mathcal D}_\alpha $ stand for the linear set described in definition \ref{defdom} with (iv) replaced by $ f_+ ( L ) \cos \alpha + f_- ( L ) \sin \alpha = 0 $. Then $ {\mathcal D}_\alpha $ is dense in $ H $, and the mapping \[ D_\alpha \colon  f \mapsto g \] on the domain $ {\mathcal D}_\alpha $ is a correctly defined selfadjoint operator in $ H $. For a given Hamiltonian $ \cH $ the compatibility condition ($ R^\prime $) fails for at most one $\alpha \in [ 0 , \pi ) $. \end{theorem} 

\section{Direct spectral theory}\la{direct}

\subsection{Synopsis} The de Branges approach to the spectral theory is as follows. Rather than diagonalizing each operator $ D_\alpha $, corresponding to a selfadjoint boundary condition at the right end (see exercise 1 above), we study a Fourier transform with respect to solutions, $ \Theta ( x , \lambda ) $, of the canonical system satisfying a fixed Cauchy problem at $ x = 0 $ for \textit{all} complex values of the spectral parameter, that is, we consider the mapping which takes an $ f $ from the Hilbert space of the system, $ H $, to $ \hat f ( w ) \colon = \llangle f , \Theta ( \cdot , \overline w ) \rrangle_H $. Then $ \hat f $ is an entire function and the mapping, at least formally, takes the symmetric operator defined by the boundary condition $ f ( L ) = 0 $, into (a restriction of) operator of multiplication by the independent variable. For each $ \theta $ the composition of this mapping and the subsequent restriction of the entire function to the spectrum of $ D_\alpha $ defines an isomorphism of $ H $ onto $ L^2 ( \R , d \mu_\alpha ) $, diagonalizing $ D_\alpha $, $ \mu_\alpha $ being a discrete measure on $ \R $. The idea is to define a Hilbert space structure on the set of entire functions $ \hat f $ in such a way that $ f \mapsto \hat  f $ is an isomorphism. It is remarkable that this structure can be described intrinsically ("on the spectral variable side"). Indeed, when a real $ z $ belongs to the spectrum of one of $ D_\alpha $ it is clear that $ \hat{ \Theta}_z $ must be proportional to the point evaluator (reproducing kernel) in the space $ \hat{ H } $, the range of the mapping. Thus $ \hat{ H } $ is a reproducing kernel Hilbert space of entire functions, and the image of $ \Theta_z $ is a multiple of the reproducing kernel at $ z $. These observations allow to calculate the norm in the space $ \hat H $ (this is done in theorem \ref{spmaptheo}). It turns out to be a de Branges  space, $ \HE $, corresponding to the function $ E  = \Theta_+^L + i \Theta_-^L $.  

\subsection{} Define $ \Theta ( x , \lambda ) $, $ \Phi ( x , \lambda ) $ to be solutions of the system (\ref{can}) satisfying $ Y ( 0 ) = \( 1 , 0 \)^T $, $ Y ( 0 ) = \( 0 , 1 \)^T $, respectively, $ M ( x , \lambda ) $ to be the fundamental solution of the canonical system, $ M ( x , \lambda ) = [ \Theta ( x , \lambda ) , \Phi ( x , \lambda )  ] $. Depending on context, either argument, or both, of these functions may be skipped in our notation or moved to subscripts.

The following important identity is used throughout. For all $ \lambda , z \in \C $, and all finite $ x \in ( 0 , L ] $
\be\la{identM} M^* ( x , \lambda ) J M ( x , z ) - J = \( z - \overline \lambda \) \int_0^x M^* ( t , \lambda ) \cH ( t ) M ( t , z ) \d t . \ee

The proof is by integrating the derivative of $ M^* J M $ with respect to the equation $ J M^\prime = z \cH M $ and taking into account that $ M ( 0 , \lambda ) = I $. It is convenient for us to write down separately the upper leftmost entry in this identity,
\be\la{identM11} 
\Theta_+ ( x , z ) \Theta_- \( x , \overline \lambda \) - \Theta_- (x , z ) \Theta_+ \( x , \overline \lambda \) = \( z - \overline \lambda \) \int_0^x \Theta^T\( t , \overline \lambda \) \cH ( t ) \Theta ( t , z ) \d t . \ee

\begin{proposition}\la{HB}
Given an $ x \in ( 0, L ) $, the function $ E_x ( \lambda ) = \Theta_+ ( x , \lambda ) + i \Theta_- ( x , \lambda ) $ is Hermite-Biehler, unless 
\be\la{nondegenerate} \cH ( t ) = \( \! \begin{array}{cc} 0 & 0 \cr 0 & 1 \end{array} \! \) \;\; \textrm{for a. e. } \, t \in ( 0,x ) . \ee
\end{proposition}

\begin{proof}
Plugging $ \lambda = z $ in (\ref{identM11}) gives
\be\la{11} \Im ( \Theta_+ {\overline\Theta}_- ) = \Im z \, \int_0^x \llangle \cH \Theta , \Theta \rrangle_{ \C^2 } . \ee 
 On  the other hand, 
\[ \left| E_x ( z ) \right|^2 -\left| E_x ( \overline z ) \right|^2 = 4 \, \Im ( \Theta_+ ( x, z )\overline{\Theta_- ( x, z) }) \] 
for all $ z \in \C $ by straightforward computation. 
The integrand in the r. h. s. of (\ref{11}) is nonnegative, hence $ \Im \(  \Theta_+ ( z ) \overline{ \Theta_- ( z ) } \) \ge 0 $ for all $ z \in \C_+ $, and the equality is only achieved if $ \Theta ( t , z ) \in \ker \cH ( t ) $ for a. e. $ t \in  ( 0 , x ) $, in which case $ \Theta^\prime = 0 $ a. e. on $ ( 0, x ) $, hence $ \Theta ( t ) \equiv \( \begin{array}{c} 1 \cr 0 \end{array} \) $ and $ \cH ( t ) $ satisfies (\ref{nondegenerate}). The assertion follows. \end{proof}

Let $ L < \infty $. The function $ \Theta_z $ then belongs to the space $ H $ for all $ z \in \C $. From now on, we assume that the Hamiltonian does not satisfy (\ref{nondegenerate}) for $ x = L $, hence the de Branges space $ \HE $, $ E = E_L $, is defined. The function $ E_L $ will be then called the \textit{de Branges function} of the system.

\begin{theorem}\la{spmaptheo}
The reproducing kernel, $ K_\lambda ( z ) $, of the space $ \HE $ is given by 
\be\la{repker} K_\lambda ( z ) = \frac 1\pi \llangle \Theta_z , \Theta_\lambda \rrangle_H . \ee 
Let the compatibility condition (L) be satisfied. Then the mapping 
\be\la{spmap} \cU \colon f \mapsto \pi^{ -1/2 } \llangle f , \Theta_{ \overline w } \rrangle_H \ee
is an isomorphism of $ H $ onto $ \HE $.
\end{theorem}

\begin{proof} On plugging $ E = \Theta_+ + i \Theta_- $ into the formula for the reproducing kernel of an abstract de Branges space, 
\[ K_\lambda ( z ) = \frac 1{2 \pi i } \frac{ E ( z ) \overline{ E ( \lambda ) } - E^* ( z ) E ( \overline \lambda ) }{ \overline \lambda - z } ,  \] 
we write 
\be\la{reprokernel} K_\lambda ( z ) = \frac 1\pi \frac{ \Theta_- ( z ) \Theta_+( \overline \lambda ) - \Theta_+ ( z ) \Theta_- ( \overline \lambda ) }{ \overline \lambda - z } . \ee 
The first assertion of the theorem is then (\ref{identM11}) with $ x = L $ divided by $ \pi ( z - \overline \lambda ) $. Using the notation $ \cU $ one can write (\ref{repker}) as follows, $ (\cU \Theta_\lambda ) ( w ) = \sqrt \pi \overline {K_\lambda ( \overline w )} $. Thus, $ \cU \Theta_z \in \HE $ for any complex $ z $, and
\[ \llangle \cU \Theta_\lambda , \cU \Theta_z \rrangle_{\HE} = \pi \llangle \overline {K_\lambda }  , \overline {K_z }  \rrangle_{\HE} =  \pi \llangle K_z , K_\lambda \rrangle_{\HE} = \pi K_z ( \lambda ) = \llangle \Theta_\lambda , \Theta_z \rrangle_H . \] 
We see that $ \cU $ is an isometry on the linear span of $ \Theta_\lambda $, $ \lambda \in \C $. Let us show that this span coincides with $ H $ if (L) is obeyed. Indeed, fix an $ \alpha \in [ 0 , \pi ) $ such that the condition ($R^\prime $) is satisifed. Then $ D_\alpha $ is a selfadjoint operator in $ H $ by exercise \ref{operatorbc}. The set of eigenvectors of $  D_\alpha $ coincides with $ \{ \Theta ( \cdot , \lambda_k ) \} $, where $ \lambda_k $ are defined by \be\la{specDtheta} \Theta_+ ( L  , \lambda_k ) \cos \alpha + \Theta_- ( L , \lambda_k ) \sin \alpha = 0 ,
\ee
 and forms an orthonormal basis in $ H $ because the spectrum of $ D_\alpha $ is discrete (see exercise \ref{discspect} below). It remains to notice that the range of $ \cU $ is the whole of $ \HE $ since it contains any reproducing kernel. \end{proof}

The mapping $ \cU $ defined in this theorem is often called \textit{the Fourier transform} associated with the system $ ( \cH , L ) $. 

\begin{proposition}\la{Dalpha} Let the compatibility conditions ($ L $ ) and ($ R^\prime $) be satisfied, and $ \lambda_k $ be defined by (\ref{specDtheta}). Then the mapping $ \cU_\alpha \colon f \mapsto \{ ( \cU f ) ( \lambda_k ) \} $ is an isomorphism of $ H $ onto the space $ L^2 ( {\mathbb X} , d\mu_\alpha ) $, $ \mathbb X = \{ \lambda_k \} $, $ \mu_\alpha \( \{ \lambda_k \} \)  = \left\| \Theta \( \cdot , \lambda_k \) \right\|^{ -2 } $, and the operator $ \cU_\alpha D_\alpha \cU^*_\alpha $ is the multiplication by the independent variable.  
\end{proposition}

This is just the spectral theorem for $ D_\alpha $.

\begin{exercise} \la{discspect} The operator $ D_\alpha $ of exercise \ref{operatorbc} has discrete spectrum. \end{exercise}

\medskip

\textit{Hint.} {\small For a $ \lambda \notin \R $ the resolvent of $  D_\alpha $ is an integral operator with the kernel \[ c_\lambda \left\{ \begin{array}{cc} \Theta ( x )  \Psi^T ( x^\prime ) , & x < x^\prime , \cr \Psi ( x ) \Theta^T ( x^\prime ) , & x > x^\prime , \end{array} \right. \] $ \Psi $ being the solution of (\ref{can}) satisfying the boundary condition $ \Psi ( L ) = \( - \sin \alpha , \cos \alpha \)^T $, $ c_\lambda $ a complex constant. This integral operator is Hilbert-Schmidt, hence $ D_\alpha $ has discrete spectrum.} 

\medskip

Which entire functions $ E $ are de Branges functions of canonical systems, that is, when $ E ( z ) = E_L ( z ) $ for some canonical system $ ( \cH , L ) $ with $ L < \infty $? It is easy to see a necessary condition. Let $ ( \cH , L ) $ be a canonical system satisfying the condition of theorem \ref{spmaptheo}. Consider the element $ f \in H $ defined by $ f ( x ) \equiv \( \! \begin{array}{c} 0 \cr 1 \end{array} \! \) $. A straightforward calculation shows (a more general formula (\ref{diprob}) is to be derived later in Sect. \ref{diprobl}, so we skip the details here)  
\[ ( \cU f ) ( z ) = \frac 1{\sqrt \pi } \frac{\Theta_+ ( L, z ) - 1}z . \]
The function $ \Theta_+ ( L , \cdot ) / E $ is bounded (in fact, contractive) in $ \C_+ $, hence $ \cU f \in \HE $ implies that $ \( z + i \)^{ -1 } E^{ -1 } ( z ) \in H^2 $. HB functions satisfying this latter property are called regular (see the formal definition in Sect. \ref{facts}). We thus proved   

\begin{proposition} If the assumption of theorem \ref{spmaptheo} holds then the function $ E $ is regular.
\end{proposition}

The main result of the inverse theory in the regular case (to be obtained in Sect. \ref{Invreg}) is that this necessary condition together with a trivial normalization is sufficient. 

The following two implications from (\ref{identM}) are going to be used later on. The first one is immediate.

\begin{corollary} For all $ x \in ( 0 , L ) $, $ z \in \C_+ $
\be\la{Jcontractive} \frac 1i ( M^* ( x , z ) J M ( x , z ) - J ) \ge 0 . \ee
\end{corollary}

The other one  is obtained by setting $ \lambda = 0 $ in (\ref{identM}) and differentiating it  in $ z $ at $ z = 0 $, 
\be \la{derM} J \dot M  ( x , 0 ) = \int_0^x \cH ( s ) \d s . \ee 
The r. h. s. is nonnegative and has trace equal to $ x $, hence

\begin{lemma}\la{trace}
$ \dot\Theta_- ( x , 0 ) < 0 $, $ \dot\Phi_+ ( x , 0 ) > 0 $, and
\[ \operatorname{tr} J \dot M  ( x , 0 ) \equiv \dot\Phi_+ ( x , 0 ) - \dot\Theta_- ( x , 0 ) = x . \]
\end{lemma} 

\section{Inverse problem for polynomial HB functions}\la{invpol}

\subsection{Synopsis} In this section we solve inverse problems for finite dimensional de Branges spaces, or, in other words, for finite Jacobi matrices, by what can be described as a variant of the Stiltjes algorithm. The solution in this case is given by a canonical system defined on finitely many singular intervals. The main point in the argument is the factorization of the monodromy matrix for the system into the product of monodromy matrices for singular intervals (lemma \ref{factorsi}).

\subsection{}\la{invpol1} Let $ E $ be an HB polynomial of degree $ n $ such that $ E ( 0 ) = 1$, $ E ( t ) \ne 0 $ for $ t \in \R $. Then $ \HE $ is the space of polynomials of degree not greater than $ n - 1 $. Consider $ X \subset \HE $ - the subspace of polynomials of degree not greater than $ n - 2 $. $ X $ is an axiomatic de Branges space in the sense that it satisfies the conditions of theorem \ref{axiomatic}, hence by that theorem $ X = {\mathcal H} ( \hat{E}_1 ) $ for the function
\be\la{axE} \hat{E}_1 ( z ) = c ( z -  a ) K^X_{\overline a} ( z ) . \ee 
Here $ a \in \C_- $ is arbitrary, $ c $ is a constant depending on $a $ only, and $ K^X_{ \overline a } $ is the reproducing kernel of the space $ X $ at the point $ \overline a $. We claim that $ \hat{E}_1 $ is a combination of $ E $ and $ E^* $ with coefficients linear in $ z $ (see (\ref{E1prep}) below). Indeed, let $ P_X $ be the orthogonal projection on $ X $ in $ \HE $, \[ P_X = I - \len e \rin^{ -2 } \langle \cdot , e \rangle e , \] where $ e = E + e^{ i \alpha } E^* \perp X $ (see a remark after theorem \ref{axiomatic} if this is not obvious). On calculating,
\[ K^X_w  = P_X K_w  = K_w - \frac { \overline { e ( w ) }}{ \len e \rin^2 } e . \]
Plugging it in the definition of $ \hat{E}^1 $, 
\begin{eqnarray}\la{E1prep} & & \hat{E}_1 ( z ) =  \\ & & \mbox{const} \left[ \frac 1{2\pi i } \( E^* ( z ) E ( a) - E ( z ) E^* ( a ) \) - ( z - a ) \frac { e^* ( a ) }{ \len e \rin^2 } ( E ( z ) + e^{ i \alpha } E^* ( z ) ) \right]. \nonumber\end{eqnarray} 
Notice that $ \hat{E}_1 $ defined by (\ref{axE}) does not vanish on the real axis, for if it did, say, at a $ w \in \R $, then any element of $ {\mathcal H} ( \hat{E}_1 ) $ would vanish at $ w $, which is not the case,  hence one can define
\[  E_1 =  \frac {\hat{E}_1}{  \hat{E}_1 (0) } . \] On substituting, we get
\be\la{E1} E_1 ( z ) = ( cz + d ) E ( z ) + ( c_1 z + d_1 ) E^* ( z ) , \; \; c , d , c_1 , d_1 \in \C . \ee 

\subsection{singular interval - the direct problem}

Given a finite interval $ I = ( 0 , a )$ and a vector $ e \in \R^2 $ of unit norm, the monodromy matrix of the canonical system $ ( \cH, a ) $ with the Hamiltonian $ \cH ( x ) = \langle \cdot , e \rangle e $, $ x \in I $, is easily verified to be $ M ( \lambda ) = I - \lambda a \langle \cdot , e \rangle J e $, that is, 
\be M ( \lambda ) = I + \lambda \( \begin{array}{cc} a e_- e_+ & a e_-^2 \cr  - a e_+^2 & - a e_- e_+ \end{array} \) \la{monodrsing}  . \ee

The matrix in the second term in the r. h. s., $ R $, obeys $ R^2 = 0 $, $ R_{ 12} \ge 0 $, $ R_{ 21 } \le 0 $. Conversely, for any nonzero matrix $ R $ satisfying these three properties there exists an $ a > 0 $ and $ e \in \R^2 $, $ \len e \rin = 1 $, such that $ I + \lambda R $ is a monodromy matrix for the corresponding canonical system - it is enough to let $ a = R_{12} - R_{ 21} $, $ e_-  = \sqrt{ R_{12} / a } $, $ e_+  = ( \operatorname{sign} R_{11} ) \sqrt{ - R_{21} / a } $.

\subsection{back}\la{invpol2}
With HB polynomials $ E $ and $ E_1 $ defined above let 
\be\la{ThetaE} \Theta = \frac 12 \( \begin{array}{c} E + E^* \cr \frac 1i ( E - E^* ) \end{array} \), \ee \[ \Theta_1 = \frac 12 \( \begin{array}{c} E_1 + E^*_1 \cr \frac 1i ( E_1 - E^*_1 ) \end{array} \) , \] 
so $ E = \Theta_+ + i \Theta_- $, $ E_1 = \Theta_+^1 + i \Theta_-^1 $.
Then (\ref{E1}) says that 
\be\la{TT1} \Theta_1 ( \lambda ) = ( \Lambda_0 + \lambda \Lambda_1 ) \Theta ( \lambda ) \ee for some $ 2\times 2 $ constant matrices $ \Lambda_0 $, $ \Lambda_1 $ with real entries. We are going to show that $ \Theta = M \Psi $ where $ M $ is the monodromy matrix corresponding to a singular interval and $ \Psi $ corresponds to an HB function in the same way as $ \Theta $ corresponds to $ E $.
 
\begin{lemma}\la{factorsi} Let $ n \ge 2 $. Then

$1^\circ $. $ \Lambda_0 $ is invertible.

$2^\circ $. $ S \colon = \Lambda_0^{ -1 } \Lambda_1 = \( \begin{array}{cc} s & s_1 \cr s_2 & -s \end{array} \) $, where $ s^2 + s_1 s_2 =0 $, $ s_2 \ge 0 $, $ s_1 \le 0 $. 

$3^\circ $.  Let $ \Psi =  \Lambda_0^{ -1 } \Theta_1 $. Then 
\be\la{factout} \Theta ( \lambda ) = ( I - \lambda S ) \Psi ( \lambda ),\ee 
and $ \Psi_+ + i \Psi_-  $ is an Hermite - Biehler function. 
\end{lemma}

\begin{proof}
$ 1^\circ $. $ \Theta_1 $ and $ \Theta $ are polynomials of degrees $ n-1 $ and $ n $, resp. Let $ \theta_n , \theta_{ n-1 } \in \R^2 $ be the coefficients of the vector polynomial $ \Theta $ at $ \lambda^n $, $ \lambda^{ n-1 } $, resp., hence $ \theta_n \ne 0 $. Then (see (\ref{TT1}))
\begin{eqnarray} \Lambda_1 \theta_n  & = & 0 , \la{L0} \\ \Lambda_1 \theta_{ n-1 } + \Lambda_0 \theta_n & = & 0 . \la{L1} \end{eqnarray}
First we check that the two terms in the second identity do not vanish separately. Indeed, suppose $ \Lambda_1 \theta_{ n-1 } = \Lambda_0 \theta_n = 0 $, which implies that $ \theta_n $ belongs to kernels of both $ \Lambda_0 $ and $ \Lambda_1 $. This means that
\[ \Lambda_0 + \lambda \Lambda_1 = \langle \cdot , g \rangle \cdot \left[ \textrm{linear vector function of } \lambda \right] , \]
where $ g \in \R^2 $ is a non-zero vector orthogonal to $ \theta_n $. Thus, 
\[ \Theta_1 (\lambda ) = \( \Lambda_0 + \lambda \Lambda_1 \) \Theta ( \lambda ) = p( \lambda ) \cdot \left[ \textrm{linear function of } \lambda \right] , \] where $ p ( \lambda ) $ is a scalar polynomial, $ p ( \lambda ) = \langle \Theta ( \lambda ) , g \rangle $. The vector function $ \Theta_1 $ does not vanish since $ E_1 $ has no real zeroes, hence $ p ( \lambda ) $ is a non-zero constant, that is, all coefficients of the polynomial $ \Theta $ are orthogonal to $ g $, except at most the constant term. It follows that $ \Theta ( \lambda ) = q ( \lambda ) \theta_n + \Theta ( 0 ) $ for a scalar polynomial $ q ( \lambda ) $. Since $ \Theta ( 0 ) = \( 1 , 0 \)^T $, we infer that 
\[ \frac{\Theta_+}{\Theta_-} = \textrm{a real constant } + \frac 1{\theta^-_n q } . \] 
The l. h. s. is a Herglotz function, which implies that $ q $ is a polynomial of degree at most $ 1 $, and the same is true of $ \Theta $, a contradiction ($ n \ge 2 $). Thus $ - \Lambda_1 \theta_{ n-1 } = \Lambda_0 \theta_n \ne 0 $, so $ \dim \Ran \Lambda_1 = 1 $ and $ \Ran \Lambda_1 \subset \Ran \Lambda_0 $. If $ \Ran \Lambda_1 = \Ran \Lambda_0 $ then $ \Theta_1 = p e $ where $ p $ is a scalar real polynomial and $ e \in \R^2 $ is a constant vector. The latter contradicts the HB property of $ E_1 $. We obtain that $ \Ran \Lambda_1 $ is contained in $ \Ran \Lambda_0 $ as a proper subspace. $ 1^\circ $ is proved.  

$ 2^\circ $. (\ref{L0}) and  (\ref{L1}) say that $ S \theta_n = 0 $, $ S \theta_{ n-1 } = - \theta_n $ hence $ \det S = \operatorname{tr} S = 0 $, $ \( I + \lambda S \)^{ -1 } = I - \lambda S $, and the identity (\ref{factout}) is immediate from solving  (\ref{TT1}) with respect to $ \Theta $. To establish $ 2^\circ $, it remains to show that $ s_2 \ge 0 $.  

For all $ \lambda \in \C_+ $ by the HB property of the function $ E $ we have 
\[ 
0 < \frac 1i \llangle J \Theta , \Theta \rrangle =  \frac 1i \llangle ( J - \lambda JS ) \Psi , ( I - \lambda S ) \Psi \rrangle = \frac 1i \llangle J \Psi , \Psi \rrangle + \frac {\overline \lambda - \lambda }i \llangle JS \Psi , \Psi \rrangle = \dots \]
Here we took into account that $ S^* J S = 0 $ for $ S $ is rank 1 and has real entries, and that $ JS $ is selfadjoint. The first term in the r. h. s. is $ O ( \lambda^{ 2 ( n - 1 ) } ) $ at large $ | \lambda | $, thus one can continue letting $ \Psi_{max} $ to be the coefficient at $ \lambda^{ n-1 } $ in the polynomial $ \Psi $, 
\[ \dots = - 2 \Im \lambda \left| \lambda \right|^{ 2 ( n - 1 ) } \llangle JS \Psi_{max} , \Psi_{max} \rrangle  + o \( \lambda^{ 2 n - 1 } \) , \; \lambda = i \tau , \tau \to + \infty . \] 
We infer that $ \llangle JS \Psi_{max} , \Psi_{max} \rrangle \le 0 $. Notice that $ JS $ is a selfadjoint matrix with zero determinant, hence either $ JS \le 0 $, or $ JS \ge 0 $, hence its quadratic form vanishes on $ \Psi_{max} $ iff $ J S \Psi_{max} = 0 $, the latter not being the case since $ S \Psi_{ max} = - \theta_n $. Thus, $ \llangle JS \Psi_{max} , \Psi_{max} \rrangle < 0 $, which means $ JS \le 0 $, that is, $ s_2 \ge 0 $, $ s_1 \le 0 $.  

$ 3^\circ $. As (\ref{factout}) has already been established, it remains to check that $ \Psi_+ + i \Psi_- $ is HB. Notice that either $ - \Psi_+ / \Psi_- $ or $ \Psi_+ / \Psi_- $ is a non-constant Herglotz function, depending on the sign of $ \det \Lambda_0 $, since $ \Theta_1^+ / \Theta_1^- $ is. $ 3^\circ $ will be proven if we show that $ \Psi_+ /\Psi_- $ is Herglotz. Consider the Jordan form of $ S $, 
\[ S = W \( \begin{array}  {cc} 0 & 1 \cr 0 & 0 \end{array} \) W^{ -1 } , \] 
and choose $ W $ to have real entries (this is possible since $ S $ has ones). Then, 
(\ref{TT1}) multiplied on the left by $ W^{ -1 } \Lambda_0^{ - 1} $ is
\[ W^{ -1 } \Psi ( \lambda ) = \( \begin{array}  {cc} 1 & \lambda \cr 0 & 1 \end{array} \) W^{ -1 } \Theta ( \lambda ) ,\] 
which implies that
\[ \frac{\( W^{ -1 } \Psi  \)_+}{\( W^{ -1 } \Psi  \)_- } =  \frac{\( W^{ -1 } \Theta \)_+}{\( W^{ -1 } \Theta \)_- } + \lambda . \] 
Now, if $ \det W > 0 $ then the r .h. s. is Herglotz, hence so is $ \Psi_+ /\Psi_- $. If $ \det W < 0 $, arguing by contradiction, assume that $ - \Psi_+ / \Psi_- $ is Herglotz. Then the l. h. s. is Herglotz, the first term in the r. h. s. is "minus - Herglotz". Gathering the ratio terms in the l. h. s. we obtain that a sum of two Herglotz functions equals to $ \lambda $.  This means that either function is linear which is only possible if $ n = 1 $, a contradiction, hence $ \Psi_+ / \Psi_- $ is Herglotz in any case.     
\end{proof}

By construction the function $ E_\Psi = \Psi_+ + i \Psi_- $ is an HB polynomial of degree $ n - 1 $ having no real zeroes, and $ E_\Psi ( 0 ) = 1 $ since $ \Psi ( 0 ) = \Theta ( 0 ) = \( \begin{array}{c} 1 \cr 0 \end{array} \)$, hence the lemma can be applied to $ E_\Psi $ in the place of $ E $ if $ n > 2 $. Applying it $ n - 1 $ times we find that 
\[ \Theta ( \lambda ) = ( I + \lambda R_1 ) \dots ( I + \lambda R_{ n-1 } ) \Theta_{ n-1 } ( \lambda ) , \] 
where each of the matrices $ R_j $ has the form $ \( \begin{array}{cc} \rho & \nu \cr \tau & - \rho \end{array} \) $, $ \rho^2 + \nu \tau = 0 $, $ \tau \le 0 $, $ \nu \ge 0 $, and $ \Theta_{ n-1 } $ is a linear function such that $ \Theta_{ n -1 }  ( 0 ) = \( 1, 0 \)^T $, and $ \Theta_{n-1}^+ + i \Theta_{ n-1}^- $ is an HB function. It follows that $ \Theta_{ n-1 } = \Bigl( \! \begin{array}{c} 1 \cr 0 \end{array} \! \Bigr) + \lambda \Bigl( \! \begin{array}{c} a \cr b \end{array} \! \Bigr) $, with $ b < 0 $, that is, 
\[ \Theta_{ n-1 } = \( I + \lambda \( \begin{array}{cc} a & -a^2/b \cr b & -a \end{array} \) \) \Bigl( \! \begin{array}{c} 1 \cr 0 \end{array} \! \Bigr) . \] 
We thus have established the following

\begin{theorem}\la{poll}
Let $ E $ be a polynomial HB function having no real zeroes and such that $ E ( 0 ) = 1 $. Then ($ n = \deg E $) 
\[ \Theta ( \lambda ) = M_1 ( \lambda ) \dots M_n ( \lambda ) \Bigl( \! \begin{array}{c} 1 \cr 0 \end{array} \! \Bigr) , \] where $ M_j = I + \lambda R_j $, $ \det R_j = \operatorname{tr} R_j =0 $, $ R_{12} \ge 0 $, $ R_{21} \le 0 $, and $ \Theta $ is defined by (\ref{ThetaE}).
\end{theorem}

Comparing this theorem with a remark after (\ref{monodrsing}) we obtain the following

\begin{corollary}\label{solfindim}
For any $ E $ satisfying the assumption of the theorem there exists a canonical system $ ( \cH , L ) $, $ L < \infty $, such that $ ( 0, L ) $ is a union of finitely many $ \cH $-indivisible intervals, and $ E_L  = E $.
\end{corollary}

Notice that the compatibility condition (L) is satisfied for the constructed system, for $ b \ne 0 $. 

\subsection{inverse problem for a measure supported on finitely many points}\la{fsupp}\footnote{The material of this subsection is not used until section \ref{invsing}.}
From the viewpoint of mathematical physics the inverse problem just studied is not the most natural - the function $ E $  is never given initially. The natural one is - given a measure supported by finitely many points and an $ \alpha $, find a canonical system such that the measure is the spectral measure of operator $ D_\alpha $ in the sense given by proposition \ref{Dalpha}. The following theorem essentially solves this latter problem. The final result is formulated as remark at the end of this section.

\begin{theorem}\la{finmes} Let $ \mu $ be a measure supported on finitely many points and such that $ \mu (\{ 0 \} ) \ne 0 $. Then there exists a canonical system $ ( \cH , L ) $ such that for all $ z \in\C_+ $
\[ \Im \( \frac{ \Phi_- ( L , \lambda ) }{ \Theta_- ( L , \lambda ) } \) = \Im \lambda \int \frac{ d \mu ( t ) }{ t^2 + \( \Re \lambda \)^2 } + c \Im \lambda , \] $ c $ being a real nonnegative constant.  \end{theorem}

\textit{Heurisitics}. Assuming that the theorem holds, that is, 
\[ \frac{ \Phi_- ( L , z ) }{ \Theta_- ( L , z ) }= \sum_{ t_j \in \operatorname{supp} \mu } \frac { \mu ( \{ t_j \} ) }{ t_j - z }  +  c z  + d ,  \; z \in \C_+ ,\]
\[ \mu  ( \{ t_j \} )= - \frac{\Phi_- ( t_j )}{ \dot{\Theta}_- ( t_j ) } , \]   let us try to calculate the system in terms of the measure. Since $ \Phi_- ( L , \cdot ) $ and $ \Theta_- ( L , \cdot ) $ do not have common zeroes, the set of zeroes of $ \Theta_- $ coincides with $ \operatorname{supp} \mu $. This means that $ \Theta_- = \textrm{const} \prod_{ t_j \in \operatorname{supp} \mu } ( z - t_j ) $. Since $ 0 \in \operatorname{supp} \mu $, the constant in front of the product is $ \dot{\Theta}_- ( 0) $, and $ \mu ( 
\{ 0 \} ) = - 1 / \dot{\Theta}_- ( 0) $ on account of $ \Phi_- ( 0) = 1 $. Thus $ \Theta_- $ is defined completely, 
\be\la{tminus} \Theta_- ( z) = - \frac 1{\mu_0 } \prod_{ t_j \in \operatorname{supp} \mu } ( z - t_j ) . \ee
Then \[ \frac{\Theta_+}{\Theta_-} = \sum_{ t_j \in \operatorname{supp} \mu } \frac{ \Theta_+ ( t_j ) }{ \dot{\Theta}_- ( t_j ) } \frac 1{ z - t_j } + c_1 z + d_1 , \; z \in \C_+ , \] for some $ c_1 \ge 0 $, $ d_1 \in \R $. The monodromy matrix $ [ \Theta , \Phi ] $ has unit determinant, hence $ \Theta_+ ( t_j ) \Phi_- ( t_j ) = 1 $. This allows to express $ \Theta_+ ( t_j ) $ via $ \mu ( \{ t_j \} ) $ and already known $ \Theta_- $ to find that $ \Theta_+ $ is defined up to the constants $ c_1 $ and $ d_1 $,
\be\la{tplus} \Theta_+ ( z ) =\(  - \sum_{ t_j \in \operatorname{supp} \mu } \frac 1{  \mu ( \{ t_j \} ) \dot{\Theta}_-^2 ( t_j ) } \frac 1{ z - t_j } + c_1 z + d_1 \) \Theta_- ( z )  . \ee
With $ \Theta_\pm $ defined from the measure one can use the solution of the inverse problem for the function $ E = \Theta_+ + i \Theta_- $ obtained in the previous subsection to recover the system.

\begin{proof} Let us \textit{define} $ \Theta_\mp  $ by (\ref{tminus}) and (\ref{tplus}), respectively, letting $ c_1 = 0 $ and choosing an arbitrary $ d_1 \in \R $. Since $ \mu ( \{ t_j \} ) > 0 $, the function $ \Theta_+ / \Theta_- $ is Herglotz in $ \C_+ $ and non-constant, hence the polynomial $ E \colon = \Theta_+ + i \Theta_- $ is an HB function. By construction, $ E( 0 ) = 1 $. By corollary \ref{solfindim} there exists a canonical system $ ( \cH , L ) $, $ L < \infty $, such that $  \Theta ( L , z ) \equiv \( \! \begin{array}{c} \Theta_+ ( z) \cr \Theta_- ( z ) \end{array} \!\) $. Let $ \Phi ( x , \lambda ) $ be the solution of this system with $ \Phi ( 0 , \lambda ) = \( \! \begin{array}{c} 0 \cr 1 \end{array} \! \) $. Then the function  $ \Phi_- ( L , z )  / \Theta_- ( L , z ) $ is Herglotz (see exercise \ref{columnrow} and the text following it). Consider its Herglotz representation,
\[ \frac{ \Phi_- ( L , z )  }{ \Theta_- ( L , z ) } = \sum \underbrace{ - \frac{
 \Phi_- ( L , t_j ) }{ \dot{\Theta}_- ( L,  t_j ) } } \frac 1{ t_j - z } + {\textrm{a linear function}} . \]  
Since the monodromy matrix of a canonical system has unit determinant, the underbraced expression equals to $ - \( \Theta_+ ( L , t_j ) \dot{\Theta}_- ( L,  t_j ) \)^{ -1 }  $, which is $ \mu ( \{ t_j \} ) $ by construction. The theorem follows.
\end{proof}

Notice that the system constructed in the proof of the theorem satisfies compatibility conditions (L) and (R). Indeed, it satisfies (L) because the system constructed in theorem \ref{poll} and corollary \ref{solfindim} does. The condition (R) is satisfied, for if it were not, then the monodromy matrix, $ M ( z ) $, of the constructed system would be a product of the monodromy matrix for the singular interval  whose right end is $ L $ and $ N ( z ) $, the monodromy matrix of the system $ ( \cH , L^\prime ) $ obtained by cutting off the said singular interval. The former is $ \begin{pmatrix} 1 & \nu z \cr 0 & 1 \end{pmatrix} $ (condition (R) being violated means the monodromy matrix for this interval is (\ref{monodrsing}) with $ e_+ =  0 $, $ e_- = 1 $), hence  
\[ M ( z ) =   \begin{pmatrix} 1 & \nu z \cr 0 & 1 \end{pmatrix}  N ( z ) \] 
with a $ \nu > 0 $. It would follow that $ \Theta_+ ( L , z ) /  \Theta_- ( L , z ) = \textit{ a Herglotz function} +  \nu z $, which is a contradiction because $ \Theta_+ ( L , z ) /  \Theta_- ( L , z ) = o ( |z| ) $ as $ |z |\to \infty $ by construction (we let $ c_1 = 0 $ in the proof of theorem \ref{finmes}). Also, $  - \Theta_+ ( L , t_j ) \dot{\Theta}_- ( L,  t_j ) = \len \Theta ( \cdot , t_j ) \rin^2 $ by (\ref{repker}) with $ z = \lambda = t_j $, so $ \mu ( \{ t_j \} ) = \len \Theta ( \cdot , t_j ) \rin^{ -2 } $, and it follows that 

\begin{remarknonumb} For any measure $ \mu $ satisfying the assumptions of theorem \ref{finmes} there exists a canonical system $ ( \cH , L ) $ satisfying the compatibility conditions (L) and (R) and such that the selfadjoint operator $ D_{ \pi/2 }$ is defined correctly, and $ \mu $ coincides with its spectral measure in the sense that the measure $ \mu_{ \pi/2 } $ from proposition \ref{Dalpha} coincides with $ \mu $.
\end{remarknonumb}

\section{Inverse problem - regular case}\la{Invreg}

Let $ E $ be a regular HB function having no real zeroes, $ E ( 0 ) = 1 $, and let $ \Theta $ be defined by (\ref{ThetaE}). The inverse problem is to construct a canonical system $ ( \cH , L ) $, $ L < \infty $, such that $ \Theta = \Theta_L $.

The structure of the argument is as follows. On the first step, we approximate the required canonical system by a system made of finitely many singular intervals. To do so, we use a polynomial approximation to $ \Theta $ constructed by cutting off the Weierstra\ss\  product for $ \Theta_- $. The result of this step is given by proposition \ref{altern}, where $ \Theta $ is represented as the action of the monodromy matrix of a canonical system on an HB vector, $ G_h $. After this is done, the main problem is that we do not know if the lengths, $ L_N $, of the systems constructed are bounded above. To circumvent this difficulty we notice that (\ref{mh}) is precisely the condition that the de Branges space constructed from the function $ G_h^+ + i G^-_h $ is a subspace in $ \HE $ (Theorem \ref{isometry}). By the abstract theory of de Branges spaces, a de Branges subspace of a regular de Branges space is regular itself (Theorem \ref{isomembed}). Thus we will have established that $ G_h^+ + i G^-_h $ is a regular HB function. On the next step, we reconstruct a matrix satisfying properties of a monodromy matrix from its first column (Appendix I). Applied to $ G_h $ and $ \Theta $, this assertion allows to conclude that $ L_N $ are indeed bounded above, and thus the limit of the polynomial approximation solves the inverse problem.    

Let $ t_j $, $ j \ge 0 $, $ t_0 = 0 $, be the set of zeroes of $ \Theta_- $ ordered by $ | t_j | \le | t_{ j+1 } | $. The function $ \Theta_- $ is of finite exponential type by the Krein theorem, hence for each $ N > 0 $ we have \[ \Theta_- ( z ) = \Theta_-^N ( z ) e^{ \alpha_N z } R_N  ( z ) ,  \]
\[ \Theta_-^N ( z ) = \dot\Theta_- ( 0 ) z \prod^{N-1} \( 1 - \frac z{t_j} \) , \alpha_N \in \R , R_N ( z ) = \prod_{ j \ge N } \( 1 - \frac z{t_j} \) e^{ \frac z{t_j} } - \] the canonical product for $ \Theta_- $.

Let \[ \frac {\Theta_+ ( z ) }{ \Theta_- ( z ) } = \sum_j \frac{\mu_j}{ z - t_j } + a + b z \] be the Herglotz representation of the function $ \Theta_+ / \Theta_- $. Define \[ \Theta_+^N ( z ) =  \( \sum_{j = 0 }^N \frac{\mu_j}{ z - t_j } + a + b z \) \Theta_-^N ( z ) . \] Then $ \Theta_N \colon = \( \! \begin{array}{c} \Theta_N^+ \cr \Theta_N^- \end{array} \! \) $ is a polynomial, $ E_N = \Theta_N^+ + i \Theta_N^- $ is an HB polynomial having no real zeroes, and $ E_N ( 0 ) = 1 $ since $ E ( 0 ) = 1 $ and hence $ \mu_0 = 1/\Theta^\prime_-  ( 0 ) $. Corollary \ref{solfindim} applies to $ E_N $. Let $ ( \cH_N , L_N ) $ be the corresponding canonical system, $ M_N ( x , z ) $ its fundamental solution, $ M_N ( z ) $ its monodromy matrix ($ M_N ( z ) = M_N ( L_N , z ) $), $ M_N ( x , y ; z ) $ the monodromy matrix for the interval $ ( x , y ) \subset ( 0 , L_N ) $, 
$ M_N^h ( z ) $ the monodromy matrix for the interval $ ( L_N - h , L_N ) $, $ 0 < h < L_N $. Then for any $ h \in ( 0 , L_N ) $ 
\be\la{mult} M_N ( z ) = M_N^h ( z ) M_N ( L_N - h ; z ) . \ee
In the following argument we choose consecutively subsequences of $ N $. The indices labeling the subsequences are suppressed in the notation. Notice that given an $ \ell \ge 0 $, the systems $ ( \cH_N , L_N ) $ can be chosen so that $ L_N \ge \ell $. To this end, if $ L_N < \ell $ for the system constructed from $ E_N $ in corollary \ref{solfindim}, we augment the system at the left end by a singular interval, $ I $, of length $ \ell - L_N $ with $ \cH ( x ) = \( \! \begin{array}{cc}  0 & 0 \cr 0 & 1 \end{array} \! \) $ for $ x \in I $. An appropriate $ \ell $ will be chosen towards the end of the argument.

Fix an arbitrary $ \ell > 0 $, and let as assume that the augmentation is done wherever necessary, so that $ L_N \ge \ell $ for all $ N $. Let $ F_N ( x ) = \int_{ L_N - h }^{ L_N - h + x } \cH_N ( s ) \d s $. Given an $ h \in [ 0 , \ell ] $ the functions $ F_N $ belong to $ C ( 0 , h ) $, are uniformly bounded ($ \len F_N ( x ) \rin \le x $) and equicontinuous ($ \| F_N ( x + \delta ) - F_N ( x ) \| \le \delta $). Hence there exists a subsequence $ \{ F_N \} $ and a function $ F \in C( 0, h ) $ such that $ F_N \Longrightarrow F $ in $ C ( 0 , h ) $. The limit function is a. c. (even Lipschitz) hence for a. e.  $ x \in ( 0, h ) $ there exists $ F^\prime ( x ) $, $ F^\prime ( x ) \ge 0 $ ($ F $ is monotone non-decreasing since each $ F_N $ is), and $ \operatorname{tr} F^\prime ( x ) = 1 $. The function $ F^\prime $ is going to be the Hamiltonian of a canonical system solving the inverse problem.  

Notice that $ \len M_N^h ( z ) \rin \le \exp ( h |z| ) $ whenever $ 0 < h < L_N $ by (\ref{estM}). By the Montel theorem it follows that for each $ h \in [ 0 , \ell ] $ there exists a subsequence $ N_k $ and an entire function $ M_h $ such that $ M_{ N_k }^h \Rightarrow M_h $ as $ N_k \to \infty $ uniformly on compacts in $ \C $. Our first goal is to show that $ M_h $ is a monodromy matrix of a canonical system on the interval $ ( 0 , h ) $ with the Hamiltonian $ F^\prime $. On writing the integral equation ($ L_N - h < x < L_N $), 
\be\la{inte} J M_N ( L_N - h , x + \delta ) - J M_N ( L_N - h , x )  = z \int_{ L_N - h + x}^{L_N - h + x + \delta} \cH_N ( s ) M_N^h ( s ) \d s . \ee 
Then
\[ \len M_N ( L_N - h , x + \delta ) - M_N ( L_N - h , x ) \rin \le |z| \int_x^{x + \delta} \|  M_N^h ( s ) \| \d s \le |z| e^{ h |z| } \delta \] 
rendering the system $ \{ M_N ( L_N - h , \cdot \, + L_N - h , z ) \} $ equicontinuous and hence compact in $ C ( 0 , h ) $. Thus there exists a subsequence of $ N $ and an $ \widetilde M \in C ( 0 , h ) $ such that $ M_N ( L_N - h , \cdot + L_N - h , z ) \Longrightarrow \widetilde M $ in $ C ( 0 , h ) $. Consider (\ref{inte}) written in the form ($ x \in ( 0 , h ) $) 
\[ J M_N ( L_N - h , L_N - h + x ) - J = z \int_0^x \cH_N ( s + L_N - h ) M_N^h ( L_N - h , L_N - h + s ) \d s .  \] Letting $ N \to \infty $ we find
\[ \begin{array}{ccc} J \widetilde M ( x ) - J = z & \underbrace{ \int_0^x \cH_N ( s + L_N - h ) \widetilde M ( s ) \d s } & + r_N ( x ) , \cr & (I) & \end{array} \] where $ r_N ( x ) \to 0 $ uniformly in $ x \in ( 0 , h ) $. Integrating by parts gives 
\[ (I) = F_N ( x ) \widetilde M ( x ) - \int_0^x F_N ( s ) {\widetilde M}^\prime ( s )  \d s . \] 
In this formula one can pass to the limit $ N \to \infty $ and integrate by parts back. We end up with
\[ J \widetilde M ( x ) - J = z \int_0^x F^\prime ( s  ) \widetilde M ( s ) \d s , \] 
that is, $ \widetilde M ( x ) $ is a fundamental solution of the canonical system $ ( F^\prime , h ) $, and $ \widetilde M ( h , z ) = M_h ( z ) $ by construction. Notice that $ M_h $ is an entire function of finite exponential type not greater than $ h $.

Consider now the first column in (\ref{mult}). On multiplying by $ e^{ \alpha_N z } $, 
\[ \begin{array}{ccccc} e^{ \alpha_N z } \Theta_N ( z ) &\!  = \! & M_N^h ( z ) & \!\!\!\!\! G_N^h ( z ) , & \! G_N^h ( z ) \colon = e^{ \alpha_N z } M_N ( L_N - h , z )  \Bigl( \begin{array}{c} 1 \cr 0 \end{array} \Bigr) .\cr \noalign{\vskip-3pt} \Downarrow & & \Downarrow & & \cr \Theta ( z ) & & M_h ( z ) & & \end{array} \]  
The convergence in the l. h. s. is uniform on compacts in $ \C \setminus \R $, in the r. h. s. on compacts in $ \C $. Since $ \det M_N^h = 1 $, the inverses of $ M_N^h $ converge uniformly on compacts in $ \C $ as well, hence $ G_N^h $ converge uniformly on compacts in $ \C \setminus \R $ to the real entire function $ G_h \colon = M_h^{ -1 } \Theta $. This function is of finite exponential type since so are $ M_h $ and $ \Theta $.

\begin{lemma}
The function $ G_h^+ + i G^-_h $ is either HB, or $ \equiv 1 $. \end{lemma}

\begin{proof} We have, \[ \frac{\Theta_N^+ ( L_N - h , z )}{\Theta_N^- ( L_N - h , z )} \to \frac{G_h^+ ( z ) }{ G_h^- ( z ) } \] at any point $ z \notin \R $ such that $ G_h^- ( z ) \ne 0 $. The functions in the l. h. s. are Herglotz in $ \C_+ $ for each $ N $, hence if they converge for all $ z $ from an open subset in $ \C_+ $, then they converge in the whole of $ \C_+ $ uniformly on compacts. It follows that either $ G_h^- \equiv 0 $, or $ G_h^- ( z ) \ne 0 $ for $ z \notin \R $, because $ G^+_h $ and $ G^-_h $ have no common zeroes as $\Theta ( z ) $ does not vanish by assumptions about $ E $. We are going to show that these two possibilities correspond exactly to the alternatives in the lemma. Suppose first that
 $ G_h^- ( z ) \ne 0 $ for $ z \notin \R $. Then $ G_h^+ / G_h^- $ is a Herglotz function in $ \C_+ $ as a limit of Herglotz functions. This function is not constant since $ G_h ( 0 ) = \( 1 , 0 \)^T $, hence $ G_h^+ + i G^-_h $ is HB.

Now let $ G_h^- \equiv 0 $. We have, $ \Theta = G_h^+ \widetilde \Theta $ where $ \widetilde \Theta $ stands for the first column of the matrix $ M_h $. Obviously, the function $ G_h^+ $ does not vanish, hence $ G_h^+ ( z ) = A e^{ b z } $ for some real $ A , b $. It follows that $ b = 0 $. Indeed, $ \int^\infty t^{ -2 } \log \| \Theta ( t ) \| \d t $ is finite by the regularity of $ E $ (see Appendix), and $ \int^\infty t^{ -2} \log \| \widetilde \Theta ( t ) \| \d t $ is finite because $ \widetilde \Theta $ is the first column of a monodromy matrix of a canonical system on a finite interval, hence by proposition \ref{HB} $ \widetilde \Theta_+ +i \widetilde \Theta_- $ is either a regular HB function, or equal to $ 1 $ identically. These two integrals can be finite simultaneously only if $ b = 0 $. Finally, $ A = 1 $ since $ G_h^+ ( 0 ) = \Theta_+ ( 0 ) = 1 $, hence  $ G_h^+ \equiv 1 $ and the lemma is proved.
\end{proof}

On account of $ \ell > 0 $ being arbitrary, we have thus established the following

\begin{proposition}\la{altern} In the assumptions of this section for any nonnegative $ h $ there exists a canonical system $ ( \cH_h , h ) $ such that either 
\be\la{mh} \Theta ( z ) = M_h ( z ) G_h ( z ) ,\ee 
where $ M_h $ is the monodromy matrix for this system, and $ E_h \colon = G_h^+ + i G^-_h $ is an HB function, or $ \Theta ( z ) $ coincides with the first column of $ M_h $. \end{proposition}
  
Thus, either $ \Theta ( z ) $ is the first column of $ M_h $  for some $ h > 0 $, and then the inverse problem is solved, or $ E_h $ is an HB function for all $ h > 0 $. We are going to show eventually that the latter is impossible. Let $ h $ be such that $ E_h $ is an HB function. Our immediate aim is to establish that  $ E_h $ is regular in order to apply theorem \ref{restore} to $ G_h $. To do so, we prove that $ {\mathcal H} ( E_h ) $ is isometrically embedded into $ \HE $ (possibly, up to a one-dimensional subspace) and then use a theorem saying that a de Branges subspace of a regular de Branges space is regular. 

\begin{theorem}\la{isometry} Let $ W $ be an HB function, $ \mu $ a measure on $ \R $. Denote $ S = W^* / W $. Then $ \HW $ is isometrically embedded into $ L^2 \( \R , \left| W \right|^{ -2 } \d \mu \) $ if and only if there exists an $ A \in H^\infty $, $ \len A \rin_\infty \le 1 $, such that  
\[ \Re \frac{ 1 + S ( z ) A ( z ) }{ 1 - S ( z ) A ( z ) } = \frac{ \Im z }\pi \int \frac {\d \mu ( t )}{\( t - \Re z \)^2 + \( \Im z \)^2 } , \; z \in \C_+ . \]

Let $ A \in H^\infty $, $ \len A \rin_\infty \le 1 $, then $ \Re \frac{ 1 + S A}{ 1 - S A } $ is a positive harmonic function in $ \C_+ $. Let 
\be\la{infinitymass} \Re \frac{ 1 + S ( z ) A ( z ) }{ 1 - S ( z ) A ( z ) } = \frac{ \Im z }\pi \int \frac {\d \mu ( t )}{\( t - \Re z \)^2 + \( \Im z \)^2 } + c_0 \Im z , \; z \in \C_+ ,  \ee 
be its Herglotz representation. If $ c_0 > 0 $ then there exists an $ \alpha \in \R $ such that $ W + e^{ i \alpha } W^* \in \HW $, and $ {\mathcal X} \colon = \HW \ominus {\mathcal L } \{ W + e^{ i \alpha } W^* \} $ is isometrically contained in $ L^2 \( \R , \left| W \right|^{ -2 } \d \mu \) $. \end{theorem}

We would like to apply it with $ \d \mu = \left| \frac{ E_h }E \right|^2 \d t $, $ W = E_h $. To do so we need to find a contractive analytic function, $ A $, in $ \C_+ $ such that 
\be\la{SA} \Re \frac{ 1 + S ( t ) A ( t ) }{ 1 - S ( t ) A ( t ) } =  \left| \frac{ E_h ( t ) }{ E ( t ) } \right|^2, \; t \in \R . \ee  

\textbf{Heuristics.} How to find the $ A $ in the situation of proposition \ref{altern}? The natural requirement is that $ A $ must depend on $ M_h $ only (and not $ G_h $). Let $ G_h $ be a linear function, then $ S ( z ) = - \frac{ z - i \tau }{ z + i \tau } $, $ G_h ( z ) = \Bigl( \! \begin{array}{c} 1 \cr -iz/\tau \end{array} \! \Bigr) $, $ \tau > 0 $. Considering the limits $ \tau \to + \infty $ and $ \tau \to 0+ $ for each fixed $ z \in \R $ we find 
\[ \Re \frac{ 1 + A ( z ) }{ 1 - A ( z ) } = \left\| M_h ( z ) \Bigl( \! \begin{array}{c} 0 \cr 1 \end{array} \! \Bigr) \right\|^{-2},\;  \Re \frac{ 1 - A ( z )}{ 1 + A ( z ) } = \left\| M_h ( z ) \Bigl( \! \begin{array}{c} 1 \cr 0 \end{array} \! \Bigr) \right\|^{-2} . \] 
On dividing these and denoting $ E_1 = \( M_h \)_{ 11 } + i \( M_h \)_{ 21 } $, $ E_2 = \( M_h \)_{ 12 } + i \( M_h \)_{ 22 } $, we find \[ \left| \frac{ 1 + A }{ 1 - A } \right| = \left| \frac{ E_2 }{ E_1 }\right| .\] This suggests to define $ A$ to be the Caley transform of $ E_2 / E_1 $.

\begin{lemma}  \[ A ( z )\colon = \frac{ E_2 ( z ) - i E_1 ( z ) }{ E_2 ( z ) + i E_1 ( z ) } \] is a contractive analytic function in $ \C_+ $, $ \| A \|_\infty \le 1 $, and (\ref{SA}) holds.
\end{lemma}

\begin{proof} For $ z \in \C_+ $ \bequnan \Im \( E_2 \overline{E_1} \) = - \frac1{2i} \llangle J \( \! \begin{array}{c} E_1 \cr E_2 \end{array} \! \) , \( \! \begin{array}{c} E_1 \cr E_2 \end{array} \! \) \rrangle = - \frac1{2i} \llangle J M_h^T \Bigl( \! \begin{array}{c} 1 \cr i \end{array} \! \Bigr) , M_h^T \Bigl( \! \begin{array}{c} 1 \cr i \end{array} \! \Bigr) \rrangle = \\ \frac1{2i} \llangle M_h^* J M_h \( \! \begin{array}{c} 1 \cr -i \end{array} \! \) , \( \! \begin{array}{c} 1 \cr -i \end{array} \! \) \rrangle \ge \llangle J \Bigr( \! \begin{array}{c} 1 \cr -i \end{array} \! \Bigr) , \Bigl( \! \begin{array}{c} 1 \cr i \end{array} \! \Bigr) \rrangle = 1 . \eequnan 
In the last step we took into account that $ i^{ -1 } \( M_h^* ( z ) J M_h ( z ) - J \) \ge 0 $ by (\ref{Jcontractive}). Thus $ E_2/E_1 $ is a Herglotz function in $ \C_+ $, and $ A $ is a bounded analytic function in $ \C_+ $, $ \len A \rin_\infty \le 1 $.

The proof of  (\ref{SA}) is a straightforward calculation which takes into account that $ \det M_h = 1 $. \end{proof}

\begin{exercise} Establish (\ref{SA}) by a "coordinate-free" matrix computation. \end{exercise}

\begin{corollary} There exists a $ c_0 \ge 0 $ such that for all $ z \in \C_+ $
\[ \Re \frac{ 1 + S ( z ) A ( z ) }{ 1 - S ( z ) A ( z ) } = \frac{ \Im z }\pi \int \left| \frac{ E_h ( t ) }{ E ( t ) } \right|^2 \frac {\d t }{\( t - \Re z \)^2 + \( \Im z \)^2 } + c_0 \Im z . \]
\end{corollary}

By the criterion of theorem \ref{isometry} we infer that either $ c_0 = 0 $ and then $ {\mathcal H} ( E_h ) $ is a subspace of $ L^2 \( \R ,  \left| E \right|^{ -2 }  \) $, or $ c_0 \ne 0 $, and then there exists an $ \alpha \ne 0 $ such that $ E_h + e^{ i \alpha } E_h^* \in {\mathcal H} ( E_h ) $ and $ {\mathcal X} \colon = {\mathcal H} ( E_h ) \ominus \mathcal L\{ E_h + e^{ i \alpha } E_h^* \} $ is a subspace of $ L^2 \( \R ,  \left| E \right|^{ -2 } \) $.  

The embedding on the the real line just established has to be supplemented with a global estimate in the plane to infer that $ {\mathcal H} ( E_h ) $ (or $ \mathcal X $) is a subset (hence, a subspace) in $ {\mathcal H} ( E ) $. The required estimate is (ii) in the following lemma the proof of which is to be found  in section \ref{facts}.

\begin{lemma}\la{deBrchar} Let $ W $ be an HB function. An entire function $ f \in {\mathcal H} ( W ) $ iff the following two conditions are satisfied,
\[ \begin{array}{cc} (i) & f \in L^2 \( \R , \left| W \right|^{ -2 } \) , \cr
(ii) &  | f ( z ) | \le C_f \sqrt{ \frac{\left|  W ( z ) \right|^2 - \left|  W ( \overline z ) \right|^2}{ \Im z }} . \end{array} \] \end{lemma}

Let $ f \in  {\mathcal H} ( E_h ) $, then it satisfies (ii) with $ W = E_h $. The expression in the r. h. s. of (ii) with $ W = E_h $ is estimated above by the same expression with $ W = E $ by (\ref{mh}) in view of the following lemma.   

\begin{lemma} Let $ A = \Bigl( \! \begin{array}{c} A_+ \cr A_- \end{array} \! \Bigr) \in \C^2 $ be such that $ \Im \( A_+ \overline A_- \) \ge 0 $, and $ N $ be such that $ \frac1i ( N^* J N - J ) \ge 0 $, $ B = N A $. Then \[ \left| B_+ + i B_- \right|^2 - \left| B_+ - i B_- \right|^2 \ge \left| A_+ + i A_- \right|^2 - \left|^2 A_+ - i A_- \right|^2 . \]    
\end{lemma}

\begin{proof} \[ \Im \( B_+ \overline B_- \) = \frac1{2i} \langle J B , B \rangle = \frac 1{2i} \langle N^* J N A , A \rangle \ge \frac 1{2i} \langle JA , A \rangle = \Im \( A_+ \overline A_- \) . \] \end{proof}

Thus, if $ c_0 = 0 $, then $ {\mathcal H} ( E_h ) $ is a subspace of $ \HE $, if $ c_0 \ne 0 $ then $ \mathcal X $ is a subspace of $ \HE $. Notice that in the latter case $ \mathcal X $ is a de Branges space (see remark after the proof of theorem \ref{isometry}).

\begin{theorem}\la{isomembed}
Let $ \HW $ be a regular de Branges space, $ \HG $ a de Branges space, $ W $ and $ G $ have no real zeroes. Assume that $ \HG \subset \HW $, and $ \len f \rin_{ \HG } = \len f \rin_{ \HW } $ for all $ f \in \HG $. Then $ \HG $ is regular. \end{theorem}  

Now, $ E $ does not have real zeroes by assumption, and $ E_h $ by (\ref{mh}). We conclude from the theorem that $ E_h $ is regular. This is trivial in the case $ c_0 = 0 $, and follows immediately from the regularity of $ \mathcal X $  since the latter is a subspace in $ {\mathcal H} ( E_h ) $ if $ c \ne 0 $.

Let us apply theorem \ref{restore} to functions $ E $ and $ E_h $ of proposition \ref{altern}. Denote the corresponding matrices by $ M $ and $ N_h $, respectively. The matrices $ M $ and $ \wt M \colon = M_h N_h $ are entire functions having the same first column, $ \Theta $. As they both have determinant one, their respective second columns, $ \Phi $ and $ \wt \Phi $, satisfy $ \wt \Phi = \Phi + c \Theta $ where $ c $ is a scalar function of $ \lambda $. Since $ \Theta_+ $ and $ \Theta_- $ have no common zeroes, $ c $ is an entire function. We will show that $ c $ is zero.

First, we show that $ c $ is linear. Indeed, 
\[ c = \frac{ {\wt \Phi}_-}{\Theta_- } - \frac{\Phi_-}{\Theta_-} . \]  
According to (iii) of theorem \ref{restore} $ \Phi_-/\Theta_- $ is a Herglotz function. Now, $ {\wt \Phi }_-/\Theta_- $ is a Herglotz function as well because $ A \colon = \wt M ( z ) $ has the property that $ \frac 1i \( A^* J A - J \) \ge 0 $ for all $ z \in \C_+ $, for it is a product of two matrices, $ M_h $ and $ N_h $, each satisfying this property. Indeed, $ M_h $ is a monodromy matrix of a canonical system, hence (\ref{Jcontractive}) holds for it, and for $ N_h $ the property is (iii) of theorem \ref{restore}. That the product of two matrices, $ B $ and $ C $, satisfying the property, also satisfies it, is immediate from the following line
\[  \frac 1i \( \(BC\)^* J BC - J \) = \frac 1i C^* ( B^* J B - J ) C + \frac 1i ( C^* J C - J ) . \]
It remains to notice that by exercise \ref{columnrow} the property implies $ \Im \( A_{22} / A_{ 21 } \) \ge 0 $. Thus, $ c $ is entire and is a difference of two Herglotz functions. The representation theorem for Herglotz functions implies that $ c ( \lambda ) = a + b \lambda $ for some real $ a , b $. 

Notice that $ a = 0 $ on account of  $ N_h ( 0 ) = I $, and we have thus proved that in the assumptions of proposition \ref{altern} 
\be\la{lingrow} M ( z ) \( \begin{array} {cc} 1 & b z \cr 0 & 1 \end{array} \! \) = M_h ( z ) N_h ( z ) . \ee

Consider the ratio of $ 12 $- and $ 11 $-entries in this equality, with the following notation for columns of $ M_ h $, $ M_h = [ \Theta_h , \Phi_h ] $,  
\bequnan \frac { \Phi_+ }{ \Theta_+ } + b z = \frac{ \Theta_h^+ N_{ h, 12 } + \Phi_h^+ N_{ h , 22 }}{\Theta_h^+ G_h^+ + \Phi_h^+ G_h^- } = \frac{ N_{ h, 12 }}{ G_h^+ }  \( \frac{ \frac{ N_{ h, 22 }}{ N_{ h, 12 }} - \frac{G_h^-}{ G_h^+ }}{ \frac{\Theta_h^+}{ \Phi_h^+ } + \frac{ G_h^-}{ G_h^+ }} + 1 \) = \\ \begin{array}{ccc} = \frac 1{ \( G_h^+ \)^2 } & \displaystyle{\underbrace{ \frac1{ \frac{\Theta_h^+}{ \Phi_h^+ } + \frac{ G_h^-}{ G_h^+ }} } } & + \frac{ N_{ h, 12 }}{ G_h^+ } . \cr & (I) & \end{array} \eequnan

Consider the limit $ z = i t $, $ t \to + \infty $. By (iii) of theorem \ref{restore} the l. h. s. is $ b z ( 1 + o ( 1 ) ) $ and $ N_{ h, 12 }/ G_h^+ = o ( t ) $. Then, $| G_h^+ ( i t ) |/t $ is bounded away from $ 0 $. This is obvious from the the canonical product for $ G_h^+ $, for $ \left| G_h^+ ( i t ) \right|^2 = C \prod_{ G_h^+ ( \tau ) = 0 } \( 1 + t^2 /\tau^2 \) $ (notice that $ G_h^+ $ does have at least one zero since it is a real part of an HB function - see proposition \ref{altern}). It follows that if $ b \ne 0 $, then the denominator in (I) should be $ o ( t^{ -2 } ) $. The imaginary parts of both terms in the denominator are of the same sign (negative) by exercise \ref{columnrow}, hence
\[ \Im \frac{ G_h^- ( it ) }{ G_h^+ (it ) } = o \( \frac 1{t^2} \) ,\]
On the other hand, $ - G^h_- / G^h_+ $ is a Herglotz function, therefore the asymptotics obtained holds iff this function is zero, a contradiction since $ E_h ( 0 ) = 1$. It follows that $ b = 0 $, and $ M = M_h N_h $. Recalling that $ \operatorname{tr} J \dot M_h ( 0 ) = h $ by lemma \ref{trace} since $ M_h $ is a monodromy matrix of a canonical system on an interval of length $h $, we find
\[ \operatorname{tr} J \dot M ( 0 ) = h + \operatorname{tr}  J {\dot{N}}_h ( 0 ) . \]
The second term in the r. h. s. is positive by corollary \ref{trpos}, therefore $ h < \operatorname{tr} J \dot M ( 0 ) $. This means that $ E_h  \equiv 1 $ for all $ h \ge \operatorname{tr} J \dot M ( 0 )  $, and the second alternative mentioned after proposition \ref{altern} is ruled out. We formulate the result,

\begin{theorem}\la{inversolution}
Let $ E $ be a regular HB function having no real zeroes, $ E ( 0 ) = 1 $, and let $ \Theta = \frac 12 \( \begin{array}{c} E + E^* \cr \frac 1i ( E - E^* ) \end{array} \) $. Then there exists a canonical system $ ( \cH , L ) $, $ L < \infty $, such that $ \Theta = \Theta ( L , \cdot ) $.
\end{theorem}

Let us comment on this theorem. Since $ E_h \equiv 1 $ for $ h = \operatorname{tr} J \dot M ( 0 ) $, the canonical system solving the inverse problem is $ ( F^\prime , L ) $ with $ \ell =   \operatorname{tr} J \dot M ( 0 ) $, $ L = \ell $. The length, $ L  $, of this system is given by an explicit formula in terms of the function $ E $ (see (\ref{phi-}), (\ref{phi+}))
\be\la{lengthL} L =  \operatorname{tr} J \dot M ( 0 ) = \dot{\Phi}_+ ( 0 ) - \dot{\Theta}_- ( 0 ) = \frac 1\pi \len \frac{ \Theta_+ - 1 }\lambda \rin_{\HE}^2 -   \dot{\Theta}_- ( 0 )  \ee
with $ \Theta_\pm $ defined by (\ref{ThetaE}). This system satisfies the compatibility condition (L). Indeed, if it didn't, the system $ ( \tilde \cH , \tilde L ) $, $ \tilde H ( x ) = F^\prime ( x + L - \tilde L ) $, obtained by truncation of the singular interval at the left end, would satisfy the compatibility condition and have the same de Branges function $ E $. Comparing (\ref{LthroughE}) and the above formula for $ L $ we see that $ \tilde L = L $.

\section{A formula for type} \la{finreg}

Let $ E $ be a regular HB function. Then it is of finite exponential type, and let $ p = \limsup_{ |z| \to \infty } \ln | E ( z ) |/ |z| $ be its type. Notice that it is enough to consider $ z \to \infty $ along the positive imaginary axis.
\be\la{exptypeE} p = \limsup_{ y \to + \infty } \ln | E ( iy ) |/ y . \ee

Indeed, by the regularity of $ E $, for $ z \in \C_+ $ we have $ E ( z ) = e^{ ihz } f ( z ) $, for an $ h \le 0 $ and an outer function $ f $. It follows that for $ 0 < \theta < \pi $
\[ \ln | E ( R e^{ i\theta } ) | = - h R \sin \theta + o ( R ) , \; R \to +\infty . \]
By the HB property of $ E $, \[ \ln | E ( R e^{ i\theta } ) | \le h R \sin \theta + o ( R ) , \; R \to +\infty , \] 
for $ -\pi < \theta < 0 $. Applying the Fragmen-Lindel\"of theorem in sectors $ - \pi/4 < \theta < \pi/2 $, $ -3\pi/4 < \theta < -\pi/4 $, $ \pi/2 < \theta < 5\pi/4 $ we find that for any $ \von > 0 $ there exists a $ C_\von $ such that\footnote{This reasoning is sometimes called a hall-of-mirrors argument.} $ | E ( z ) | \le C_\von e^{( \von - h ) | z | } $, that is, the exponential type of $ E $ is not greater than $ -h $, and hence equals to it which is (\ref{exptypeE}). 

Let $ M ( \lambda ) = [ \Theta ( \lambda ) , \Phi ( \lambda ) ] $ be a monodromy matrix corresponding to $ E $ in the sense of theorem \ref{restore}.

\begin{lemma}\la{exptypeelements} All matrix elements of $ M $ are entire functions of the same exponential type equal to $ p $ given by (\ref{exptypeE}).
\end{lemma}

\begin{proof} Since $ \Theta_\pm $ are linear combinations of $ E $ and $ E^* $, both have exponential type $ \le p $. Then, 
\[ \limsup_{ y \to + \infty } y^{ -1 } \ln | \Theta_+ ( iy ) | = \limsup_{ y \to + \infty } y^{ -1 } \ln | \Theta_- ( iy ) | \] 
because $ \Theta_+ / \Theta_- $ is a Herglotz function in $ \C_+ $ and thus for large $ y $ we have  $ C y^{ -1 } \le \left| \Theta_+ ( iy ) / \Theta_- ( iy ) \right| \le C y $. Since $ E = \Theta_+ + i \Theta_- $ it follows that these upper limits coincide with $ p $ and hence the types of $ \Theta_\pm $ are equal to $ p $. 

Proceeding to $ \Phi $ notice first that $ \Phi_\pm $ are of finite exponential type by the Krein theorem since $ \pm \Phi_- / \Theta_- $ is a Herglotz function in $ \C_\pm $ by theorem \ref{restore} and exercise \ref{columnrow}, and therefore $ \Phi_- $ is of bounded type in $ \C_\pm $ because $ \Theta_- $ is. By the same token, we conclude that for any $ \von > 0 $ $ | \Phi_-  ( z ) / \Theta_- ( z ) | = O ( | z | ) $ for $ z \to \infty $ along any ray $ z = R e^{ i \theta } $, $ \theta \ne 0 , \pi $, hence $ | \Phi_-  ( z ) | = O \( e^{ ( \von - h ) |z| } \)  $. Applying the Fragmen-Lindel\"of theorem repeatedly as was done above gives that the exponential type of $ \Phi_- $ is not greater than $ h $. That the type of $ \Phi_- $ cannot be less than $ p $ follows again from the Herglotz property, this time of $ - \Theta_- / \Phi_- $, for if it were a similar application of the Fragmen-Lindel\"of theorem would imply that the exponential type of $ \Theta_- $ is also less than $ p $. Thus, $ \Phi_- $ has the type $ p $, and so does $ \Phi_+ $ by the way of an analogous argument.
\end{proof}

\begin{theorem}\la{exptypeM} 
Let $ ( \cH , L ) $ be a canonical system corresponding to the function $ E $ in the sense of theorem \ref{inversolution}. Then the exponential type of $ E $, $ p $, satisfies
\[ p = \int_0^L \sqrt{ \det \cH ( x ) } \, \d x . \]
\end{theorem}

The proof of this theorem is based on two elementary observations brought about by the following lemma, of which (\textit{i}) eventually leads to $ p \ge (\textrm{the integral in RHS}) $, and (\textit{ii}) to the reverse inequality.

\begin{lemma}\la{2by2} Let $ A $ be a $ 2 \times 2 $-matrix.

(i) If $ A $ has real entries and $ A \ge 0 $, then $ A \ge \pm i \sqrt{ \det A } J $.

(ii) If $ \operatorname{tr} A = 0 $, then \[ \sqrt{ | \det A | } = \inf_{ \Omega \colon \det \Omega \ne 0 } \len \Omega A \Omega^{ -1 } \rin . \]
\end{lemma}

\begin{proof}
(\textit{i}) The diagonal of $ A \pm i \sqrt{ \det A } J $ is nonnegative (it is the same as the one of $ A $), and by direct computation $ \det ( A \pm i \sqrt{ \det A } J ) = 2 \det A $, which is nonnegative. 

(\textit{ii}) The spectrum of $ A $ is $ \pm \sqrt{ - \det A } $ which implies (\textit{ii}) immediately in the case $ \det A \ne 0 $. If $ \det A = 0 $ then $ A $ wlog can be assumed to have the form $ \( \! \begin{array}{cc} 0 & c \cr 0 & 0 \end{array} \! \) $. On taking $ \Omega = \mathrm{daig} ( a, a^{ -1 } ) $ and sending $ a $ to $ 0 $ we get (\textit{ii}). 
\end{proof}

\medskip

\noindent \textit{proof of theorem \ref{exptypeM}.}
I. $ p \ge \int_0^L \sqrt{ \det \cH } $. WLog one can assume that the system $ ( \cH , L ) $ satisfies compatibility condition (L). Let $ M ( x , z ) $ be the monodromy matrix of the system $ ( \cH , L ) $. Then (\ref{identM}) is satisfied with $ \lambda = \overline z $, and by lemma \ref{2by2} (\textit{i}) applied to $ A = \cH ( x ) $ we have,
\bequnan \frac 1i \( M^* ( x , \lambda ) J M ( x , z ) - J \) = 2 \Im z \int_0^x M^* ( t , z ) \cH ( x ) M ( t , z ) \d t  \ge \\ 2\Im z \int_0^x \sqrt{ \det \cH ( t ) } \( \frac 1i M^* ( t , z ) J M ( t , z ) \) \d t . \eequnan
This operator inequality implies the corresponding numerical inequality of upper leftmost entries of the matrices (whenever it is deemed convenient, $z$ is suppressed in the arguments of functions),
\[  \xi ( x ) \ge 2\Im z \int_0^x \xi \sqrt{ \det \cH } , \]
\[ \xi \colon = \frac 1i \( M^* J M \)_{ 11 } . \]
Denoting the integral in the RHS as $ h ( x ) $, we rewrite it as
\[ h^\prime ( x ) \ge 2\Im z \sqrt{ \det \cH ( x ) } h ( x ) . \] 
 As $ h ( x ) \ge 0 $ and monotone non-decreasing, one solves this inequality dividing it by $ h ( x ) $ and integrating, to obtain
\[ h ( x ) \ge e^{ 2 \Im z \int_y^x \sqrt{ \det \cH } } h ( y ) , \; 0 \le y \le x \le L . \]
It follows that 
\[ \xi ( x ) \ge 2 ( \Im z )  e^{ 2 \Im z \int_y^x \sqrt{ \det \cH } } h ( y )  . \]
Taking $ x = L $ and recalling that $ \xi ( L , z ) = \frac{ \left| E ( z ) \right|^2 -  \left| E ( \overline z ) \right|^2 }2 $ by the definition, we find for $ \Im z > 1 $ 
\[ \left| E ( z ) \right|^2 \ge 4 h ( y , z ) e^{ 2 \Im z \int_y^L \sqrt{ \det \cH } } . \] 
In view of (\ref{exptypeE}) this means that $ p \ge  \int_y^L \sqrt{ \det \cH } $ whenever $ h ( y , i \von ) $ is separated from $ 0 $ as $ \von \to + \infty $. It remains to check when the latter takes place. For $ z = i \von $, $ \von < 0 $, and all $ N > 0 $ we have
\[ \xi ( t , z ) = 2 \Im \( \frac{ \Theta_+ ( t , z ) }{ \Theta_- ( t ,z ) } \) \left| \Theta_- ( t , z ) \right|^2 \ge C_t \mu_t ( -N , N ) \frac\von{ \von^2 + N^2 } \von^2 . \]
Here $ C_t \ne 0 $, $ \mu_t $ is the measure in the integral representation of the Herglotz function $ \Theta_+ ( t , \cdot ) / \Theta_- ( t , \cdot ) $, and we took into account that $ | \Theta_- | $ grows at least linearly along the imaginary axis, which is obvious from the canonical product for $ \Theta_- $. Adjusting $ N $ properly, we find that $ \xi ( x , i \von ) \to + \infty $ as $ \von \to + \infty $, hence $ h ( y , i \von )  \to + \infty $ unless $ \det \cH ( t ) = 0 $ for a. e. $ t \in ( 0 , y ) $ by the Fatou theorem. The inequality $ p \ge \int_0^L \sqrt{ \det \cH } $ follows.

II. $ p \le \int_0^L \sqrt{ \det \cH }  $. Let $ p ( a ) $ be the exponential type of the function $ E_a ( z ) \colon = \Theta_+ ( a , z ) + i \Theta_- ( a , z ) $. The idea of the argument is to prove that $  p ( a ) $ is an absolutely continuous function, and $ | p^\prime | \le \sqrt{ \det \cH } $ a. e.

Let $ \Omega $ be an invertible $ 2 \times 2 $-matrix, $ M ( x , y, \lambda ) $ the monodromy matrix for the system $ ( \cH , L ) $ from $ x $ to $ y $, $ 0 \le x \le y \le L $. The integral equation for the monodromy matrix can be written in the form 
\[  \Omega M ( x, y,  \lambda ) = \Omega - \lambda \int_x^y \( \Omega J \cH ( t ) \Omega^{ -1 } \) \Omega M ( x, t , \lambda )  \d t . \]
By the Gronwall lemma
\[ \len \Omega M ( x , y , \lambda ) \rin \le \| \Omega \| e^{ | \lambda | \int_x^y \len  \Omega J \cH \Omega^{ -1 } \rin } . \]
On account of the multiplicative property of the monodromy matrix, 
\[ M ( y , \lambda ) = M ( x , y , \lambda ) M ( x , \lambda ) , \] 
it follows that the function $ p $ satisfies (recall that by lemma \ref{exptypeelements} the exponential type of all entries of the matrix $ M ( a , \lambda ) $ equals to $ p ( a ) $)
\[ | p ( y ) - p (  x ) | \le \int_x^y \len  \Omega J \cH \Omega^{ -1 } \rin . \]
Here we have taken into account that $ \det M ( x , y , \lambda ) = 1 $, hence the exponential type of entries of the matrix $ M^{ -1 } ( x, y , \lambda ) $ is also estimated above by $ \int_x^y \len  \Omega J \cH \Omega^{ -1 } \rin $. 

The integrand in the estimate obtained is bounded, hence $ p $ is an a. c. (even Lipschitz) function, and
\[ | p^\prime ( y ) | \le \len  \Omega J \cH ( y ) \Omega^{ -1 } \rin . \] 
Applying lemma \ref{2by2} (\textit{ii}) with $ A = J \cH ( y ) $ we find that $ | p^\prime ( y ) | \le \sqrt{ \det \cH ( y ) } $ for a. e. $ y \in ( 0 , L ) $, which implies the required inequality. \hfill $ \Box $ 

\section{Uniqueness}

All the results on uniqueness in inverse problem for canonical systems are based on the following 

\begin{theorem}\la{lattice}
Let $ \d \mu $ be a ($\sigma$-finite) measure on $ \R $, $ X , Y $ be subspaces in $ L^2 ( \R , \d \mu ) $, which are de Branges spaces, that is, there exist HB functions $ E_X $, $ E_Y $ such that $ X = \cH ( E_X ) $, $ Y = \cH ( E_Y ) $ as 
Hilbert spaces. Suppose that $ E_X $ and $ E_Y $ have no real zeroes, $ E_X / E_Y $ is a function of bounded type in $ \C_+ $. Then either $ X \subset Y $, or $ Y \subset X $.
\end{theorem}

An important partial case of this theorem is when $ X$ and $ Y $ are subspaces of a de Branges space, $ \cH ( E ) $, that is, when $ \d \mu (t) = \left| E ( t ) \right|^{-2 } \d t $. In this case $ E_X / E $ and $ E_Y / E $ are functions of bounded type in $ \C_+ $ hence their ratio is of bounded type too and 

\begin{corollary}
If two subspaces of a de Branges space are de Branges spaces themselves then one of them contains the other, provided that the corresponding HB functions have no real zeroes.
\end{corollary} 

The proof of theorem \ref{lattice} is given in Appendix III.  

\begin{theorem}\la{uniqueness}
Let $ ( \cH , L ) $, $ ( \cH_1 , L_1 ) $, $ L $, $ L_1 < \infty $, be two canonical systems satisfying the compatibility condition (L), $ \Theta $ and $ \Theta_1 $ be solutions of the respective canonical systems (\ref{can}) satisfying $ Y ( 0 ) = \( 1 , 0 \)^T $. If $ \Theta ( L , \lambda ) = \Theta_1 ( L_1 , \lambda ) $ for all $ \lambda $, then $ L = L_1 $, $ \cH = \cH_1 $.
\end{theorem}

\begin{proof}
I. $ L =L_1 $. According to corollary \ref{trace} we have $ L =  \dot\Phi_+ ( L , 0 ) - \dot\Theta_- ( L , 0 ) $. Then (see (\ref{phi+})), 
\[ \dot\Phi_+ ( L , 0  ) = G_0^+ = \frac 1\pi \len \frac {\Theta^+_L - 1 }t \rin_{ \HE }^2  \] 
with $ E = \Theta^+_L + i \Theta^-_L $. Thus, for a given canonical system, $ ( \cH , S ) $, the length $ S $ is given by an explicit formula in terms of the function $ \Theta ( S , \lambda ) $. It follows that $ L = L_1 $.

II. $ \cH_1 =\cH $. Define  for $ t \le L $, \[ E_t = \Theta_+ ( t, \cdot ) + i \Theta_- ( t , \cdot ),  \; E_t^1 = \Theta_+^1 ( t , \cdot ) + i \Theta_-^1 ( t , \cdot ) , \]  $ {\mathcal H}_t = {\mathcal H} ( E_t ) $, $ {\mathcal H}^1_t = {\mathcal H} ( E^1_t ) $. If $ t $ is not an interior point of a singular interval for $ \cH $ and $ \tau $  is not an interior point of a singular interval for $ \cH_1 $ then $ {\mathcal H}_t $ and $ {\mathcal H}_\tau^1 $ are subspaces of $ \HE $, for they are ranges of subspaces in $ H $ ($ H_1 $) comprised of functions supported on $ [0 , t ]$ ($ [ 0 , \tau ] $, resp.) under the map $ \cU $ ($ \cU_1 $) defined in (\ref
{spmap}). By the same reason the families $ \{ {\mathcal H}_t \} $ and $ \{ {\mathcal H}^1_t \} $ are increasing in $ t $. According to the  corollary to theorem \ref{lattice} either $ {\mathcal H}_t \subset {\mathcal H}_\tau^1 $, or $ {\mathcal H}_\tau^1 \subset {\mathcal H}_t  $. Denote by $ \mathcal M $, $ {\mathcal M}_1 $ the complement  in $ [ 0 , L ] $ of the set of interior points of the singular intervals for $ \cH $, $ \cH_1 $, resp. Given a $ t \in {\mathcal M}_1 $ define $ x ( t ) = \sup \{ x  \in \mathcal M \colon \, {\mathcal H}_x \subset {\mathcal H}_t^1 \}  $. Then $ x ( t ) \in \mathcal M $ since $ \mathcal M $ is closed, and $ {\mathcal H}_{x (t) } = {\mathcal H}^1_t $. Indeed,  $ {\mathcal H}_{x (t) } \subset {\mathcal H}^1_t $ trivially  if the set $  \{ x  \in \mathcal M \colon \, {\mathcal H}_x \subset {\mathcal H}_t^1 \}  $ has a maximal element, and if it does not then $  {\mathcal H}_{x (t) } = \bigvee_{ x < x ( t ) }  {\mathcal H}_x $ (see  theorem \ref{spmaptheo}),  hence $ {\mathcal H}_{x (t) } \subset {\mathcal H}^1_t $ in any case. Conversely, the subspace $ {\mathcal N } = \bigcap_{ x > x ( t ) } {\mathcal H}_x $ contains  $ {\mathcal H}^1_t $, for $ {\mathcal H}^1_t  \subset {\mathcal H}_x $ as $ x > x ( t ) $. If $  {\mathcal H}_{ x ( t ) } = \mathcal N $ we are done. If $  {\mathcal H}_{ x ( t ) } $ is a proper subspace of $ \mathcal N $ then $ x_* \colon = \inf \{ {\mathcal M}\cap ( x( t ), L ] \} \in \mathcal M $, $ \left[ x ( t ) , x_* \right] $ is a singular interval, $ {\mathcal N } = {\mathcal H}_{ x_* } $, hence $  \mathcal N \ominus  {\mathcal H}_{ x ( t ) } $ is one-dimensional. It follows that $ {\mathcal H}^1_t $ coincides with either $ {\mathcal H}_{ x( t ) } $, or $ {\mathcal N} $.  In the former case we are done, and in the latter case $ x ( t ) $ is not the supremum ( $ x_* > x ( t ) $), which is impossible. Thus, $ {\mathcal H}_{x (t) } = {\mathcal H}^1_t $.

The spaces $ {\mathcal H}_{x (t) } $ and $ {\mathcal H}^1_t $, $ t \in \cM_1 $, coincide as de Branges spaces, hence they have the same reproducing kernels. The reproducing kernel at $ z = 0 $ is $ \Theta_- ( z ) / z $ (up to an absolute constant, see (\ref{reprokernel})), hence $ \Theta_- ( x ( t ) , \lambda ) = \Theta_-^1 ( t , \lambda ) $. Considering the reproducing kernels at some nonzero point, we now conclude that $ \Theta_+ ( x ( t ) , \lambda ) = \Theta_+^1 (  t , \lambda ) + a_t \Theta_-^1 (  t , \lambda )  $ for some real constant $ a_t $, that is, for each $ t \in \cM_1 $ there exists a real $ a_t $ such that
\be\la{at} \Theta ( x ( t ) , \lambda ) = \( \begin{array}{cc} 1 & a_t \cr 0  & 1 \end{array} \) \Theta^1 ( t , \lambda ) . \ee
Our goal is to show that $ a_t =0 $, then $ x ( t ) \equiv t $ and $ \cH_1 = \cH $ will follow easily. 

By definition, $ x ( t ) $ is a monotone increasing function on $ \cM_1 $. Let us extend it to the whole of $ [ 0 , L ] $ by linearity ($ \cM_1 $ is closed!). The resulting map is a monotone bijection of $ [ 0 , L ] $ onto itself. The range of $ \cM _1 $ under this bijection coincides with $ \cM $. The identity 
\be\la{the-} \Theta_- ( x ( t ) , \lambda ) = \Theta_-^1 ( t , \lambda ) \ee 
then holds for all $ t \in [ 0 , L ] $, and thus 
\be\la{theta-} \ddot\Theta_- ( x ( t ) , 0 ) = \ddot\Theta_-^1 ( t , 0 ) . \ee

We are going to differentiate this identity in $ x $ with respect to the canonical system. To this end, we need absolute continuity of $ x ( t ) $. 

\begin{lemma}
Let  $ \Omega \subset [ 0 , L ] $. If $ | \Omega | = 0 $ then $ | x ( \Omega ) | = |x^{ -1 } ( \Omega ) | = 0 $.
\end{lemma}

\begin{proof}
Suppose by contradiction that $ |\Omega | = 0 $ but $ | x^{ -1 } ( \Omega ) | > 0 $. WLog, one can assume that $ x^{ -1 } ( \Omega ) $ is closed and contained in $ \cM_1 $. Then $ [ 0 , L ] \setminus x^{ -1 } ( \Omega ) $ is a union of mutually disjoint open intervals, $ \omega_j $. Since the ends of these intervals lie in $ \cM_1 $, the set of Fourier images in the sense of  Theorem \ref{spmaptheo} of functions from $ H_1 $, the Hilbert space of the system $ ( \cH_1 , L ) $,  supported on $ \cup_j \omega_j $ forms a subspace in $ \HE $ coinciding with the set of Fourier images of functions from $ H $, the Hilbert space of the system $ ( \cH , L ) $, supported on $ x ( \cup_j \omega_j ) $. The latter set of elements coincides with the whole of $ H $ since the complement of $ x ( \cup_j \omega_j ) $, -- the set $ \Omega $ -- has zero measure. It follows that the set of functions from $ \cH_1 $ supported on $ \cup_j \omega_j $ coincides with $ H_1 $, which means that $ x^{ -1 } ( \Omega ) = [ 0 , L ] \setminus \cup_j \omega_j $ is a set of zero measure, a contradiction. The other equality ($ | x ( \Omega ) |= 0 $)  follows by switching $ \cH $ and $ \cH_1 $ in the argument.
\end{proof}

Now, both sides in (\ref{theta-}) are a. c. functions of $ t \in [ 0 , L ] $, and
\be\la{Thetaderive} x^\prime ( t )  \ddot\Theta_-^\prime ( x ( t ) , 0 ) = \frac{d \ddot\Theta_-^1 ( t , 0 )}{ d t } \ee
for a. e. $ t \in [0, L ] $.  

\begin{lemma}
Let $ Y ( x , \lambda ) $ be a vector solution to (\ref{can}) satisfying $ Y ( 0 , \lambda ) = \begin{pmatrix} 1 \cr 0 \end{pmatrix} $. Then 
\be\la{21} \ddot{Y}_-^\prime ( x , 0 ) = 2 \left. \( \dot{Y}_-^\prime \dot{Y}_+ - \dot{Y}_+^\prime \dot{Y}_- \)\right|_{ \lambda = 0 } \ee
for a. e. $ x \in [ 0 , L ] $.
\end{lemma}

\begin{proof} Let $ M $ be the fundamental solution of (\ref{can}). Differentiating $ J M^\prime = \lambda \cH M $ twice in $ \lambda $ at $ \lambda = 0 $ we obtain $ J \ddot{M}^\prime = 2 \cH \dot M = 2 J \dot{M}^\prime \dot M $. In the second equality we took (\ref{derM}) into account. Now, if we cancel $ J $ in both sides, (\ref{21}) is the left lower entry in the resulting matrix identity. \end{proof}                                                                                    

Notice that if the compatibility condition (L) is satisfied then $ \dot{\Theta}_- ( x , 0 ) = \int_0^x \cH_{ 11 } $ (see  (\ref{derM})) is positive and monotone nondecreasing for $ x > 0 $ hence the r. h. s. in (\ref{21}) can be rewritten as 
\[ -2 \dot{\Theta}_-^2 \( \frac{ \dot{\Theta}_+ }{ \dot{\Theta}_- } \)^\prime . \]

\medskip
 
\textit{End of proof of theorem \ref{uniqueness}.} On substituting this into both sides of (\ref{Thetaderive}) and canceling out the squares of $ \dot\Theta_- ( x ( t ) ) $ and $ \dot\Theta_-^1 ( t ) $ (see (\ref{the-}); we have just explained why they do not vanish under the assumptions of the theorem) we find
\[  x^\prime ( t ) \left. \( \frac{ \dot{\Theta}_+ }{ \dot{\Theta}_-} \)^\prime \right|_{ x = x (t) } = \( \frac{ \dot{\Theta}_+^1 }{ \dot{\Theta}_-^1 } \)^\prime \]
or 
\[ \frac{d}{dt} \(  \frac{ \dot{\Theta}_+ ( x(t) ) }{ \dot{\Theta}_- ( x(t) ) } - \frac{ \dot{\Theta}_+ ( t ) }{ \dot{\Theta}_- ( t ) } \) = 0. \]
The obtained equality holds a. e. in $ t $ and the function in braces is a. c. The function is therefore constant on $ ( 0 , L ] $. On taking $ t =L $ we find that the constant is $ 0 $. Notice that for $ t \in \cM_1 $ the value of this function is precisely $ a_t $ from (\ref{at}). Thus, $ \Theta ( x(t) , \lambda ) = \Theta^1 ( t , \lambda ) $ for all $ t \in \cM_1 $, and hence for all $ t \in [ 0 , L ] $ since $ \Theta ( x ( t )) $ and $ \Theta^1 $ are linear functions on a singular interval for $ \cH_1 $ coinciding at the ends of the interval (for the ends lie in $ \cM_1 $). On account of explicit formulae (\ref{phi-}) and (\ref{phi+}) for the second special solution $ \Phi $, the monodromy matrices, $ M ( x , \lambda ) $, $ M_1 ( t , \lambda ) $, for the canonical systems $ ( \cH , L ) $, $ ( \cH_1 , L ) $, resp., satisfy
\[ M ( x ( t ) , \lambda ) = M_1 ( t , \lambda ) \]
for all $ t \in [ 0 , L ] $, $ \lambda \in \C $. On differentiating this equality in $ t $ and $ \lambda $ at $ \lambda = 0 $ we obtain on account of (\ref{derM}), 
\[ x^\prime ( t ) \cH ( x ( t ) ) = \cH_1 ( t ) \]
for a. e. $ t \in [  0 , L ] $. Taking trace in this equality we conclude that $ x^\prime ( t ) = 1 $ a. e. in $ t $. Since $ x ( t ) $ is a. c., it follows that $ x( t ) = t $ and thus $ \cH_1 = \cH $. 
\end{proof}

This uniqueness theorem allows to analyze the solution of the inverse problem in the regular case given in section \ref{Invreg}. As noticed after theorem \ref{inversolution}, the canonical system solving the problem is $ ( F^\prime , L ) $ where $ L $ is determined from $ E $ by the formula (\ref{lengthL}), and $ F $ is determined by a convergent subsequence of $ F_N $'s.  As it stands, this solution is not constructive since the convergent subsequence is chosen by a compactness argument. Recall, however, that the system $ ( F^\prime , L ) $ satisfies the compatibility condition (L) \textit{regardless of the choice of a convergent subsequence}.  Theorem \ref{uniqueness} now implies that the limiting functions $ F $ corresponding to different choices of the subsequence may only differ by an additive constant, and hence they coincide, since $ F_N ( 0 ) = 0 $. It follows that the initial sequence, $ \{  F_N \}_{  N \in \N } $,  itself converges in $ C ( 0 , L ) $. Thus the solution is actually constructive and one can formulate

\medskip

\textbf{Algorithm for solving the inverse problem in the regular case.} \textit{Let $ E $ be an HB function satisfying the assumptions of  section \ref{Invreg}. Let $ t_j $, $ j \ge 0 $, $ t_0 = 0 $, be the set of zeroes of $ \Theta_- $ ordered by $ | t_j | \le | t_{ j+1 } | $, 
\[ \Theta_-^N ( z ) = \dot\Theta_- ( 0 ) z \prod^{N-1} \( 1 - \frac z{t_j} \) , \]
\[  \Theta_+^N ( z ) =  \( \sum_{j = 0 }^N \frac{\Theta_+( t_j ) }{ \dot\Theta_- ( t_j ) }\frac 1{ z - t_j } + a + b z \) \Theta_-^N ( z ) ,\]
$ a $ and $ b $ being the constants in the linear term in the Herglotz representation of $ \Theta_+ / \Theta_- $. Then $ E_N = \Theta_N^+ + i \Theta_N^- $ is an HB polynomial satisfying the assumptions of  section \ref{invpol}. Let $ ( \cH_N , L_N ) $ be the corresponding canonical system constructed in sections \ref{invpol1} and \ref{invpol2}. Let 
\[ L = \frac 1{\pi} \len \frac{\Theta_+ - 1 }\lambda \rin_{\HE}^2 -   \dot{ \Theta}_- ( 0 ) . \] 
Define 
\[  \wt\cH_N ( x ) = \left\{ \begin{array}{cl} \cH_N \( x - \max\{ 0 , L - L_N \}  \) , &  x \ge \max\{ 0 , L - L_N \} \cr
 \( \! \begin{array}{cc} 0 & 0 \cr 0 & 1 \end{array} \! \), 
 & 0 \le x \le \max\{ 0 , L - L_N \},
 \end{array} \right. \]
\[  F_N ( x ) = \int_{ l_N - h }^{ l_N - h + x } \wt\cH_N ( s ) \d s , \; l_N = \max \{ L , L_N \}.  \] 
Then $ F_N $ converges in $ C ( 0 , L ) $ to a monotone non-decreasing function, $ F $. The canonical system $ ( F^\prime , L ) $ solves the inverse problem and satisfies the compatibility condition (L).} 

\medskip

\subsection{Borg Theorem}\la{borgtheorem}

Let $ q $ being a real-valued function from $ L^1 ( 0 , 1 ) $, $ h \in \R $ , $  h_1 \in \R \cup \{ \infty \} $. Denote by $ \sigma ( q ; h , h_1 ) $ the set of $ \lambda \in \R $ such that the equation 
\be\la{Schr} - u^{\prime \prime } + q u = \lambda u \ee 
has a non-trivial solution satisfying $ u^\prime ( 0 ) = h u ( 0 ) $, $ u^\prime ( 1 ) = h_1 u ( 1 ) $,  with the usual meaning for $ h_1 = \infty $.

\begin{theorem}\la{abstrBorg}
If $ \sigma ( q ; h, h_1 ) = \sigma ( \tq ; \th , \th_1 ) $, $ \sigma ( q ; h , h_2 ) = \sigma ( \tq ; \th , \th_2 ) $ for some $ q $, $ \tq $, $ h $, $ h_{ 1,2} $, $ \th $, $ \th_{1, 2} $, $ h_1 \ne h_2 $, then $ \tq = q $, $ \th = h $, $ \th_1 = h_1 $, $ \th_2 = h_2 $.
\end{theorem}

Before proving the Borg theorem, let us dwell on details of the correspondence between the Schr\"odinger operator and the canonical systems described in section \ref{examples}. According to what has been said there, $ y $ is a solution to (\ref{Schr}) iff $ y = Y^\circ \cdot Y $ where $ Y $ is a solution to the canonical system (\ref{can}) with the Hamiltonian $ \cH ( x ) = \llangle \cdot , Y^\circ (x) \rrangle Y^\circ (x) $, $ Y^\circ = \begin{pmatrix} \yOone \cr \yOtwo \end{pmatrix} $, $ {\stackrel{\lower.5mm\hbox{$ \scriptstyle\circ $}}y}_{1,2} $ being real solutions to (\ref{Schr}) with $ \lambda = 0 $ satisfying $ W \{ \yOone , \yOtwo \} = -1 $, $ \cdot $ standing for the real (without conjugation of the second member) scalar product. We would like to choose $ {\stackrel{\lower.5mm\hbox{$ \scriptstyle\circ $}}y}_{1,2} $ so that the eigenfunctions of the canonical system corresponding to the Schrodinger boundary conditions above  be $ \Theta $, that is,  the condition corresponding to $ y^\prime ( 0 ) = h y ( 0 ) $ be $ Y_- ( 0 ) = 0 $. We have 
\[ \begin{array}{rcl} y^\prime = & \underbrace{Y^\circ \cdot Y^\prime } & \hskip-6mm + \, {Y^\circ}^\prime \cdot Y = {Y^\circ}^\prime \cdot Y , \cr & \| &  \cr & \lambda J Y^\circ \cdot \cH Y & = \lambda ( Y^\circ \cdot Y )( J Y^\circ \cdot Y^\circ ) = 0 \cr  \end{array} \]
and thus 
\be\la{foyrbc} y^\prime - h y = \( {Y^\circ}^\prime - h Y^\circ \) \cdot Y . \ee
It follows that given a solution $ y $ to (\ref{Schr}) satisfying $ u^\prime ( 0 ) = h u ( 0 ) $ the following are equivalent,

\textit{(i)} $ Y_- ( 0 ) = 0 $,

\textit{(ii)} $ \yOone^{\lower3pt\hbox{$ \scriptstyle \prime $}}( 0 ) - h \yOone ( 0 ) = 0 $.

From now on, the assumption in \textit{(ii)} is supposed to be satisfied. 

Consider now the right end of the interval. Let $ y $ be a solution to (\ref{Schr}), $ h_r \in \R \cup \{ \infty \} $. Then, 
\begin{eqnarray} y^\prime ( 1 ) = h_r y ( 1 ) \mathop{\Leftrightarrow}^{ (\ref{foyrbc}) }  \left. \( {Y^\circ}^\prime - h_r Y^\circ \) \cdot Y \right|_1 = 0 \Leftrightarrow Y ( 1) = \left. (\mbox{const}) J  ( {Y^\circ }^\prime - h_r Y^\circ ) \right|_1  \Leftrightarrow \nonumber \\  
  Y_+ ( 1 ) = H Y_- ( 1 ) ,  H \colon = \displaystyle{\frac{\yOtwo^{\lower3pt\hbox{$ \scriptstyle \prime $}}  ( 1 ) - h_r \yOtwo ( 1 ) }{ - \yOone^{\lower3pt\hbox{$ \scriptstyle \prime $}} ( 1 ) + h_r \yOone ( 1 ) }},  \la{Hhr} \end{eqnarray} 
the notation convention being that $ H = \infty $  means $ Y_- ( 1 ) = 0 $.

The  fractional linear relation determining $ H $ through $ h_r $ is nondegenerate since $ {\stackrel{\lower.5mm\hbox{$ \scriptstyle\circ $}}y}_{1,2} $ are linear independent, hence (\ref{Hhr}) establishes a one-to-one correspondence between $ h_r \in \R \cup \{ \infty \} $ and $  H \in \R \cup \{ \infty \} $, and $ y = Y^\circ \cdot Y $ is a one-to-one correspondence between eigenfunctions corresponding to $ \sigma ( q ; h , h_r ) $ and solutions to (\ref{can}) satisfying $ Y_- ( 0 ) = 0 $, $ Y_+ ( 1 ) = H Y_- ( 1 ) $.

Let the assumption of the Borg theorem be satisfied. Let $ \cH $, $ \wt{ \cH } $ be the Hamiltonians of the canonical systems corresponding to $ q $ and $ \tq $, resp. In what follows the solutions of the Schr\"odinger equation with the potential $ \tq $, of the canonical system with the Hamiltonian $ \wt { \cH } $, and other objects corresponding to $ \tq $ analogous to the ones defined above for $ q $ are denoted by the same symbols supplied with tilde.  We assume that the condition \textit{(ii)} and a similar one for the solutions of the tilded problem are satisfied. Define $ E ( \lambda ) = \Theta_+ ( 1 , 
\lambda ) + i \Theta_- ( 1 , \lambda ) $, $ \tilde E ( \lambda ) = \tilde \Theta_+ ( 1 , \lambda ) + i \tilde \Theta_-  ( 1 , \lambda ) $. Our immediate goal is to show that $ E = \tilde E $. To this end one has to take a transform of these functions (a linear combination of $ E $ and $ E^* $, $ \tilde E $ and $ \tilde E^* $) for which the two coinciding spectra are, respectively,  the zeroes of their real and imaginary parts. By the following lemma, such a transform exists under a  "no zero mode" condition.

\begin{lemma} Given a $ q $, $ h $, $ h_{ 1 ,2 } $ subject to the conditions of the Borg theorem, assume that $ 0 \notin \sigma( q ; h , h_2 ) $. Then one can choose the solutions $ y^\circ_{ 1,2 } $ satisfying $ \yOone^{\lower3pt\hbox{$ \scriptstyle \prime $}}( 0 ) - h \yOone ( 0 ) = 0 $ and $ W \{ \yOone , \yOtwo \} = -1 $ in such a way that $ H_1 $ is either $ 1 $, or $ -1 $, and $ H_2 = 0 $, $ H_{ 1,2 } $ being defined by (\ref{Hhr}) with $ h_r = h_{ 1, 2 } $, respectively. \end{lemma}

Indeed, there are four conditions (Wronskian, \textit{(ii)} for $ \yOone $, $ H_2 = 0 $, $H_1 = \pm 1$) for four coefficients parametrizing all possible choices of $ {\stackrel{\lower.5mm\hbox{$ \scriptstyle\circ $}}y}_{1,2} $, hence the assertion of the lemma looks plausible.  

\begin{proof} Fix arbitrary solutions $ \xi $ and $ \eta $ to (\ref{Schr}) with $ \lambda = 0  $ satisfying $ W (\xi  , \eta ) = - 1 $ and $ \xi^\prime ( 0 ) = h \xi ( 0 ) $. Let $ \yOone = c \xi  $, 
\[ \yOtwo = \frac 1c \( \eta + \frac{ \eta^\prime ( 1 ) - h_2 \eta ( 1 ) }{ - \xi^\prime ( 1 ) + h_2 \xi ( 1 ) }\xi \) . \] 
The condition about zero not being in the spectrum means that the denominator in the coefficient in front of $ \xi $ does not vanish. It is obvious that $ \yOtwo^{\lower3pt\hbox{$ \scriptstyle \prime $}} ( 1) = h_2 \yOtwo ( 1 ) $ ( the linear combination of $ \eta $ and $ \xi $ above is produced to this effect), hence $ H_2 = 0$. Then
\be\la{H1} H_1 = \frac 1{ c^2 } \(  \frac{ \eta^\prime ( 1 ) - h_1 \eta ( 1 ) }{ - \xi^\prime ( 1 ) + h_1 \xi ( 1 ) } -  \frac{ \eta^\prime ( 1 ) - h_2 \eta ( 1 ) }{ - \xi^\prime ( 1 ) + h_2 \xi ( 1 ) } \)  \ee 
It is obvious that, depending on the sign of the expression in parentheses, one can choose $ c $ so that $ H_1 $ is either $ 1$, or $ -1 $.  \end{proof}

From now on, we assume that $ 0 \notin \sigma( q ; h , h_2 ) $, $ 0 \notin \sigma( \tq ; \th , \th_2 ) $, and the solutions $ {\stackrel{\lower.5mm\hbox{$ \scriptstyle\circ $}}y}_{1,2} $, $ {\stackrel{\lower.5mm\hbox{$ \scriptstyle\circ $}}{\tilde y}}_{1,2} $ are chosen so that the assertion of the lemma is satisfied. This does not incur any loss of generality, since one can change potentials by adding the same constant to make these assumptions hold. We also assume for convenience that the enumeration of spectra is chosen so that $ H_1 = 1 $ (notice that  switching $ h_1 $ and $ h_2 $ in (\ref{H1}) changes sign of $ H_1 $). 

Let $ \Theta_\pm = \Theta_\pm ( 1 , \cdot ) $, $ \tilde \Theta_\pm = \tilde \Theta_\pm ( 1 , \cdot ) $. The functions $ \Theta_+ $  and $ \tilde \Theta_+ $ have finite exponential type and their sets of zeroes coincide by the assumption of the Borg theorem (the said sets are $ \sigma ( q ; h , h_2 )$ and $  \sigma ( \tq ; \th , \th_2 ) $, respectively), hence there exist real $ A $ and $ b $ such that $ \Theta_+ ( z ) / \tilde \Theta_+ ( z )  = A e^{ b z } $.  The regularity of $ E $ and $ \tilde E $ implies that $ \( 1 + t^2 \)^{ -1 } \ln | E ( t ) |  \in L^1 ( \R ) $, $ \( 1 + t^2 \)^{ -1 } \ln | \tilde E ( t ) |  \in L^1 ( \R ) $, from whence $ b = 0 $. Also, $ A = 1 $ because $ \Theta_+ ( 0 ) =  \tilde \Theta_+ ( 0 ) = 1 $. Thus, $ \tilde  \Theta_+ = \Theta_+ $. Similarly, $ \sigma ( q ; h , h_1 )$, the set of zeroes of the function $ \Theta_+ - \Theta_- $, coincides with $  \sigma ( \tq ; \th , \th_1 ) $, the set of zeroes of the function $ \tilde \Theta_+ - {\tilde H}_1 \tilde\Theta_- $, hence the ratio of these functions is $ e^{ d z } $ for some real $ d $. Then $ d = 0 $ by the same reason as above, so  $ \Theta_+ - \Theta_- = \tilde \Theta_+ - {\tilde H}_1 \tilde\Theta_- $ and thus either $ \tilde H_1 =1 $, $ \tilde \Theta_- = \Theta_- $, or $ \tilde H_1 =-1 $, $ \tilde \Theta_- = - \Theta_-  $. The latter is impossible, for instance, because both $ \Theta_+ / \Theta_- $ and $ \tilde \Theta_+/ \tilde\Theta_- $ are Herglotz functions. It follows that $ \tilde \Theta_- = \Theta_- $, and the equality $ \tilde E = E $ is established.

One can now apply the uniqueness theorem \ref{uniqueness}. According to it, the lengths of the intervals on which the corresponding normalized canonical systems are defined coincide, that is,
\[ L \equiv \int_0^1 \( \yOonequ + \yOtwoqu \) \d x =   \int_0^1 \( \yOonequti + \yOtwoquti  \) \d x , \]
and the normalized Hamiltonians coincide. Let  
\[ \xi ( s ) = \int_0^s \( \yOonequ + \yOtwoqu \) \d x , \;\; \tilde \xi ( s ) = \int_0^s \( \yOonequti + \yOtwoquti \) \d x . \] 
The values of the normalized Hamiltonians at a point $ t \in ( 0 , L ) $ are orthogonal projections on vectors $ Y^\circ \( \xi^{ -1} (t) \) $ and $ {\widetilde Y}^\circ ( {\tilde \xi}^{ -1 } ( t ) ) $, resp. It follows that these vectors are proportional, that is, 
\be\la{ratio} \frac{\yOone \( \xi^{ -1} (t) \)}{\yOtwo \( \xi^{ -1} (t) \)}  = \frac{\yOoneti \bigl( {\tilde \xi}^{ -1 } ( t )  \bigr)}{\yOtwoti \bigl( {\tilde \xi}^{ -1 } ( t ) \bigr) } \ee
for all $ t \in [ 0 , L ] $. Notice that, as nontrivial solutions of 2nd order linear ODE, $ \yOtwo $ and $ \yOtwoti $ are $ C^1 $ and have at most finitely many zeroes.   
Differentiating the last identity and taking into account that  $ W \{ y^\circ_1 , y^\circ_2 \} = W \{  \yOoneti , \yOtwoti \} = -1 $  we find,
\[  \left. \frac 1{\yOonequ + \yOtwoqu } \, \frac1{\yOtwoqu } \right|_{ x = \xi^{ -1 } ( t ) }  = \left. \frac 1{\yOonequti + \yOtwoquti } \, \frac1{\yOtwoquti } \right|_{ x = \tilde \xi^{ -1} ( t )  }  . \]
Denoting by $ c( t ) $ the equal expressions in (\ref{ratio}), one can rewrite this as follows, 
\[  \frac 1{ c^2 (t) + 1 } \frac1{\stackrel{\lower1mm\hbox{$ \scriptstyle\circ $}}{ \widetilde y  }_2^{\lower4pt\hbox{$ \scriptstyle 4 $}} \( \xi^{ -1 } ( t ) \) }   =  \frac 1{ c^2 (t) + 1 } \frac1{\stackrel{\lower1mm\hbox{$ \scriptstyle\circ $}}{ \widetilde y  }_2^{\lower4pt\hbox{$ \scriptstyle 4 $}} \bigl( \widetilde \xi^{ -1 } ( t ) \bigr) }  . \]
Thus, $ \bigl| \yOtwo \( \xi^{ -1 } (t ) \) \bigr|  = \Bigl| \yOtwoti \bigl( \tilde \xi^{ -1 } ( t ) \bigr) \Bigr| $, and using (\ref{ratio}) again, we find that $ \bigl| \yOone \( \xi^{ -1 } (t ) \) \bigr|  = \Bigl| \yOoneti \bigl( \tilde \xi^{ -1 } ( t ) \bigr) \Bigr| $. Summing the squares of these we obtain, $ \xi^\prime \circ \xi^{ -1 } = \tilde \xi^\prime  \circ \tilde \xi^{ -1 } $ from whence $ \tilde \xi = \xi $ and thus $ \yOonequ = \yOonequti $, $ \yOtwoqu = \yOtwoquti $.  It follows that $ q = \tq $ because 
\[ q  =  \frac12 \frac{\( \yOonequ + \yOtwoqu \)^{\prime\prime} - 2 \( {\yOone^\prime}^2 + {\yOtwo^\prime}^2 \) }{ \yOonequ + \yOtwoqu  } .  \] 

Another way to prove that $ q =\tq $ is to notice that either $ \yOtwo  = \yOtwoti $, or $ \yOtwo  = - \yOtwoti $ because all zeroes of $ \yOtwo $ and $ \yOtwoti $ are simple, hence each of these functions changes sign at each zero, and then to use the equation for these solutions. By the same token $ h_2 = \th_2 $. A similar consideration for $ \yOone $ and $ \yOoneti $ shows that $ h = \th $. The Borg theorem is proved. 

\section{Direct spectral theory - singular case}\la{directsin}

Let $ L = \infty $, and $ ( \cH , \infty ) $ be a canonical system obeying the compatibility condition (L) and such that the ray $ ( b , \infty ) $ is not a singular interval for any $ b > 0 $.  Notice that for any $ z \notin \R $  the equation (\ref{can}) has a solution belonging to $ H $. Indeed, consider a compactly supported $ h \in H $. Since the operator $ D $ constructed in the proof of theorem \ref{operatorLinf} is selfadjoint, for any $ z \notin \R $ the equation \[ \( D - z \) f = h \] has a solution $ f \in \cD $. To the right of the support of $ h $ this solution is an a. c. function satisfying (\ref{can}) with $ Y = f $. Picking up an arbitrary point, $ x_0 $ to the right of the support of $ h $ and solving the Cauchy problem for the equation (\ref{can}) on $ ( 0 , x_0 ) $ with the data $ Y ( x_0 ) = f ( x_0 ) $ we obtain a solution to  (\ref{can}) on the whole semiaxis $ ( 0 , + \infty ) $. This solution is unique up to a scalar factor, and $ \Theta ( \cdot , z ) $ is not the solution, for otherwise $ \Theta ( x , z ) $ would be in the domain of $ D $ and hence an eigenfunction of a selfadjoint operator corresponding to a  non-real eigenvalue. It follows that for all $ z \notin \R $ there exists a unique solution, $ U \in H $, to (\ref{can}) of the form
\be\la{mthetaphi} U ( x , z ) = \Phi ( x , z ) - m ( z ) \Theta ( x , z ) , \; x \in [ 0 , + \infty ) ,  \ee 
where $ m ( z ) $ is a complex number. 
Define 
\[ m_X ( z ) = \frac{ \Phi_- ( X , z )}{ \Theta_- ( X , z ) } , \; X \ge 0 . \] 
Since $ U ( x ) $ is small at infinity in $ x $, and $ \Theta ( x ) $ is not, one can expect that $ m_X $ approaches $ m $ as $ X \to \infty $. A precise assertion is the content of the following

\begin{theorem}\la{mN}
\[  m_X\displaystyle{\mathop{\Longrightarrow}_{ X \to \infty } } m  \] uniformly on compacts in $ \C \setminus \R $.
\end{theorem} 

\begin{proof} Let us consider the fractional linear map 
\[ G_X \colon w \mapsto  \frac{\overline{ \Phi_- ( X , z ) } w - \overline{ \Phi_+ ( X , z ) }}{ - \overline{ \Theta_- ( X , z ) } w + \overline{ \Theta_+ ( X , z ) }  } . \] 
The matrix of this map, 
\[ \begin{pmatrix}  \overline{ \Phi_- ( X , z ) } & - \overline{ \Phi_+ ( X , z ) } \cr - \overline{ \Theta_- ( X , z ) } & \overline{ \Theta_+ ( X , z ) }  \end{pmatrix} , \] 
to be denoted $ G_X $ as well, coincides with the complex conjugate of $ {  M ( X ) }^{ -1 } $, $ M ( X ) $ being the monodromy matrix of the system under consideration for the interval $ ( 0 , X ) $. The matrix $ G_X $ is $J $-contractive by lemma \ref{columnrow} since $ M( X ) $ is. It follows that the image of $ \C_+ $ under the map $ G_X $ is a disc in $ \C_+ $  (see exercise \ref{JcontrC} below) which we denote by $ D_X $. We claim that the discs $ D_X $ are nesting, $ D_X \subset D_{ X^\prime }$ for $ X > X^\prime $. Indeed, the multiplicative property of monodromy matrices implies that $ G_X = G_{ X^\prime }\Omega $ where $\Omega $ is the inverse complex conjugate of the monodromy matrix for the interval $ ( X^\prime , X ) $. Then $ \Omega $ is $ J $-contractive, hence the corresponding fractional linear map brings $ \C_+ $ into itself. It follows that $ D_X  = G_{ X^\prime } ( \textrm{a subset in } \C_+ ) \subset D_{ X^\prime } $.

The point $ - \overline {\Phi_- ( X , z ) / \Theta_- ( X , z ) } $ lies on the boundary of the disc $ D_X $, for it coincides with the image of infinity under $ G_X $. Let us show that the radii of $ D_X $ converge to zero locally uniformly in $ z \in \C \setminus \R $, and $ \cap_{ X > 0 } D_X = \{ - \overline{m ( z )} \} $. The theorem will be proven then. The diameter of $ D_X $ equals to the distance between the image of infinity and that of the point $ w = \Re \( \Theta_+ ( X, z )/ \Theta_- ( X , z ) \) $, for the line $ \Re w = \Re \( \Theta_+ ( X , z ) / \Theta_- ( X , z ) \) $ is orthogonal to the real axis hence $ G_X $ sends it into a line containing a diameter of $ D_X $. Thus, for any $ z \in \C_+ $ 
\bequnan \operatorname{diam} D_X = \left| - \frac{ \Phi_- ( X , z)}{ \Theta_- ( X , z ) } - \frac{ \Phi_- ( X , z ) \Re \frac{ \Theta_+ ( X , z ) }{ \Theta_- ( X , z ) } -  \Phi_+ ( X , z ) }{ - \Theta_- ( X , z ) \Re \frac{ \Theta_+ ( X , z ) }{ \Theta_- ( X , z ) } + \Theta_+ ( X , z ) } \right| = \\ \left| \frac 1{ \Theta_-^2 ( X , z ) \Im  \frac{ \Theta_+ ( X , z ) }{ \Theta_- ( X , z ) } } \right| =  \frac 1{\Im \( \Theta_+ ( X , z ) \overline{\Theta_- ( X , z ) } \) }} \displaystyle{ \mathop{=}^{(\ref{11})} \frac 1{ \Im z \int_0^X \Theta^* ( t , z ) \cH ( t ) \Theta ( t , z ) \d t } . \eequnan 
The denominator in the r. h. s. tends to infinity as $ X \to \infty $, for $ \Theta ( \cdot , z ) \notin H $, and it is easy to see from the continuity of $ \Theta $ in $ z $ that it tends to infinity locally uniformly in $ z $. Thus the diameter of $ D_X $ converges to zero, $ \cap_{ X > 0 } D_X $ consists of a single point, say, $ m_* ( z ) \in \C_+ $, and 
\[ - \overline { \( \frac{\Phi_- ( X , z )}{ \Theta_- ( X , z ) } \) } \longrightarrow m_* ( z ) . \]   
In a similar way, $ - \overline{ \Phi_+ ( X )/ \Theta_+ ( X ) } $ lies on the boundary of $ D_X $ (it is the image of $ w = 0 $ under $ G_X $) and therefore converges to $ m_* ( z ) $ as well. Combining these together, we get
\[ \Phi ( X ) = \Theta ( X ) ( - \overline{m_* ( z ) } + o ( 1 ) ) , \; X \to + \infty . \]
It follows that \[ U  = \Phi - m ( z ) \Theta  = \Theta ( - \overline{m_* ( z ) } - m ( z ) + o ( 1 ) ) . \] Since $ U $ belongs to $ H $ and $ \Theta $ doesn't, $ m ( z ) =  - \overline{m_* ( z ) } $, as required. \end{proof}

\begin{exercise}\la{JcontrC}
Let $ S $ be a $ 2 \times 2 $-matrix, $ \det S = 1 $. Then the following are equivalent,

(i) $ \frac 1i \( S^* J S - J \) \ge 0 $,

(ii) The fractional linear map corresponding to $ S $ maps $ \C_+ $ into itself.
\end{exercise}

\medskip

\textit{Hint.} {\small (ii) $ \Longrightarrow $ (i). Any fractional linear map of $ \C_+ $ into itself is a composition of maps $ z \mapsto z + a $, $ a \in \C_+ $, and automorphisms of $ \C_+ $ (fractional linear maps defined by elements of $ SL ( 2 , \R ) $). (i) is trivially satisifed for the latter (the equality in fact holds) and  is easily verified for the former.

(i) $ \Longrightarrow $ (ii). An easy calculation shows that 
\[ \Im S ( z ) = \frac 1{2i} Q ( z ) \llangle S^* J S \begin{pmatrix} z \cr 1 \end{pmatrix} ,  \begin{pmatrix} z \cr 1 \end{pmatrix} \rrangle \] 
with a $ Q ( z ) \ge 0 $. This implies that the real line is transformed by $ S $ into a cricle or line lying in $ \overline{\C_+} $.} 

\medskip

According to exercise \ref{columnrow}, $ m_X $ are Herglotz functions for each $ X $, hence so is $ m $. The function $ m $ is called the \textit{Weyl-Titchmarsh function} for the canonical system $ ( \cH , \infty ) $. Let $ \mu_X $, $ \mu $ be the measures in the integral representation of the functions $ m_X $, $ m $, respectively, that is,
\[ \frac{ \Phi_- ( X , \lambda ) }{ \Theta_- ( X , \lambda ) } = \frac 1\pi \int \frac{ d \mu_X ( t ) }{ t - \lambda } + \textrm{ a linear function} , \] 
and similar for $ \mu $.  For a compactly supported $ f \in H $ let us define the entire function $ \cU f $ by the formula (\ref{spmap}). The set of such $ f $ is obviously dense in $ H $.

\begin{theorem}\la{sptheoremsing}
The mapping $ \cU $ defines an isomorphism between $ H $ and $ L^2 ( \R  , d \mu ) $, that is, $ \cU f \in  L^2 ( \R  , d \mu ) $, 
 \be\la{UN} \len f \rin_H = \len \cU f \rin_{ L^2 ( \R  , d \mu ) } \ee
for any compactly supported $ f \in H $, and $ \operatorname{Ran} \cU  $ is dense in $ L^2 ( \R  , d \mu ) $. 

Let $ A $ be the operator of multiplication by the independent variable in  $ L^2 ( \R  , d \mu ) $. Then 
\[  D = \( \overline \cU \)^* A \overline \cU . \]
\end{theorem}

\begin{proof} Fix  an arbitrary compactly supported $ f \in H $. Then for all $ X $ greater than the right end of the support of $ f $ we have
\bequnan \len \cU f \rin^2_{ L^2 ( \R  , d \mu_X ) } = - \pi \sum_{ \Theta_- ( X, t ) = 0 } \frac{ \Phi_- ( X,  t ) }{ \dot{\Theta}_- ( X ,  t ) } \left| ( \cU f ) ( t ) \right|^2 = \\ - \sum \frac 1{ \Theta_+ ( X, t ) \dot{\Theta}_- ( X, t ) } \left| \llangle f , \Theta_t \rrangle_H  \right|^2 = \sum \frac  {\left| \llangle f ,  \Theta_t \rrangle_H  \right|^2}{ \len  \Theta_t \rin^2 } = \dots \eequnan
Here the first equality is the definition of the l.h.s., the second one is implied by the monodromy matrix $ M ( X , t ) $ having unit determinant, the third one holds on account of the identity $ - \Theta_+ ( X , t ) \dot{\Theta}_- ( X,  t )  =  \len \Theta ( \cdot , t ) \rin^2_{ L^2 ( ( 0 , X) ; \cH ( x ) )}  $ valid whenever $ \Theta_- ( X, t ) = 0 $. The latter identity is immediate from  (\ref{identM}) upon differentiating the upper leftmost entry in $ z $ and letting $ z = \lambda = t $  (we have already used it when solving the inverse problem for a measure supported on finitely many points). If the operator of the canonical system on $ ( 0 , X ) $ with boundary condition $ f_- ( 0 ) = f_- ( X ) =0 $ is selfadjoint (that is, the compatibility condition (R) is satisfied at the point $ x = X $), then the set $ \{ \Theta_t \}_{ \Theta_- ( t ) = 0 } $ is the orthogonal basis of eigenfunctions of this operator (see proposition \ref{Dalpha}), hence the equality can be continued,
\[ \dots =  \len f \rin_H^2 . \]

Notice that the condition (R) is satisfied for a sequence of $ X $ monotone increasing to infinity on account of the assumption that the ray $ x > a $, $ a > 0 $, cannot be a singular interval. From now on let us assume that $ X $ belongs to such a sequence. We have just shown that the restriction of $ \cU $ to the subspace in $ H $ of functions supported on $ ( 0 , X ) $ is an isomorphism onto $ L^2 ( \R, d\mu_X ) $. Our goal is to pass to the limit $ X \to \infty $ in this assertion. First, $ d \mu_X \displaystyle{ \mathop{\longrightarrow}^{*-weak}} d \mu $ by theorem \ref{mN}. This and the isomorphism just established imply that $ \cU f \in L^2 ( \R , d \mu ) $ for any compactly supported $ f \in H $, and
\[ \len \cU f \rin_{  L^2 ( \R, d\mu ) } \le \limsup_{ X \to \infty }   \len \cU f \rin_{ L^2 ( \R  , d \mu_X ) } =  \len f \rin_H . \]
Let us assume additionally that $ f \in \cD $, $ \cD $ being the domain of the operator $ D $, and let $ g = ( D + i ) f $. Then 
\[ \sup_X \int_{ |\lambda| > A } \left| \cU f \right|^2 \d \mu_X \mathop{\longrightarrow}_{ A \to \infty } 0 . \]
Indeed, $ g $ is compactly supported together with $ f $ and $ (\cU g ) ( \lambda ) = ( \lambda + i ) ( \cU f ) ( \lambda ) $ which is easily verified by straightforward computation, hence the integral in the l. h. s.  equals to 
\[ \int_{ |\lambda| > A } \frac{\left| ( \cU g )( \lambda ) \right|^2}{ 1 + \lambda^2 } \d \mu_X ( \lambda ) \le \frac { \len \cU g \rin^2_{ L^2 ( \R , d \mu_X ) } }{ A^2 } . \]
From this it follows that $ \len \cU f \rin_{ L^2 ( \R , d \mu_X ) } \to \len \cU f \rin_{ L^2 ( \R , d \mu ) } $, and (\ref{UN}) is established for all $ f \in \wt \cD $, $ \wt \cD $ being the set of compactly supported vectors in $ \cD $. In the proof of theorem \ref{operatorLinf} we have shown that $ \wt \cD $ is dense in $ H $, hence (\ref{UN}) is proved.

 Integrating by parts we find that $ \cU D f = A \cU f $ for any $ f \in \wt \cD $. By remark \ref{essentialself} the set $ \wt \cD $ is an essential domain of the operator $ D $, hence this identity admits closure and holds true of any $ f \in \cD $ if we replace $ \cU $ with its closure. 

It remains to show that $ \operatorname{Ran} \cU  $ is dense in $ L^2 ( \R  , d \mu ) $. For $ z_0 \notin \R $ one can rewrite this as follows,
\[  \( A - z_0 \)^{-1 } \cU h = \cU \( D - z_0 \)^{-1 } h , \]
\[ h \colon = ( D - z_0 ) f . \] 
Essential selfadjointness of $ D $ on $ \wt \cD $ means that the set $ ( D - z_0 ) \wt\cD $ is dense in $ H $ for $ z_0 \notin \R $. Thus the last displayed equality holds for all $ h $ from a dense set. It follows that $ \wt H = \overline{ \operatorname{Ran} \cU } $ is an invariant subspace of  $ \( A - z_0 \)^{ -1 } $, and hence of the operator $ A $ itself. In turn, this implies that $ \wt H $ is the subspace of functions vanishing outside a subset of $ \R $. Assume that the orthogonal complement of $ \wt H $ is nontrivial, then $ \wt H $ is the subspace of functions vanishing $ \mu $-a.e. on a set, $ \cM $, of positive $ \mu $-measure. The set $ \cM $ is necessarily discrete since there are nonzero entire functions in $ \operatorname{Ran} \cU $, and therefore $ \mu \( \{ z_0 \} \) \ne 0 $ for some $ z_0 \in \cM $. Now, let $ N > 0 $ be a point which is not interior to a singular interval, and let $ f = \chi_N \Theta_{ z_0 } $, $ \chi_N $ being the indicator function of the interval $ ( 0 , N ) $. Then, $ f \in \cH $, and $ ( \cU f ) ( z _0 ) = \int_0^N \langle \cH  \Theta_{ z_0 } ,  \Theta_{ z_0 } \rangle \d x > 0 $ on account of the compatibility condition (L).  Thus, the assumption that $ \cM $ has positive $ \mu $-measure leads to a contradiction, and $ \wt H = L^2 ( \R  , d \mu ) $. 
\end{proof}

Neither the compatibility condition (L), nor the assumption of absence of a singular semiaxis are essential for the definition of the Weyl-Titchmarsh function, they were only imposed to be able to formulate the result in the form of a spectral theorem for a selfadjoint operator. 

\begin{lemma}\la{generalm} Suppose that  %
\be\la{nondegenerateanew}  \cH ( t )\not\equiv \( \! \begin{array}{cc} 0 & 0 \cr 0 & 1 \end{array} \! \) \ee
on $ ( 0 , \infty ) $. Then for any $ \lambda \notin \R $ there exists a unique complex $ m $ such that $ \Phi_\lambda - m \Theta_\lambda \in H $. 
\end{lemma}

\begin{proof} If $ ( A , + \infty ) $ is not a singular interval for any $ A \ge 0 $ then, by cutting off a singular interval $ ( 0 , A^\prime ) $, $ A^\prime > 0 $, we obtain a canonical system on $ ( A^\prime , \infty ) $ satisfying a compatibility condition at $ x = A^\prime $ analogous to $ (L) $. By theorem \ref{operatorLinf} there exists a selfadjoint operator corresponding to this  canonical system and the boundary condition $ f_- ( A ) = 0 $, and the argument proceeds verbatim as in the beginning of this section.  

Assume now that $ ( A , + \infty ) $ is a singular interval for some $ A \ge 0 $, $ \cH ( x ) = \langle \cdot , e \rangle e $ for an $ e \in \R^2 $, $ \| e \| = 1 $. For any solution, $ Y $, of (\ref{can}) we then have (see (\ref{monodrsing})) for $ x \ge A $,
\[ Y ( x , \lambda ) = Y ( A , \lambda ) - \lambda \langle Y ( A , \lambda ) , e \rangle ( x - A ) J e . \]
It follows that $ Y \in H $ iff $ \langle Y ( A , \lambda ) , e \rangle = 0 $. Since $ \Theta $ and $ \Phi $ form a basis in the space of solutions for the system, a linear combination of them satisfies the latter condition. Then, $ \langle \Theta ( A , \lambda ) , e \rangle \ne 0 $, for otherwise 
\[ z \int_0^A \langle \cH \Theta , \Theta \rangle =\int_0^A \langle J \Theta^\prime , \Theta \rangle  = \left. \langle J \Theta , \Theta \rangle \right|_0^A + \int_0^A \langle \Theta , J \Theta^\prime \rangle = \overline z \int_0^A \langle \Theta , \cH \Theta \rangle , \] 
the boundary term at $ x = A $ vanishing on account of $ e $ being a vector with real components, so $ \langle \Theta , \cH \Theta \rangle = 0 $ and $ \cH \Theta = 0 $ on $ ( 0 , A ) $ which implies by the system (\ref{can}) that $ \Theta ( x ) = \begin{pmatrix} 1 \cr 0 \end{pmatrix} $ for all $ x \in [ 0 , A ] $. In turn, on considering $ x = A $ this gives that $ e = \begin{pmatrix} 0 \cr 1 \end{pmatrix} $. Combining these we conclude that the only possibility for $ \cH $ is the one prohibited by the assumption (\ref{nondegenerateanew}), hence $ \langle \Theta ( A , \lambda ) , e \rangle  \ne 0 $ indeed, and the linear combination with unit coefficent at $ \Phi_\lambda $ exists and is unique.
\end{proof}

According to this lemma, the Weyl-Titchmarsh function is defined for any Hamiltonian satisfying the condition (\ref{nondegenerateanew}). It is easy to see that theorem \ref{mN} holds for this function. This rather trivial remark is required in solution of the inverse problem in the singular case. The point is that the Hamiltonian in that solution is defined via some approximation procedure, which can well lead to Hamiltonians having a singular semiaxis, even if the approximating ones do not. Theorem \ref{sptheoremsing} can also be formulated without additional restrictions, but that would require messy dealing with extensions of non-densely defined operators, which we prefer to avoid.

\section{Inverse problem - singular case}\la{invsing}

In the case of a canonical system on the semiaxis ($ L = \infty $), there is no function $ E ( L , \lambda ) $, for $ \Phi $ and $ \Theta $ generally do not  have boundary values as $ x \to \infty $, and the corresponding inverse problem does not make sense. The inverse problem with respect to spectral measure of the operator $ \cD_\alpha $ described in section \ref{fsupp}, however, does. This problem will be solved, in a sense, in this section. In the case when the measure is supported on finitely many points it was solved in section \ref{fsupp}. The argument here uses that solution via approximation.

Let $ d \mu $ be a measure such that $ \int_\R \( 1 + t^2 \)^{ -1 } \d \mu ( t ) $ is finite and $ \mu ( \{ 0 \} ) = 0 $. The latter condition is technical and will be removed later on. Define $ \mu_N $, $ 1 \le N <\infty $, to be an arbitrary sequence of measures, each being supported on finitely many points, such that 
\[ \d \mu_N \displaystyle{ \mathop{\longrightarrow}^{*-weak}} \d \mu , \] and   
\be\la{esttails}  
\sup_N \int_{ |t| > s } \frac{ \d \mu_N ( t ) }{  t^2 }  \mathop{\longrightarrow}_{ s \to \infty } 0 , \ee \[ \mu_N ( \{ 0 \}) \ne 0 . \]

It easy to see that such a sequence exists. Let $ ( \cH_N , L_N ) $ be the canonical system obtained by applying theorem \ref{finmes} to the measure $ \d \mu_N $. Notice that $ L_N \to \infty $. Indeed, $ \mu_N ( \{ 0 \} )  = - 1/ \dot\Theta^-_N ( 0 ) \to \mu ( \{ 0 \} ) = 0 $, hence $ - \dot\Theta^-_N ( 0 ) \to + \infty $.  Since $ L_N \ge  - \dot\Theta^-_N ( 0 ) $ by the assertion of lemma \ref{trace}, $ L_N \to \infty $.

Define $ G_N ( x ) = \int_0^x \cH_N ( t ) \d t $ for $ 0 < x \le L_N $. For $ x > L_N $ define $ G_N ( x ) $ to be an arbitrary function so that $ G_N $ is continuous on $ [ 0 , + \infty ) $. Then the family $ G_N $ is compact in $ C ( 0 , A ) $ for any $ A > 0 $. Applying a  diagonal process we obtain that there exists a subsequence, $ G_N $, and a function, $ G $, such that $ G_N $ converges to $ G $ uniformly on any compact $ ( 0 , A ) $. Since $ G_N $ are monotone increasing and absolutely continuous, so is $ G $. Define $ \cH ( x ) =   G^\prime ( x ) $, $ x \ge 0 $. The function $ \cH $ is defined for a. e. $ x \ge 0 $, ans satisfies $ \cH ( x ) \ge 0 $, $ \operatorname{tr} \cH ( x) = 1 $ for a. e. $ x $.

Let $ M_N ( x , z ) $, $ M ( x , z ) = [ \Theta ( x , z ) , \Phi ( x , z ) ] $ be the monodromy matrices for the systems $ ( \cH_N , L_N ) $,  $ ( \cH , \infty ) $, resp. First, we would like to show that $ M_N ( x , z ) \longrightarrow M ( x , z ) $ for each $ z \in \C $ and $ x > 0 $. Just as we did in the regular case (see (\ref{inte}) and the argument following it) one can establish using the integral equation for $ M_N $ that for any $ z \in \C $ there exists a subsequence, $ N_k $, of $ N $'s and a continuous function $ \tilde{ M } ( \cdot , z ) $ such that $ M_{ N_k } ( \cdot , z ) \Longrightarrow \tilde M ( \cdot , z )  $ in $ C ( 0 , A ) $ for all $ A > 0 $. The same argument based on integration by parts as in the regular problem shows that $ \tilde M ( \cdot , z ) $ satisfies the same integral equation as $ M ( \cdot , z ) $ from whence they coincide. If follows that the limit does not depend on the choice of the subsequence, and  $ M_N ( x , z ) \longrightarrow  M ( x , z )  $ for all $ z \in \C $. 

Let us first assume that $ \cH \not\equiv \( \! \begin{array}{cc} 0 & 0 \cr 0 & 1 \end{array} \! \) $ on $ \R_+ $. Then $ m $,  the Weyl-Titchmarsh function for the system $ ( \cH , \infty ) $, is defined (see the remark at the end of section \ref{directsin}). Let
\[ m_N ( x , z ) = \frac{ \Phi^-_N ( x , z ) }{ \Theta^-_N ( x , z ) }, \; x \le L_N , \]
\[  m_N ( z ) = m_N ( L_N , z ) .  \]
We have for $ x < L_N $, $ z \in \C_+ $,
\bequnan & &  | m_N ( z ) - m ( z ) | \le \\ & & \!\!\! \begin{array}{ccccc}  \underbrace{ | m_N ( z ) - m_N ( x , z ) |  } & + & \underbrace{\left| m_N ( x , z ) - \frac{ \Phi^- ( x , z ) }{ \Theta^- ( x , z ) }\right|} & + & \underbrace {\left| \frac{ \Phi^- ( x , z ) }{ \Theta^- ( x , z ) } - m ( z ) \right| } \cr (I) & & (II) & & (III) \end{array} . \eequnan
Then
\[ (I) \le \frac 1{ \Im \( \Theta_N^+ ( x,z ) \overline{\Theta_N^-  ( x, z )}  \) }, \; (III) \le \frac 1{ \Im \( \Theta^+ ( x, z ) \overline{\Theta^-  ( x, z) } \) } \]
by the nesting circles analysis in the proof of theorem \ref{mN}. The convergence of $ M_N ( x ) $ to $ M ( x ) $ implies that $ ( II) \to 0 $ as $ N \to \infty $,  and 
\[ \Im \( \Theta_N^+ ( x ) \overline{\Theta_N^-  ( x)}  \) \mathop{\longrightarrow}_{ N \to \infty } \Im \( \Theta^+ ( x ) \overline{\Theta^-  ( x) } \)  \mathop{=}^{(\ref{11})}  \Im z \int_0^x \Theta^*  \cH \Theta . \]
Combinig these estimates, on passing to the limit $ N \to \infty $ we find that for any $ z\in \C_+ $ and  $ x > 0 $
\[ \limsup_{ N \to \infty }  | m_N ( z ) - m ( z ) | \le \frac 2{ \Im z \int_0^x \Theta^*  \cH \Theta} \]
The r. h. s. vanishes as $ x \to \infty $ because $  \Theta ( \cdot , z ) \notin H $ according to lemma \ref{generalm}. It follows that $ m_N (z ) \to m ( z ) $. On the other hand,  the choice of $ \mu_N $ implies that for any $ z \in \C_+ $
\[ \Im z \int \frac{ \d \mu_N ( t )  }{ \( \Re z - t \)^2 + \( \Im z \)^2 }\mathop{\longrightarrow}_{ N \to \infty } \Im z \int \frac{ \d \mu ( t ) }{ \( \Re z - t \)^2 + \( \Im z \)^2 }. \] 
Notice that it is at this point we have used the technical condition (\ref{esttails}). Comparing these we conclude that $ \mu $ coincides with the measure in the Herglotz representation of the function $ m $, that is,
\[ \Im ( m( z )) =  \frac{\Im z}\pi \int \frac{ \d \mu ( t ) }{ \( \Re z - t \)^2 + \( \Im z \)^2 } + c \Im z , \; \Im z > 0  , \]
for a nonnegative constant $ c $. We have just conditionally proved the following

\begin{theorem}\la{invsingular} 
Let $ \mu $ be a measure on $ \R $, $ \int_\R \frac{ \d \mu ( t ) }{ 1 + t^2 } < \infty $ and $ \mu ( \{ 0 \} ) = 0 $. Then there exists a canonical system $ ( \cH , \infty ) $, $ \cH \not\equiv \( \! \begin{array}{cc} 0 & 0 \cr 0 & 1 \end{array} \! \) $, such that 
\[ \Im m ( z ) = \frac{ \Im z }\pi \int_\R \frac{ \d \mu ( t ) }{ \( \Re z - t \)^2 + \( \Im z \)^2 } + c \Im z , \; \Im z > 0 \] 
for a nonnegative constant $ c $. 

If the system satisfies compatibility condition (L) and no semiaxis is a singular interval for $ \cH $ then $ \mu $ coincides with the spectral measure of the corresponding selfadjoint operator in the sense given by theorem \ref{sptheoremsing}.
\end{theorem}

\begin{proof} The theorem will be proved established if we show that the condition (\ref{nondegenerateanew}) is satisfied for $ \cH $. Suppose by contradiction that $ \cH ( x ) = \( \! \begin{array}{cc} 0 & 0 \cr 0 & 1 \end{array} \! \) $ for a. e. $ x > 0 $. Consider the composition of the mapping $ G_x $  from the proof of theorem \ref{mN} corresponding to the Hamiltonian $ \cH_N $  and the transform $ \lambda \mapsto 1/\lambda $,
\[ w \mapsto  \frac{ - \overline{ \Theta_-^N ( x , z ) } w + \overline{ \Theta_+^N ( x , z ) }  }{\overline{ \Phi_-^N ( x , z ) } w - \overline{ \Phi_+^N ( x , z ) }} . \]
This map transforms $ \C_+ $ into the disc $ \tilde{D}_x^N $, $ \tilde{D}_x^N  = \{ \lambda^{ -1 } , \lambda \in D_x^N \} $ in the notation of the proof of theorem \ref{mN}. The discs $ \tilde{D}_x^N $ are nesting, as $ D_x $'s are, and a calculation from the proof of theorem \ref{mN} can be reproduced verbatim with $ \Theta $ and $ \Phi $ exchanged to give
\[ \mathrm{diam } \tilde{D}_x^N =  \frac 1{\Im \( \Phi_+^N ( x , z ) \overline{\Phi_-^N ( x , z ) } \) } \mathop{\longrightarrow}_{ N \to \infty }  \frac 1{\Im \( \Phi_+ ( x , z ) \overline{\Phi_- ( x , z ) } \) } = \frac 1{ x \Im z } . \]
The last equality here is implied by the explicit formula for the solution $ \Phi $ in the situation under consideration, $ \Phi ( x , z ) = \begin{pmatrix} z x \cr 1 \end{pmatrix} $.
By the nesting circles and the triangle inequality we have for $ x \le N $
\[ \frac1{\left| m_N ( z ) \right|} \le \left| \frac{ \Theta_-^N ( x , z ) }{ \Phi_-^N ( x , z ) } \right| + \left| \frac 1{ x \Im z } \right| . \]
The first term in the r. h. s. vanishes as $ N \to \infty $. Since $ x $ is arbitrary we conclude that 
$  \left| m_N ( z ) \right| \displaystyle{\mathop{\longrightarrow}_{ N \to \infty }} \infty $.
This, in turn, contradicts the assumptions about $ \mu_N $ which imply that $ \sup_N \int \frac{ \d \mu_N ( t ) }{ 1 + t^2 } $ is fnite and hence $ | m_N ( z ) | $ is bounded in $ N $ for each $ z \in \C_+ $. Thus,  the condition \ref{nondegenerateanew} is satisfied for $ \cH $, and the first assertion of the theorem is established. The second assertion is just theorem \ref{sptheoremsing} for the Hamiltonian $ \cH $. \end{proof}

\section{Appendix I. Reconstruction of monodromy matrix from the first column.}\la{appI}

\begin{theorem}\la{restore}
Let $ E $ be a regular HB function having no real zeroes, $ E ( 0 ) = 1 $, and let $ \Theta $ be defined by (\ref{ThetaE}). There exists an entire vector-function $ \Phi ( \lambda ) $ such that the matrix $ N  ( \lambda ) = [ \Theta ( \lambda ) , \Phi ( \lambda ) ] $ satisfies $ N ( 0 ) = I $ and

(i) $ \det N ( \lambda ) = 1 $,

(ii) \[ \frac 1i \( N^* ( \lambda ) J N ( \lambda ) - J \) \ge 0 ,\;   \lambda \in \C_+ . \]

(iii) The functions $\Phi_- / \Theta_- $, $\Phi_+ / \Theta_+ $ are Herglotz in $ \C_+ $ and do not have a linear growing term in its integral representation\footnote{This means that they are $ o ( |z| ) $ when $ z \to \infty $ along the positive imaginary axis.}. \end{theorem}

\subsection{Direct problem}\la{diprobl} In the notation of section \ref{direct}, for a given canonical system satisfying the compatibility condition (L) let us find the Fourier transform $ \cU \Phi_z $. We have, 
\bequnan \int_0^L \Phi_z^T \cH \Theta_w = \frac 1w \int_0^L \Phi_z^T J \Theta_w^\prime = \frac 1w \( \left. \Phi_z^T J \Theta_w \right|_0^L +\int_0^L \( J \Phi_z^\prime \)^T \Theta_w \) = \\ \frac 1w \( \Phi_z^T ( L )  J \Theta_w (L) - 1 + z \int_0^L \Phi_z^T \cH \Theta_w \) ,
\eequnan
 hence
 \be\la{diprob} \(  \cU \Phi_z \) ( w ) = \frac 1{\sqrt \pi }  \frac{ \Phi_z^T ( L )  J \Theta (L , w ) - 1 }{ w - z } . \ee
Let us calculate the scalar product of $ \Phi_\lambda $ and $ \Phi_0 \equiv \Bigl( \! \begin{array}{c}0 \\ 1 \end{array} \! \Bigr) $ twice. First, by differentiating it with respect to (\ref{can}) (compare the $22$-element of (\ref{identM})), we find that
\[ \llangle \Phi_\lambda , \Phi_0 \rrangle_H = \frac{\Phi_\lambda^+ ( L ) }\lambda . \]
On the other hand, since $ \cU $ is an isometry, it coincides with 
\[ \begin{array}{ccc} \llangle \cU \Phi_\lambda , \cU \Phi_0 \rrangle_{\HE} = \displaystyle{\frac 1\pi } & \displaystyle{ \llangle \frac{ \Phi_L^T ( \lambda )  J ( \Theta_L - \Theta_L ( \lambda ) ) }{ \cdot - \lambda } \right.} & \left. , \displaystyle{\frac{\Theta_L^+ - 1 }t} \rrangle_{ \HE } . \cr \noalign{\vskip2pt} & \overbrace{\Phi^T ( L , \lambda )  J \Theta ( L , \lambda ) = \det M = 1} & \hskip2cm (\star ) \end{array}\] Thus,
\be\la{sysphi}  \Phi_\lambda^T ( L ) J  G_\lambda = \frac{\Phi_\lambda^+ ( L ) }\lambda  , \ee 
\[ G^\pm_\lambda \colon = \frac 1\pi \llangle \frac{ \Theta_L^\pm - \Theta_L^\pm ( \lambda )}{ \cdot - \lambda }, \frac{\Theta_L^+ - 1 }t \rrangle_{ \HE } . \]

\subsection{back to inverse problem}
Let us consider (\ref{sysphi}) and ($\star $) as a system on $ \Phi^\pm_\lambda $. Notice that the coefficients of this system are well-defined in the framework of the inverse problem, for they only depend on values of $ \Theta $ at $ x = L $, and the space $ \HE $ is regular which ensures that the scalar products in the definition of $ G^\pm $ are defined. Let us \textit{define} $ \Phi^\pm_\lambda $ to be the solution of the system (\ref{sysphi}), ($\star $) with $ \Theta_L \equiv \Theta $. This system is solvable,
\be\la{phi-} \Phi^-_\lambda = 1 + \lambda G^-_\lambda , \ee
\be\la{phi+} \Phi^+_\lambda = \lambda G^+_\lambda .  \ee 
The solution is a obtained by a straightforward calculation taking into account that (see  (\ref{reprokernel}))
\[ \Theta_\lambda^T J G_\lambda = \llangle K_{ \overline \lambda } ,  \frac{\Theta^+ - 1 }t \rrangle_{ \HE } = \frac{ \Theta^+ ( \lambda ) - 1 }\lambda . \] 
By construction, $ N  ( \lambda ) \colon = [ \Theta_\lambda , \Phi_\lambda ] $ has determinant $ 1 $ and is an entire function. It remains to establish (\textrm{ii}) and (\textrm{iii}).

Assume first that $ \Theta_+ \notin \HE $. By theorem \ref{orthog}, the set $ \{ K_\tau \}_{ \Theta_+ ( \tau ) = 0 } $ is then an orthogonal basis in $ \HE $. On calculating the scalar product defining $ G^\pm $ in this basis,
\[ G^+_\lambda = \Theta^+_\lambda \sum_ { \Theta^+_\tau = 0 } \frac {\mu_\tau}\tau \frac 1{ \tau - \lambda }  =  \frac{\Theta^+_\lambda}\lambda \sum_ { \Theta_+ ( \tau ) = 0 } \mu_\tau \left[ \frac 1{ \tau - \lambda } - \frac 1\tau \right]  , \] 
\[ \mu_\tau = \frac 1{\pi \len K_\tau \rin^2 } = \frac 1{{\dot{\Theta}}^+ ( \tau ) \Theta_- ( \tau ) } . \] 
Thus 
\be\la{thplu} \frac{\Phi_\lambda^+}{\Theta_\lambda^+ } = \sum_ { \Theta_+ ( \tau ) = 0 } \mu_\tau \left[ \frac 1{ \tau - \lambda } - \frac 1\tau \right]  , \ee
where $ \mu_\tau $ are positive constants, hence the l. h. s. is a Herglotz function.

Assertion (\textrm{ii}) will follow from the identity 
\be\la{Jcontr} \frac 1{2i} \llangle \( N^*_\lambda J N_\lambda - J \) \mathbf{e} , \mathbf{e} \rrangle_{ \C^2 } = \pi ( \Im \lambda ) \len \xi_\lambda^T J N_\lambda  \mathbf{e}% \Bigl( \! \begin{array}{c} -i h \cr 1 \! \end{array} \Bigr) 
\rin^2_{ \HE } , \ee 
\[ \xi_\lambda ( t ) = \frac1\pi \frac 1{ t - \lambda } \( \Theta ( t ) - \Theta ( \lambda ) \) ,  \mathbf{e} = \Bigl( \! \begin{array}{c} h \cr 1 \! \end{array} \Bigr) , h \in \C . \]

\begin{exercise}
Identity (\ref{Jcontr}) will only be used as a tool to show that its l. h. s. is nonnegative, the precise form of the expression in the r. h. s. is irrelevant. Is it possible to derive (\textit{ii}) without actually proving the identity?
\end{exercise} 

\subsection{Derivation of identity (\ref{Jcontr}).} The square norm in the r. h. s. is
\bequnan \begin{array}{ccccc} \len \xi_\lambda^T J \( \Phi_\lambda + h \Theta_\lambda \) \rin^2 & = & \len \xi_\lambda^T J \Phi_\lambda -  h K_{\overline\lambda } \rin^2 & = & 
\cr & \overbrace{\xi_\lambda^T J \Theta_\lambda = - K_{\overline \lambda } } & & \overbrace{ \Theta_\lambda^T J \Phi_\lambda = -1 } & \end{array} \\ \len \xi_\lambda^T J \Phi_\lambda \rin^2 - 2 \Re \left[ \overline h \frac{ \Theta_{\overline \lambda}^T J \Phi_\lambda + 1 }{ \pi ( \overline \lambda - \lambda ) } \right] + \left|  h \right|^2 K_\lambda (\lambda ) , 
\eequnan
\bequnan \mathrm{ L. H. S. \, in \, }  (\ref{Jcontr}) =  \frac 1{2i} \llangle \( \begin{array}{cc} \overline{ \Theta_\lambda^T} J \Theta_\lambda & \overline{ \Theta_\lambda^T} J \Phi_\lambda \cr \overline{ \Phi_\lambda^T} J \Theta_\lambda & \overline{ \Phi_\lambda^T} J \Phi_\lambda \end{array} \) \mathbf{e} , \mathbf{e} \rrangle - \Im h  =  \frac 1{2i} \overline{ \Phi_\lambda^T} J \Phi_\lambda + \\ \Im \left[ \overline h \( { \Theta_{ \overline \lambda }^T} J \Phi_\lambda + 1 \) \right] + \frac 1{2i} \Theta_{ \overline \lambda}^T J \Theta_\lambda \left| h \right|^2 \eequnan

Comparing the last two terms in the obtained expressions (quadratic and linear ones in $ h $), and taking into account that $ K_\lambda ( \lambda ) =  { \pi \Im \lambda }^{ -1 } \Im \( \Theta^+_\lambda \overline{ \Theta^-_\lambda} \) $, we see that the respective terms differ exactly by the $ \pi \Im \lambda $ multiple.  Thus, to prove (\ref{Jcontr}) it remains to check that
\be\la{formaux} \pi \Im \lambda \len \xi_\lambda^T J \Phi_\lambda \rin^2 = \frac 1{2i} \overline{ \Phi_\lambda^T} J \Phi_\lambda \ee
To this end, we get rid of $ \Phi_\lambda^- $ in both sides by using $ \det N ( \lambda ) = 1 $, 
\[ \xi_\lambda^T J \Phi_\lambda = \frac 1{ \Theta^+_\lambda } \xi_\lambda^T J \( \Bigl( \! \begin{array}{c} 0 \cr 1 \end{array} \! \Bigr) + \Phi_\lambda^+ \Theta_\lambda \) = - \frac 1{ \Theta^+_\lambda } \( \xi_\lambda^+ + \Phi_\lambda^+ K_{ \overline \lambda } \) , \]
\[ \frac 1{2i} \overline{ \Phi_\lambda^T} J \Phi_\lambda =   \frac 1{\left| \Theta^+_\lambda \right|^2 } \( \Im \( \Phi_\lambda^+ \Theta_\lambda^+ \) + \frac 1{2i} \left| \Phi^+_\lambda \right|^2 \Theta_{ \overline \lambda } J \Theta_\lambda \) . \]
Inserting these in (\ref{formaux}) we see that the required fromula is equivalent to the following,
\be\la{formulaatlast} \pi \Im \lambda \len\xi_\lambda^+ + \Phi_\lambda^+ K_{ \overline \lambda }  \rin^2 = \Im \( \Phi_\lambda^+ \Theta_\lambda^+ \) + \frac 1{2i} \left| \Phi^+_\lambda \right|^2 \Theta_{ \overline \lambda } J \Theta_\lambda . \ee
We have 
\[  \len\xi_\lambda^+ + \Phi_\lambda^+ K_{ \overline \lambda }  \rin^2 =  \len \xi_\lambda^+ \rin^2 + 2 \Re \( \xi_\lambda^+ ( \overline \lambda ) \overline{ \Phi_\lambda^+ } \) +  \left| \Phi^+_\lambda \right|^2 K_{ \overline \lambda } ( \overline \lambda ) . \]
Finally, 
\be\la{xiplusnorm} \len \xi_\lambda^+ \rin^2 = \frac 1{\pi^2 } \sum_{ \Theta^+ ( \tau ) = 0 } \pi \mu_\tau \frac { \left| \Theta^+_\lambda \right|^2 }{  \left| \tau - \lambda \right|^2 } \mathop{=}^{ (\ref{thplu}) } \frac 1{\pi \Im \lambda }  \left| \Theta^+_\lambda \right|^2 \Im \( \frac{\Phi_\lambda^+}{\Theta_\lambda^+ } \) . \ee 
On calculating, we find that (\ref{formulaatlast}) holds, and thus (\ref{Jcontr}) is proved.
% \frac 1{\left| \Theta^+_\lambda \right|^2 }

It remains to establish (iii). The required assertion for the function $ \Phi^+ /\Theta^+ $ has already been established, see (\ref{thplu}). To deal with  $ \Phi^- /\Theta^- $ we are going to use the following 

\begin{lemma}
\la{columnrow} Let a matrix $ B $, $ \det B = 1 $, satisfy\footnote{Matrices satisfying (\ref{JCo}) are called \textit{$ J $-contractive}} 
\be\la{JCo} \frac 1i \(B^* J B- J \) \ge 0 . \ee 
Then $ \Im \( B_{22} / B_{ 21 } \) \ge 0 $, $ \Im \( B_{12} / B_{ 11 } \) \ge 0 $. 
\end{lemma}

\begin{proof} Taking complex conjugate of (\ref{JCo}) and multiplying out $ \overline B^* $ and $ \overline B $ in the result we find that $ A = \( {\overline B} \)^{ -1 } $ is $ J $-contractive as well. It follows that the daigonal elements of $ A^* J A $ must have non-negative imaginary part. These elements are $ 2i \Im \( B_{22} \overline{B_{ 21 }} \) $ and $ 2i\Im \(\overline{  B_{11}} B_{ 12 } \) $, and the assertion follows.
\end{proof}

When applied to $ N_\lambda $ this lemma says that $ \Phi^- /\Theta^- $ is Herglotz. To establish that it is $ o ( \lambda ) $ on the imaginary axis as $ \lambda \to i \infty $, consider $ \det N (\lambda ) = 1 $ written in the form
\[ \frac{\Phi_\lambda^-}{\Theta_\lambda^- } = \frac{\Phi_\lambda^+}{\Theta_\lambda^+ } + \frac 1{\Theta_\lambda^- \Theta_\lambda^+ } . \]
We already know that the first term in the r. h. s. is $ o ( y ) $. The functions $ \Theta_\pm $ are real entire and have all their zeroes real, hence by considering the corresponding canonical factorizations we infer that $ | \Theta_\pm ( iy ) | $ are  growing functions of $ y \in \R_+ $. It follows that the modulus of the second term in the r. h. s. is bounded (in fact, vanishes) as $ \lambda \to i \infty $. Thus the whole sum is $ o ( y ) $, (iii) holds and the proof of theorem \ref{restore} is completed in the case $ \Theta_+ \notin \HE $.

Let now $ \Theta_+ \in \HE $. In this case $ \Theta_+ \perp (\Theta_+ - 1 )/t $ by theorem \ref{orthog}. Hence (\ref{thplu}) holds unchanged. Indeed, $ \Theta_+ \( \Theta_+ - 1 \) / \left| E \right|^2 =  \Re \left[  ( \Theta_+ - 1 )/E \right] $ from whence
\[ \llangle  \frac{\Theta_+ - 1 }t , \Theta_+ \rrangle = \int \Re \left[  \frac{\Theta_+ - 1 }{ t E } \right] =  \lim_{ N \to + \infty } \Re \left[  \int \frac N{ t^2 + N }  \frac{\Theta_+ - 1 }{ t E } \right]  = \dots \]  
We had to insert the limit in $ N$ because $ \frac{\Theta_+ - 1 }{ t E } $ is not $ L^1 $ at infinity, to be able to switch the integration and taking the real part. Proceeding,
\[ \dots =  \lim_{ N \to + \infty } \Re \( \frac\pi{i} \frac{\Theta_+ ( i \sqrt N )  - 1 }{ E ( i \sqrt N ) } \) \mathop{\longrightarrow}_{ N \to + \infty}  0 \]
because $ \Theta_+ / E \in H^2 $ and $ | E ( i \tau ) | = | \Theta_- ( i \tau ) | \, \left| \Theta_+ / \Theta_- + i \right| \ge   | \Theta_- ( i \tau ) | \to \infty $ as $ \tau \to + \infty $. More generally, a similar calculation shows that $ \Theta_+  \perp \xi_\lambda^+ $ for any $ \lambda \in \C $.

The only other point in the argument where we have used the assumption $ \Theta_+ \notin \HE $ is in the formula (\ref{xiplusnorm}) where we have written down the norm of $ \xi^+_\lambda $ using the basis of $ K_\tau $. Since $ \xi_\lambda^+  \perp \Theta_+ $, as we have just mentioned, this formula also holds unchanged, and the proof of  the proof of theorem \ref{restore} is completed in this case as well.

\begin{exercise} Show that $ \Theta_+ \in \HE $ iff for the solution of the corresponding inverse problem the point $ x = L $ is a right endpoint of a singular interval with $ \cH ( x ) = \begin{pmatrix} 1 & 0 \cr 0 & 0 \end{pmatrix} $. \end{exercise}

\begin{exercise}\la{counterex} Show that there exists a matrix $ G = \begin{pmatrix} A & B \cr C & D \end{pmatrix} $ for which $ \Im ( A/ C ) , \Im ( B/D ) $, $ \Im ( B/A) $, $ \Im ( D/C ) $ are positive but $ \frac 1i ( G^* J G - J ) $ is not $ \ge 0 $. \end{exercise}

The part (\textit{ii}) of theorem \ref{restore} has an important consequence.
 
\begin{corollary}\la{trpos} The matrix $ N ( z ) $ constructed in theorem \ref{restore} satisfies $ \operatorname{tr} J \dot{N} ( 0 ) > 0 $. 
\end{corollary}

\begin{proof}
 On developing (\textit{ii}) at $ \lambda = 0 $ we find $ i^{ -1 } \( \overline \alpha {\dot N}^* ( 0 ) J + \alpha  J \dot N ( 0 ) \) \ge 0 $ for any $ \alpha \in \C_+ $, $ | \alpha | =1 $. In particular, $ \alpha  = i $ gives $ \Re ( J \dot N ( 0 ) ) \ge 0 $. On differentiating $ \det N ( \lambda ) = 1 $ at $ \lambda = 0 $, we obtain that $ \operatorname{tr} \dot N ( 0 ) = 0 $. This and $ \dot  N ( 0 ) $ having real entires imply that $ J \dot N ( 0 ) $ is selfadjoint. Thus $ J \dot N ( 0 ) \ge 0 $ .
\end{proof}

When discussing the constructive aspect of the solution of the inverse problem we are going to use another corollary of calculations in section \ref{diprobl} and formulae (\ref{phi-}), (\ref{phi+}). Let $ ( \cH , L ) $ be a regular canonical system, satisfying the compatibility condition (L). Then (\ref{phi-}), (\ref{phi+}) hold and plugging them into lemma \ref{trace} with $ x = L $ we find that
\be\la{LthroughE}  L =  \operatorname{tr} J \dot M ( 0 ) = \dot{\Phi}_+ ( 0 ) - \dot{\Theta}_- ( 0 ) = G_0^+ -  \dot{\Theta}_- ( 0 ) = \frac 1\pi \len \frac{ \Theta_+ - 1 }\lambda \rin_{\HE}^2 -  \Im \dot{E} ( 0 )  \ee
where $ E = E ( L , \cdot ) $.

\section{Appendix II}

\noindent\textit{Proof of theorem \ref{isometry}.} Throughout, $ S ( z ) = e^{ ia z} B ( z ) $, $ a \ge 0 $, $ B $ being the Blaschke product over zeroes of $ W $, is the canonical factorization of the inner function $ S = W^* / W $. 

$ \Rightarrow $ Let $ \overline w $ be a root of $ W $, then 
\[ \len K_w \rin^2 = K_w ( w ) = \frac 1{ 4\pi \Im w } \left| W ( w ) \right|^2 . \] 
On the other hand, 
\[ \len K_w \rin_{ L^2 \( \R , \left| W \right|^{ -2 } \d \mu \) }^2 = \frac 1{\( 2 \pi \)^2 } \left| W ( w ) \right|^2 \int \frac { \d \mu ( t ) }{ \left| w - t \right|^2 } , \] 
hence $ \int \frac{ \d \mu ( t ) }{ 1 + t^2 } $ is finite, and 
\[ \int \frac{ \d \mu ( t ) }{ \left| w - t \right|^2 } = \frac \pi{\Im w } . \] 
Thus the Poisson transform of the measure $ \d \mu $ exists and defines a positive harmonic function in $ \C_+ $, hence there exists a $ \Xi \in H^\infty $, $ \len \Xi \rin \le 1 $, such that 
\be\la{xi} \Re \frac{ 1 + \Xi ( z ) }{ 1 - \Xi ( z ) } = \frac{ \Im z }\pi \int \frac {\d \mu ( t )}{\( t - \Re z \)^2 + \( \Im z \)^2 } , \; z \in \C_+ . \ee

Notice that the function $ \Xi $ is defined by (\ref{xi}) up to a real parameter (to be fixed later). The implication will be proved if we show that $ \Xi / S $ is a contractive analytic function in $ \C_+ $ for an appropriate choice of this parameter. 

Observe first that, by the calculation above, if $ S ( z ) = 0 $, that is, $ W ( \overline z ) = 0 $, then the r. h. s. in (\ref{xi}) is $ 1 $, which implies $ \Xi ( z )  = 0 $, and it is easy to check that the multiplicity of zero of $ S $ at a point $ z $ is not greater than that of $ \Xi $. Thus, $ \Xi/ S $ is indeed analytic in $ \C_+ $, and therefore $ B $ is an inner factor of $ \Xi $. It remains to show that $  e^{  ia z} $ is also an inner factor of $ \Xi $, that is, $ \Xi e^{ - iaz} \in H^\infty $. We are going to establish that 

$ 1^\circ $. $ \Xi ( z ) = O \( e^{ - a^\prime \Im z } \) $ as $ \Im z \to + \infty $ for any $ a^\prime < a $.

This assertion will impy $ \Xi e^{ - iaz} \in H^\infty $, and hence the $ \Rightarrow $ implication of the theorem, by an elementary Fragmen-Lindel\"of theorem (or, alternatively, by the arithmetics of inner functions).

If $ a = 0 $ then $ 1^\circ $ is trivial, so let $ a > 0 $. Let us calculate the norm of reproducing kernels, $ K_w $, at $ w = k + i \tau , \tau \to + \infty $, 
\[ \len  K_{ w } \rin^2 = \frac 1{ 4 \pi \tau } \( \left| W ( w ) \right|^2 - \left| W^* ( w ) \right|^2 \) = \frac 1{ 4 \pi \tau } \left| W ( w ) \right|^2 \( 1 + O \( e^{ -2 a \tau } \) \) .  \]
On the other hand, by the isometry of the inclusion the same quantity is
\bequnan \frac 1{ \( 2 \pi \)^2 } \int \frac{ \d\mu ( t ) }{ \left| W (  t ) \right|^2 } \frac{ \left| W ( t ) \overline { W ( w) } - W^* ( t ) W ( \overline w ) \right|^2 }{ \left| w - t \right|^2 } = \frac 1{ \( 2 \pi \)^2 } \left| W ( w ) \right|^2  \int \frac {\d \mu ( t )}{ \left| t - w \right|^2 } \( 1 + O \( e^{ - a \tau } \)  \) . 
\eequnan 
On comparing, 
\[ \frac \tau\pi \int \frac {\d \mu ( t )}{ \left| t - w \right|^2 } = 1 + O \( e^{ - a \tau } \tau \) , \] 
and this estimate is locally uniform in $ k = \Re w $. Now, let $ g ( w ) $ be the analytic function in $ \C_+ $ whose real part is the Poisson transform of $ \d \mu $. We have just proved that $ \Re g ( w ) = 1 + O \( e^{ - a^\prime \Im w } \) $ for any $ a^\prime < a $, and this estimate is locally uniform in $ \Re w $. From the Cauchy-Riemann equations it follows that $ \Im g ( w ) = \mathrm{const} + O \( e^{ - a^\prime \Im w } \) $. Choose this constant to be zero, and define $ \Xi $ by $ g = \frac{ 1 + \Xi }{ 1 - \Xi } $. By construction, this function satisfies (\ref{xi}), and $1^\circ $ holds. 

$ \Leftarrow $ The proof of this implication consists in checking that the scalar products in $ \HW $ and in $ L^2 \( \R , \left| W \right|^{ -2 } \d \mu \) $ coincide on elements of a total set in $ \HW $. The total set to be used is\footnote{To keep notation at minimum we will assume throughout the argument that all the zeroes of $ W $ are simple.} $ {\mathcal N} = \{ K_w \}_{ W ( \overline w ) = 0 } \cup \{ f_b \}_{ 0 < b < a } $, $ f_b ( z ) \colon = \frac{ 1 - e^{ izb} }z W ( z ) $. We postpone the derivation of its totality in $ \HW $ until the end of the proof and first establish the implication taking the totality for granted. It then suffices to check that for all $ K_w $ and $ f_b $ from this set
\bequnan 
1^\circ . & \llangle K_w , K_{ w^\prime } \rrangle_{ \HW } = \llangle K_w , K_{ w^\prime } \rrangle_{ L^2 \( \R , \left| W \right|^{ -2 } \d \mu \) } ,  \cr 
2^\circ . & \llangle f_b , f_{ b^\prime } \rrangle_{ \HW } = \llangle f_b , f_{ b^\prime } \rrangle_{ L^2 \( \R , \left| W \right|^{ -2 } \d \mu \) } , \cr
3^\circ . & \llangle K_w , f_b \rrangle_{ \HW } = \llangle K_w , f_b \rrangle_{ L^2 \( \R , \left| W \right|^{ -2 } \d \mu \) }  . \eequnan

The l. h. s. in $ 1^\circ $ is 
\[ K_w ( w^\prime ) = \frac 1{2 \pi i } \frac{\overline {W ( w )} W ( w^\prime ) }{ \overline w - w^\prime } , \] 
hence $ 1^\circ $ will be established if we show that
\be\la{scprod1} \int \frac { \d \mu ( t ) }{ (\overline w - t ) ( w^\prime - t ) } = - 2 \pi i \frac 1{ \overline w - w^\prime } . \ee 
Indeed, let 
\[ g = \frac{ 1 + A S }{ 1 - A S } , \]
then the assumption of the theorem means that 
\be\la{Hergl} g ( z ) = - \frac i\pi \int \( \frac 1{t-z} - \frac t{ 1+t^2 } \) \d \mu ( t ) + i \gamma , \; \gamma \in \R . \ee 
It follows that 
\be\la{scprod}  \mathrm{L. H. S. \,\, in} \, (\ref{scprod1})  = \frac \pi i \frac{ \overline{ g ( w ) } + g ( w^\prime ) }{ \overline w - w^\prime } . \ee  As $ g ( w ) = g ( w^\prime ) = 1 $ since $ S ( w ) = S ( w^\prime ) = 0 $, we infer (\ref{scprod1}). $ 1^\circ $ is established.

Proceeding, the r. h. s. of $ 2^\circ $ is \bequnan \int \frac{ \( 1 - e^{ ikb} \) \( 1 - e^{ -ik b^\prime } \)  }{ k^2 } \, \d \mu ( k )  = \\ \begin{array}{ccc} \displaystyle{ \lim_{ \von \downarrow 0 }} & \underbrace{\int \frac{ \( 1 - e^{ ikb} \) \( 1 - e^{ -ik b^\prime } \)  }{ k^2 } \Re g ( k + i \von ) } & \d k , \cr & (I) & \end{array} \eequnan 
\bequnan (I) \equiv \frac 12 \int \frac{ \( 1 - e^{ ikb} \) \( 1 - e^{ -ik b^\prime } \)  }{ k^2 } ( g ( k + i \von ) - 1 ) \d k + \\ \frac 12 \overline{ \int \frac{ \( 1 - e^{ -ikb} \) \( 1 - e^{ ik b^\prime } \)  }{ k^2 } ( g ( k + i \von ) - 1 ) \d k } + \\ \int \frac{ \( 1 - e^{ ikb} \) \( 1 - e^{ -ik b^\prime } \)  }{ k^2 } \, d k  . \eequnan 
The first two terms in the r. h. s. of the formula for (I) vanish by contour integration since $  g ( k + i\von ) - 1 = O \( e^{ - a \von } \)$ uniformly in $ k $, and the last term is precisely the l. h. s. in $ 2^\circ $. $ 3^\circ $ is proved similarly to $ 2^\circ $ - the measure in the r. h. s. of $ 3^\circ $ is approximated by $ \Re g ( k + i \von ) \d k $, and the resulting integral is calculated explicitly to be the l. h. s. The first assertion of the theorem is established modulo totality of the set $ \mathcal N $.

\textit{Proof of totality of $ \mathcal N $.} Let $ f \in \HW $ be orthogonal to $ \mathcal N $. Then $ f ( w ) = \llangle f , K_w \rrangle_{\HW } = 0 $ for any $ w $ such that $ W ( \overline w ) = 0 $. It follows that $ B $ is an inner factor of the function $ f / W \in H^2 $. Now $ f \perp f_b $ means that $ f/W $ is orthogonal in $ H^2 $ to $ ( 1 - e^{ izb} )/z $ for all $ b \in ( 0 , a ) $.  The latter function is the inverse Fourier transform of the indicator of the interval $ ( 0 , b ) $ up to a numerical factor. It follows that $ f/W $ is the inverse Fourier transform of an $ L^2 ( \R ) $ function orthogonal to functions supported on $ ( 0 , a ) $. By the Paley-Wiener theorem this means that $ f/W \in e^{ iaz} H^2 $. Combining these, we conclude that $ f/W \in S H^2 $. On the other hand, $ f / W \perp S H^2 $ by the very definition of a de Branges space (it's $ f^*/W \in H^2 $ spelled out!), so $ f = 0 $.      

We now proceed to proving the second assertion of the theorem. Let (\ref{infinitymass}) be satisfied with a $ c_0 > 0 $. Then $ S $ is a Blaschke product, hence the set $ \{ K_w \}_{ W ( \overline w ) = 0 } $ is total in $ \HW $. Let us calculate the scalar product of reproducing kernels, $ K_w $ and $ K_{ w^\prime } $, from this set. The Herglotz representation for the function $ g = \frac{ 1 + A S }{ 1 - A S } $ differs from (\ref{Hergl}) in this case by an extra term of the form $ -i c_0 z $ in the r. h. s., hence (compare (\ref{scprod})) 
\bequnan \int \frac { \d \mu ( t ) }{ (\overline w - t ) ( w^\prime - t ) } = \frac \pi i \frac{ \overline{ g ( w ) + i c_0 w  } + g ( w^\prime ) + i c_0 w^\prime }{ \overline w - w^\prime } = [ g ( w ) = g ( w^\prime ) = 1 ] = \\  - 2 \pi i \frac 1{ \overline w - w^\prime } - c_0 \pi . \eequnan

Multiplyed by $ \overline{ W ( w ) } W ( w^\prime ) $ the last line becomes 
\[ \llangle K_w , K_{ w^\prime } \rrangle_{ \HW } = \llangle K_w , K_{ w^\prime } \rrangle_{ L^2 \( \R , \left| W \right|^{ -2 } \d \mu \) } + c_0 \frac \pi{ \( 2 \pi i\)^2 } \overline{ W ( w ) } W ( w^\prime ) . \]
By bi-linearity it follows that for any finite linear combinations, $ f = \sum f_j K_{ w_j } $, $ g = \sum g_j K_{ w_j } $, of elements of the total set,
\[ \llangle f , g \rrangle_{ L^2 \( \R , \left| W \right|^{ -2 } \d \mu \) } = \llangle f , g \rrangle_{ \HW } + \( - \frac { c_0 }{ 4 \pi } \) \( \sum f_j  \overline {W ( w_j )} \) \overline{ \sum g_j  \overline {W ( w_j )} } \]
In partcular, this means that the closure in $ \HW $ of the linear subset 
\[ {\mathcal X} = \left\{ f = \sum_{ j = 1}^N f_j K_{ w_j } , \; N = N ( f ) < \infty \colon \sum f_j \overline{ W ( w_j ) } = 0 \right\} \]
of such combinations is isometrically embedded into $ L^2 \( \R , \left| W \right|^{ -2 } \d \mu \) $. 

It remains to show that the orthogonal complement of $ \mathcal X $ is $ \HW $ is spanned by a linear combination of $ W $ and $ W^* $. Let $ f \in \HW \ominus \mathcal X $. Pick up an arbitrary $ w \colon W ( \overline w ) = 0 $, and let $ c = f ( w ) / W ( w ) $. Notice that $ c $ dose not depend on the choice of $ w $. Indeed, \[ K_w - \frac{ \overline{ W ( w ) } }{ \overline{ W ( \lambda )} } K_\lambda \in {\mathcal X} \] if $ W ( \overline \lambda ) = W ( \overline w ) = 0 $, hence $ f $ is orthogonal to this element, that is, $ f ( w ) -  \frac{ W ( w ) }{ W ( \lambda ) } f ( \lambda ) = 0 $. 

It follows that \[ \frac{ f - c W }{ \cdot - \lambda } \in \HW \] for any $ \lambda \colon W ( \overline \lambda ) = 0 $. Obviously, this element vanishes at all points $ z $, $ z \ne \lambda $, $ W ( \overline z ) = 0 $, hence is orthogonal to all the reproducing kernels from the total set except for at most $ K_\lambda $, that is, it is a multiple of the corresponding element, $ K_{ \overline \lambda } $, of the biorthogonal system, 
\[ \frac{ f - c W }{ \cdot - \lambda } = c_1 K_{ \overline \lambda } \equiv c_2 \frac{ W^* }{ \cdot - \lambda } \] for some constants $ c_1 , c_2 $ depending on $ \lambda $ only. Thus $ f $ is a linear combination of $ W $ and $ W^* $. Since $ f \in \HW $, this is only possible if $ f $ is a multiple of $ W + e^{ i \alpha } W^* $ for some real $ \alpha $. Finally, $ \mathcal X $ is not dense in $ \HW $, for if it were, then one can use the first assertion of the theorem \textit{in the reverse direction} to infer that $ c_0 = 0 $. \hfill $ \Box $ 

\medskip

\begin{remarknonumb}
The subspace $ \mathcal X $ in the second assertion of the theorem is a de Branges space. \end{remarknonumb}

\begin{proof} By the axiomatic characterization of de Branges spaces it suffices to check that $ f  / ( \cdot - a ) \in \mathcal X $ whenever $ f \in \mathcal X $, $ a \notin \R $ and $ f ( a ) = 0 $. By theorem \ref{orthog}, the set of vectors, $ e_k $, of the form 
\[ e_k = \frac {W + e^{ i \alpha } W^* }{ \cdot - t_k } , \] 
where $ t_k $ ranges over the zeroes of the numerator, forms an orthogonal basis in $ \mathcal X $.
Let $ f = \sum f_k e_k $ be the decomposition of $ f $ w. r. t. this basis. Then 
\bequnan \frac {f ( z)}{ z- a } = \sum f_k \frac {W ( z )  + e^{ i \alpha } W^* ( z)  }{ ( z - t_k )( z - a ) } = \sum f_k \frac 1{t_k - a } \frac {W ( z )  + e^{ i \alpha } W^* ( z)  }{ z - t_k } - \\ \frac 1{ z - a }\sum f_k \frac {W ( z )  + e^{ i \alpha } W^* ( z)  }{t_k - a } . \eequnan 
The sum in the second term in the r. h. s. is $ - f ( a ) $, thus it vanishes, and the first term is a linear combination of $ e_k $'s hence an element of $ \mathcal X $. \end{proof} 

\begin{remark} \la{Polya} Let $ G $ be an HB entire function of exponential type. Then for all $ k \in \R $ the function $  | G ( k + i \von ) | $ of $ \von $ is monotone increasing on $ \R_+ $ and tends to $ \infty $ as $ \von \to +\infty $. \end{remark}

\begin{proof} Let \[ G ( z ) = C e^{ az } \prod \( 1 - \frac z{z_j} \) e^{ \frac z{z_j} } \] be the Weierstra\ss\  product for $ G ( z ) $. Notice that the sum $ \sum \Im \( 1/z_j \) $ converges because the function $ G^* / G $ is inner and the Blaschke condition is hence satisfied. Then \[ G (  z ) = C e^{ b z } \prod \( 1 - \frac z{z_j} \) e^{ z \Re \( \frac 1{z_j} \) } . \] The assertion will be proved if we show that $ \Im b \le 0 $. We have \[ \frac{G^* ( z )}{ G ( z ) } = e^{ ( \overline b - b ) z } \prod \frac{z_j}{\overline{z_j}}\, \frac{ z -  \overline{z_j} }{ z -  z_j } . \] Comparing the r. h. s. with the canonical factorization of the inner function in the l. h. s., we infer that $ \Im b \le 0 $ indeed. \end{proof}

Solving inverse problems requires studying hereditary properties of regularity with respect to inclusion. A trivial remark is that if a de Branges space $ \HG $ is regular and is isometrically contained in a de Branges space $ \HE $, then $ \HE $ is also regular. According to theorem \ref{isomembed} the converse to this assertion holds true if the functions $ G $ and $ E $ have no real zeroes.     

\medskip

\noindent\textit{Proof of theorem \ref{isomembed}.} Let $ g \in \HG $, 
\[ f ( z ) \colon = \frac{g ( z) - g  ( 0 ) }z . \] 
It suffices to show that $ f \in \HG $. Given an $ h \in \HW \ominus \HG $ let 
\bequnan \Psi_1 ( z ) & = &\llangle \frac{ G ( z ) f - f( z) G}{ \cdot - z } , h \rrangle_{ \HW } , \\ \Psi_2 ( z ) & = & \llangle \frac{ G^* ( z ) f - f( z) G^* }{  \cdot -z } , h \rrangle_{ \HW } . \eequnan Then, $ \Psi_{ 1,2 } $ are obviously entire, and $ \Psi_1 G^* - \Psi_2 G = 0 $. On account of $ G $ and $ G^* $ having no common zeroes by the assumption, this means that 
\be\la{psi+-} \left\{ \begin{array}{ccc} \Psi_1 & = & \Psi G \hfill \cr
\Psi_2 & = &\Psi G^*  \end{array} \right. \ee
for some entire function $ \Psi $. Let us show that $ \Psi \equiv 0 $. 

We have, 
\be\la{Psistructure} \begin{array}{ccrc} \Psi ( z ) = & \underbrace{\llangle \frac f{ \cdot - z } , h \rrangle_{L^2 \( \R , \left| W \right|^{ -2 } \) } } & \!\! - \, \displaystyle{\frac{ f( z ) }{G (z ) }} & \underbrace{ \! \llangle \frac G{ \cdot - z } , h \rrangle_{L^2 \( \R , \left| W \right|^{ -2 } \) }} , \cr \noalign{\vskip4pt} & \textrm{Cauchy transform of } & & (II) \cr & \textrm{an } L^1-\textrm{function} & & \end{array} \ee
\bequnan (II) = ( z - z_0 ) \llangle \frac G{ ( \cdot - z )  ( \cdot - z_0 ) } , h \rrangle_{L^2 \( \R , \left| W \right|^{ -2 } \) } + \llangle \frac G{ \cdot - z_0 } , h \rrangle_{L^2 \( \R , \left| W \right|^{ -2 } \) } \equiv \\ ( z - z_0 ) \( \textrm{Cauchy transform of an } L^1-\textrm{function}\) + \textrm{const} .\eequnan This and a similar representation obtained from the second line in (\ref{psi+-}) mean that $ \Psi $ is a function of bounded type in both halfplanes, hence an entire function of exponential type by the Krein theorem. Also, this means that $ \( 1 + t^2 \)^{ -1 } \ln |\Psi | \in L^1 ( \R ) $, and 
\[ \ln | \Psi ( z ) | \le \frac {\Im z }\pi \int \frac{ \ln_+ | \Psi ( t ) | \mathrm{d} t }{ \( \Re z - t \)^2 + \( \Im z \)^2 } , \]
 which implies that $ \ln \Psi ( z ) = o ( |z| ) $ as $ z \to \infty $ along any ray $ \mathrm{arg} \, z  = \theta $, $ 0 < \theta < \pi $. Looking once more at (\ref{Psistructure}) we see that for $ \Im z > 1 $, $ \Re z = 0 $
\[ \begin{array}{ccc} | \Psi ( z ) | \le \displaystyle{ C \left[ \frac 1{ \Im z } + \left| \frac {f ( z ) }{ G ( z ) } \right|  \( 1 + \frac {| z - z_0 |}{ \Im z } \) \right]} & \le & \displaystyle{ C \( \frac1{\Im z } + \frac 1{ G ( z ) } \) }. \cr & \overbrace{  g / G \in H^2 } & \end{array} \]
According to remark \ref{Polya} $ | G (  i \von ) | \to + \infty $ as $ \von \to + \infty $. Thus, $ \Psi ( i \von ) \to 0 $ as $ \von \to + \infty $. A similar argument  starting from $ \Psi = \Psi_2 / G^* $ gives that $ \Psi ( i \von ) \to 0 $ as $ \von \to - \infty $. Bringing everything together, 
\[ \begin{array}{c} \ln | \Psi ( z ) | = o ( |z| ) \textrm{ on any ray } e^{ i \theta } [ 0 , +\infty ) , \theta \ne 0 , \pi \cr \noalign{\vskip4pt}  \Psi ( z ) \textrm{ is of exponential type}, \cr \noalign{\vskip4pt}  \Psi ( i \tau ) \displaystyle{ \mathop{\longrightarrow}_{\tau \to \pm \infty }} 0  . \end{array} \]  
Now a hall-of-mirrors argument (see also the beginning of section \ref{finreg}) shows that $ \Psi \equiv 0 $.

Thus, 
\[ \frac{ G ( z ) f - f( z) G}{ \cdot - z } \in \HG . \] 

\begin{exercise} In itself, the assertion just established says $ \( t + i \)^{-2} G^{ -1 } \in H^2 $, while what we need is $ \( t + i \)^{-1} G^{ -1 } \in H^2 $. Is it possible to proceed with the solution of the inverse problem using just the easier inclusion? Notice also that so far we have not used the isometry of the embedding, just the equivalence of $ \HG $- and $ \HW $-norms on $ \HG $. \end{exercise}

Let $ P $ be the orthogonal projection on $ \HG $ in $ \HW $, $ P^\perp = I - P $, then \[ \begin{array}{ccccc} \underbrace{\frac{ G ( z ) f - f( z) G}{ \cdot - z }} & = & \underbrace{\frac{ G ( z ) Pf - (Pf) ( z) G}{ \cdot - z }} & + & \underbrace{\frac{ G ( z ) P^\perp f - ( P^\perp f) ( z)  G}{ \cdot - z } .} \cr \in \HG & & \in \HG & \Longrightarrow & \in \HG \end{array} \] On taking scalar product with $ P^\perp f $, we find
\[ 0 = \llangle \frac{ G ( z ) P^\perp f - ( P^\perp f) ( z)  G}{ \cdot - z } , P^\perp f \rrangle_{ \HW } , \] that is, for $ z \notin \R $, 
\be\la{fperp} \int_{ \R } \left| \frac{ ( P^\perp f ) (t) }{ W ( t ) } \right|^2 \frac1{ t - z } \d t = \frac{ ( P^\perp f ) (z) }{ G ( z ) } \llangle \frac G{\cdot - z } , P^\perp f \rrangle_{ \HW } . \ee

Consider the limit $ z = i \tau $, $ \tau \to + \infty $. The l. h. s. is $ \( - i \tau \)^{ -1 } \int_{ \R } \left|
\frac{ ( P^\perp f ) (t) }{ W ( t ) } \right|^2 \d t ( 1 + o ( 1 ) ) $. Pick up a $ z_0 $, $ G ( z_0 ) = 0 $, then the scalar product in the r. h. s. is 
\[ \begin{array}{cccc} \llangle \displaystyle{ \frac G{\cdot - z }} , P^\perp f \rrangle_{ \HW } - & \underbrace{ \llangle \frac G{\cdot - z_0 } , P^\perp f \rrangle_{ \HW }} & = ( z - z_0 ) \llangle \displaystyle{ \frac G{( \cdot - z ) ( \cdot - z_0 ) } } , P^\perp f \rrangle_{ \HW } = \cr & \textrm{this is zero!} & \end{array} \] \bequnan \begin{array}{cccc} \displaystyle{ \int_{ \R }} & \!\!\!\!\!\! \underbrace{\frac{ z-z_0 }{ t-z }} &
\!\!\!\!\!\!\!\!\! \underbrace{ \frac{ G ( t ) \overline{ ( P^\perp f ) (t)} }{ ( t - z_0 ) \left| W ( t ) \right|^2 }} & \d t \displaystyle {\mathop{\longrightarrow}_{ \tau \to + \infty }} - \int_{ \R } \frac{ G ( t ) \overline{ ( P^\perp f ) (t)} }{ ( t - z_0 ) \left| W ( t ) \right|^2 } \d t \equiv\cr & | \dots | \le \textrm{const} & \in L^1 & \end{array}  \\ \llangle \frac G{ z_0 - \cdot } , P^\perp f \rrangle_{ \HW }  = 0 .\eequnan 
Comparing this with the asymptotics of the l. h. s. in (\ref{fperp}) we see that the theorem will be proved if we show that 
\be\la{Pperp} \frac{ ( P^\perp f ) ( i \tau ) }{ G ( i \tau ) } = O \( \frac 1\tau \) , \; \tau \to + \infty . \ee 

Notice first that this estimate holds with $ P^\perp f $ replaced by $ f $, that is, 
\[ \frac{ f ( i\tau )}{ G ( i\tau ) } \equiv \frac{ g ( i\tau ) - g ( 0 ) }{ i\tau G ( i \tau ) }= O \( \tau^{ -1 } \) , \] 
for $ g / G \in H^2 $ and $ | G ( i \tau ) | $ is monotone increasing for $ \tau > 0 $ by remark \ref{Polya}. This makes it reasonable to try to estimate $ P f / G $. However, the only thing about $ Pf/G $ we know immediately is that it is an $ H^2 $-function and is therefore $  O \( \tau^{ -1/2 } \)  $ as $ \tau \to + \infty $. To proceed, let us recall that by lemma \ref{deBrchar} $ Pf / G $ is not just an $ H^2 $-function, but satisfies   
\be\la{forthogdeBr} \left| \frac{  ( Pf ) ( i \tau )}{ G ( i\tau ) } \right| \le  C \sqrt{ \frac{ \left| G (i \tau ) \right| - \left| G^* (i \tau ) \right|}{ \tau \left| G (i \tau ) \right| } } . \ee
Here we took into account that $ | G ( i \tau ) | + | G^* ( i \tau ) | \asymp | G ( i \tau) | $ for $ \tau > 0 $. We are going to estimate the r. h. s. via $ ( P^\perp f ) ( i \tau) $. To this end, we shall show $ P^\perp f $ to be a linear combination of $ G $ and $ G^* $.   
Let us first establish that 
\be\la{prodnum} \llangle \frac{ G ( z ) P^\perp f - (P^\perp f) ( z) G}{ \cdot - z } , h \rrangle_{ \HG } = 0 \ee 
for any $ h \in \HG $, $ z \notin \R $ such that $ h \( \overline z \) = 0 $. Since the embedding $ \HG \subset \HW $ is isometric, the scalar product in the l. h. s. can be replaced by the product in $ L^2 \( \R , \left| W \right|^{ -2 } \) $, and then split into a sum of two products according to terms in the numerator. We are going to show that they both vanish, $ \llangle \frac{ P^\perp f }{ \cdot - z } , h \rrangle_{ L^2 \( \R ,  \left| W \right|^{ -2 } \) } = 0 $ and $ (II) \equiv \llangle \frac G{ \cdot - z } , h \rrangle_{ L^2 \( \R ,  \left| W \right|^{ -2 } \) } = 0 $. Indeed, 
\[ \llangle \frac{ P^\perp f }{ \cdot - z } , h \rrangle_{ L^2 \( \R ,  \left| W \right|^{ -2 } \) } = \llangle P^\perp f , \frac h{ \cdot - \overline z } \rrangle_{\HW } = 0 . \]
To handle \textit{(II)} we check that the product in there can be replaced by the $ L^2 \( \R , \left| G \right|^{-2} \) $-product. Pick up an arbitrary $ \xi \in \HG $ such that $ \xi ( z ) \ne 0 $. Then 
\bequnan \begin{array}{rclcl} (II) = \displaystyle{\frac 1{\xi ( z ) }} \Bigl\langle \!\!\!\!\! & \underbrace{ \frac{ \xi( z ) G  - G( z) \xi }{ \cdot - z }} & \!\!\!\!\! , h \Bigr\rangle_{  L^2 \( \R ,  \left| W \right|^{ -2 } \) } + \displaystyle{\frac { G ( z ) }{ \xi ( z ) }} \Bigl\langle \xi , \!\!\!\!\!\!\! & \underbrace{\frac h{ \cdot - \overline z }} & \!\!\!\!\!\!\!\Bigr\rangle_{  L^2 \( \R ,  \left| W \right|^{ -2 } \) } = \cr & \in \HG & & \in \HG & \end{array} \\ \frac 1{\xi ( z ) } \Bigl\langle \frac{ \xi( z ) G  - G( z) \xi }{ \cdot - z } , h \Bigr\rangle_{ \HG } + \frac { G ( z ) }{ \xi ( z ) } \Bigl\langle \xi , \frac h{ \cdot - \overline z }  \Bigr\rangle_{ \HG} = \\ \int \frac{ G( t ) \overline{ h ( t ) }}{ t - z } \frac{ \d t}{ \left| G ( t ) \right|^2 } =  \overline{ \int \frac{ h ( t ) } { G ( t ) } \frac{ \d t}{ t - \overline z }} = 0  \eequnan 
for $ z \in \C_+ $ by $ h / G \in H^2 $ and for $ z \in \C_- $ by $ h ( \overline z ) = 0 $. Thus, (\ref{prodnum}) is established, hence there is a $ c = c ( z ) \in \C $ such that 
\[ \frac{ G ( z ) P^\perp f - (P^\perp f) ( z) G}{ \cdot - z } = c K^G_{ \overline z } . \]
Recalling the expression for a reproducing kernel via the corresponding HB function, we see that the numerator in the l. h. s. is a linear combination of $ G $ and $ G^* $. One concludes that $ P^\perp f = \alpha G + \beta G^* $ for some constants $ \alpha , \beta $. Without loss of generality one can assume that $ f $ is real, hence so is $ P^\perp f $. It follows that $ \beta = \overline \alpha $ and therefore 
\[ \left| G (z ) \right| - \left| G^* (z) \right| \le C \left| ( P^\perp f ) ( z ) \right| . \] 
Plugging this into (\ref{forthogdeBr}) we find that 
\[ \textrm{ R.H.S. of (\ref{forthogdeBr}) } = O \(  \sqrt {\frac{ | ( P^\perp f ) ( i \tau )|}{ | G ( i \tau ) | \tau } } \) . \]
 On the other hand, as has been mentioned, $ f / G $ is $ O \( \tau^{ -1 } \) $. Combining these we find that when $ \tau \to + \infty $  \[ \frac{( P^\perp f) ( i \tau )}{ G ( i \tau ) } = O \( \tau^{ -1 } + \sqrt { \frac 1\tau \left| \frac{ ( P^\perp f ) ( i \tau )}{  G ( i \tau ) } \right| } \) . \]
Solving this inequality we get (\ref{Pperp}). \hfill $ \Box $

\section{Appendix III. Proof of the lattice property (Theorem \ref{lattice})}\la{Latticeproof} 

Suppose by contradiction that there exist $ f_X , f_Y \in L^2 ( \R , \d \mu ) $ such that $ f_X \perp  X $, $ f_X \notin Y^\perp $; $ f_Y \perp Y $, $ f_Y \notin X^\perp $. For nonzero $ h \in Y $ and $ g \in X $ (to be chosen later) define 
\[ \xi_Y ( z ) = \frac 1{h(z)} \llangle \frac{ g ( z ) h - h( z) g}{ z - \cdot } , f_Y \rrangle_{ L^2 ( \R , \d \mu ) }  . \]
Notice that $ \xi_Y $ does not depend on $ h \in Y $. Indeed, the difference of $ \xi $'s corresponding to two different nonzero $ h $, say, $ h_1 $ and $ h_2 $, is  
\[ 
 \begin{array}{ccc}
\displaystyle{\frac {g ( z )}{ h_1 ( z ) h_2 ( z ) }} \bigg\langle \!\!\!\! & \underbrace{ \frac{ h_2 ( z ) h_1 - h_1 ( z) h_2 }{ z - \cdot }} & , f_Y \bigg\rangle_{ L^2 ( \R , \d \mu ) } = \langle \dots , \dots \rangle_Y = 0 \cr & \in Y &  \end{array}  \]
on account of the fact that $ Y $ is a \textit{subspace} in $  L^2 ( \R , \d \mu ) $. It follows that $ \xi_Y $ is an entire function. Our immediate goal is to show that $ \xi_Y $ has finite exponential type for an appropriate choice of $ g $. By its definition, in each of the halfplanes $ \C_\pm $
\be\la{xiYstr} \xi_Y = \frac gh \cdot \( \textrm{Cauchy transform of a finite measure}\) + \( \textrm{Cauchy transform of a finite measure} \) . \ee

Now, let us choose $ g $ and $ h $ so that they have maximal admissible by their de Branges spaces growth at infinity, that is, they are of magnitude $ | E_X | / |z| $ ($ | E_Y | / |z| $ resp.) by taking 
\[ g  = K_\lambda^X , h = K_{ z_0 }^Y , \; z_0 \in \C_+ , \, \lambda  \in \C  . \]
Here $ K^X $ and $ K^Y $ stand for the reproducing kernels of the spaces $ X $ and $ Y $, resp. With this choice of $ g $ and $ h $ by the HB property of $ E_Y $ and $ E_X $ we have,
\[ | h ( z ) | \ge C \left| \frac{ E_Y ( z ) } {z - \overline{z_0} } \right|  , | g ( z ) | \le C \left| \frac{E_X ( z )}z \right| \] 
for $ z \in \C_+ $ with $ |z| $ large enough. It follows that in the upper half plane
\be\la{gonh} 
\frac gh = \frac{ E_X  }{ E_Y } \cdot (\mathrm{a\ bounded\ function} ) . \ee
Now, $ E_X / E_Y $  has bounded type in $ \C_+ $  by the assumption of the theorem, hence so does $ g/h $. Plugging this expression into (\ref{xiYstr}) we find that (i) $ \xi_Y $ is of bounded type in $ \C_+ $, for Cauchy transforms of \textit{finite} complex measures are of bounded type (for instance, because they are linear combinations of four Herglotz functions), and (ii) by  a crude estimate of the Cauchy integrals, there exists a constant $ C $ such that
\be\la{xiY1} | \xi_Y ( z ) | \le \frac C{ \Im z } \(  \left| \frac{ E_X ( z ) }{ E_Y ( z ) } \right|  + 1 \),  \; z \in \C_+ . \ee  

Similarly, choosing $ z_0 \in \C_- $ in the definition of $ h $ we obtain that in the lower half plane $ \C_- $
\[ \frac gh = \frac{ E_X^*  }{ E_Y^* } \cdot (\mathrm{a\ bounded\ function} ) , \]
and therefore $ \xi_Y $ has bounded type in $ \C_- $ as well, and there exists a constant $ C> 0 $ such that 
\be\la{xiY2} | \xi_Y ( z ) | \le \frac C{ \Im z } \(  \left| \frac{ E_X^* ( z ) }{ E_Y^* ( z ) } \right|  + 1 \) , \;  z \in \C_- . \ee

The assumption that $ E_X / E_Y $  is of bounded type in $ \C_+ $ also implies that in this halfplane $ E_X / E_Y  = e^{ ibz } \cdot (\mathrm{an\ outer\ function\  in \ } \C_+ ) $, $ b $ being a real constant. The rest of the proof splits into two cases, $ b > 0$ and $ b = 0 $. The case $ b < 0 $ is reduced to $ b > 0 $ by interchanging the notation for spaces $ X $ and $ Y $.

I. $ b > 0 $.  In this case $ E_X / E_Y $ decays exponentially along any ray $ \mathrm{arg }\, z = \beta $,  $ 0 < \beta < \pi $, and so does $ E^*_X / E^*_Y $  for  $ -\pi  < \beta < 0 $. Thus, $ \xi_Y (z ) = O \( \left| z \right|^{ -1 } \) $ along any ray $ \mathrm{arg}\,  z = \beta $, $ \beta \ne 0, \pi $, which implies that $ \xi_Y \equiv 0 $ by an elementary Lindelof theorem. On writing this for $ z = i \tau $, $ \tau > 0 $
\[ \frac {g(i\tau)}{h(i\tau)} \llangle \frac h{ i\tau - \cdot } , f_Y \rrangle_{ L^2 ( \R , \d \mu ) } =  \llangle \frac g{ i\tau - \cdot } , f_Y \rrangle_{ L^2 ( \R , \d \mu ) } . \] 
Consider the limit $ \tau \to +\infty $. The r.h.s.  is $ -i  \llangle g , f_Y \rrangle_{ L^2 ( \R , \d \mu ) } \tau^{ -1 } + o ( \tau^{ -1 } ) $, 
while on account of (\ref{gonh}) the l. h. s. is exponentially decaying. It follows that $ \llangle g , f_Y \rrangle_{ L^2 ( \R , \d \mu ) } = 0 $. Since $ g $ is an arbitrary reproducing kernel in $ X $, it follows that $ f_Y \in X^\perp $, which is a contradiction.

II. $ b =0 $. Let \[ \xi_X ( z ) = \frac 1{h_1(z)} \llangle \frac{ g_1 ( z ) h_1 - h_1 ( z) g_1 }{ z - \cdot } , f_X \rrangle_{ L^2 ( \R , \d \mu ) }  , \]
\[ g_1 = K^Y_\lambda , \; h_1 = K^X_{ z_0 } , z_0 \notin \R .\]
In the same way as $ \xi_Y $ does not depend on $ h $ and satisfies (\ref{xiY1}) and (\ref{xiY2}), the function $ \xi_X $ does not depend on $ h_1 $ and satisfies 
\be\la{xixest} \xi_X ( z ) \le \frac C{ | \Im z | } \left\{ \begin{array}{cc} \displaystyle{  \left| \frac{ E_Y ( z ) }{ E_X ( z ) } \right|  + 1 ,}  &  z \in \C_+ ; \cr \noalign{\vskip5pt} \displaystyle{ \left| \frac{ E_Y^* ( z ) }{ E_X^* ( z ) } \right|  + 1 ,} &  z \in \C_-  . \end{array}\right. \ee
From this and (\ref{xiY1}), (\ref{xiY2}) we have
\[\min \{ | \xi_X ( z ) | , | \xi_Y ( z ) | \} \le \frac c{ | \Im z | } . \]
The main point of the argument is the following

\begin{proposition}\la{maximumpr}
Let $ \xi_1 $, $ \xi_2 $ be entire functions of exponential type of minimal type such that 
\be\la{minxi} \min \{ | \xi_1 ( z )| , | \xi_2 ( z ) | \} \le \frac c{ | \Im z | } . \ee
Then one of them is zero identically.
\end{proposition}
 
In turn, the proof of this proposition is based on the following

\begin{lemma}\la{Carleman}
Let $ u $ be a nonnegative subharmonic function in $ \C $, $ u ( 0 ) > 0 $. Suppose that $ u $ is smooth in a vicinity of $ 0 $ and that the boundary of the set $ \{ z \colon u (z ) = 0 \} $ has zero planar Lebesgue measure. Then there exists a positive $ C $ such that for all $ \tau $ large enough 
\be\la{carle} \int \left( u \( e^\tau e^{ i \theta } \) \)^2 \d \theta \ge C \int_0^\tau \exp \( \int_0^{ \tau^\prime } \eta ( s ) \d s \) \d \tau^\prime , \ee
\[ \eta ( s )= \left\{ \begin{array}{cc} \frac 1{ p ( e^s) } , & p ( e^s ) < 1 , \cr 0 , &  p ( e^s ) = 1 , \end{array} \right. \]
\[ p ( R ) = \frac 1{2\pi } \left| \{ \theta \colon u ( R e^{ i \theta } ) > 0 \} \right| . \] 
\end{lemma}

Let us first prove the proposition taking the lemma for granted.

\begin{proof} Suppose by contradiction that neither $ \xi_1 \equiv 0 $, no $ \xi_2 \equiv 0 $. WLog, one can assume that $ | \xi_1 ( 0 ) | > 1 $, $ | \xi_2 ( 0 ) | > 1 $. The functions $ u_{ 1,2 }= \ln_+ | \xi_{ 1,2 } | $ are subharmonic and satisfy other assumptions of  lemma \ref{Carleman}, hence (\ref{carle}) holds for them.  In the notation of that  lemma the assumption (\ref{minxi}) means that there exists a $ C > 0 $ such that 
\be\la{b12} p_1 ( R ) + p_2 ( R ) \le 1 + \frac CR . \ee 
 Adding up estimates (\ref{carle}) for $ u_{ 1,2 } $ we get
\be \la{eta12} \begin{array}{ll} \displaystyle{ \int  \( \ln_+^2 \left| \xi_1 \( e^\tau e^{ i \theta } \)\right| +  \ln_+^2 \left| \xi_2 \( e^\tau e^{ i \theta } \) \right| \) \d \theta} \ge & \cr\noalign{\vskip6pt}  \displaystyle{ C \int_0^\tau \left[  \exp \( \int_0^{ \tau^\prime } \eta_1 ( s ) \d s \) +  \exp \( \int_0^{ \tau^\prime } \eta_2 ( s ) \d s \) \right] \d \tau^\prime } & \underbrace\ge \cr & \mathrm{concavity} \cr & \mathrm{of\ exponent} \cr \noalign{\vskip5pt} \displaystyle{ 2C \int^\tau_0 \exp \left[ \frac 12 \int_0^{ \tau^\prime }\( \eta_1 ( s ) +\eta_2 ( s ) \) \d s  \right] \d \tau^\prime  \ge \dots } & \end{array} \ee
If  $ \eta_1 ( s) \ne 0 $ and $ \eta_2 ( s ) \ne 0 $ then
\[ \begin{array}{ccc} \frac 12 \( \eta_1 ( s ) +\eta_2 ( s ) \) \! & \underbrace\ge  & \!\displaystyle{ 2 \( \frac 1{\eta_1 ( s ) } + \frac 1{\eta_2 ( s ) } \)^{ -1 }  = \frac 2{ p_1 ( e^s ) + p_2 ( e^s ) }}  \stackrel{(\ref{b12})}{\ge} \cr & \mathrm{mean\  arithmetic} & \cr &  \bigvee\!\! / & \cr & \mathrm{mean\ harmonic} & \end{array}  \]
\be\la{eta12s} \frac 2{ 1 + C e^{ -s } } . \ee
If, say, $ \eta_1 ( s ) = 0 $, then $ p_1 ( e^s )= 1 $, and $ \eta_2 ( s ) > C e^s $, again, by (\ref{b12}). This means that (\ref{eta12s}) holds with an appropriate $ C $ for all $ s $ large enough. Plugging this we continue the estimate (\ref{eta12}) for $ \tau $ large enough, 
\[ \dots \ge C \int_0^\tau \exp \( 2 \int^{\tau^\prime}_0 \frac {\d s }{ 1 + C_0 e^{ -s } } \) \d \tau^\prime = C \int_0^\tau \exp \( 2 \ln \( e^{\tau^\prime} + C_0 \) \) \d \tau^\prime \ge C e^{ 2 \tau } . \]
Thus, for $ R $  large enough,
\[ \int  \( \ln_+^2 \left| \xi_1 \( R e^{ i \theta } \)\right| +  \ln_+^2 \left| \xi_2 \( R e^{ i \theta } \) \right| \) \d \theta \ge C R^2 \]
which contradicts the assumption that both $ \xi_1 $ and $ \xi_2 $ have minimal type.
\end{proof}
\medskip

\noindent \textit{Proof of lemma \ref{Carleman}.} Notice first that the function $ p $ is separated from zero on the interval $ [ 0 , N ] $ for any $ N $. Indeed, suppose that $ p ( R_n ) \to 0 $ for some bounded sequence $ R_n $. On account of the fact that a subharmonic function is bounded above, we conclude that $ \int u ( R_n e^{ i \theta } ) \d \theta \to 0 $ which is a contradiction because $  2\pi u (0 ) \le \int u ( R_n e^{ i \theta } ) \d \theta $ by the definition of a subharmonic function and $ u ( 0 ) > 0 $ by assumption. Let us first establish (\ref{carle}) assuming $ u $ to be smooth. Define
\[ v ( R ) = \int  u^2 ( R e^{ i \theta } ) \d \theta . \] 
Applying the Green formula to $ \nabla ( u \nabla u ) $ over the ring $ R^\prime < |z| < R $ we find 
\begin{eqnarray}\la{Green} \left. \int u \frac{\partial u}{\partial r} r \, \d \theta \right|_{ r = R } -  \left. \int u \frac{\partial u}{\partial r} r \, \d \theta \right|_{ r = R^\prime } =  \int_{ R^\prime < | z | < R } \( \left| \nabla u \right|^2 + u \Delta u \) \d S \nonumber \\ \ge \int_{ R^\prime < | z | < R } \left| \nabla u \right|^2 \d S \end{eqnarray}
on account of the subharmonicity of $ u $. Dividing this inequality by $ R - R^\prime $ and taking the limit $ R^\prime \to R $, we find
\be\la{diffineq} \frac 12 \frac\d{ \d R } \( R v^\prime ( R ) \) \ge \int R \left| \nabla u ( R e^{ i \theta } ) \right|^2 \d \theta =  \int \left[ R \(  \frac{ \partial u }{ \partial R } \)^2 +  \frac 1R \(  \frac{ \partial u }{ \partial \theta } \)^2 \right]  \d \theta . \ee
All $ R $'s fall into one of two cases, 

(\textit{i})  $ u ( R e^{i \theta } ) >0 $ for a. e. $ \theta $;

(\textit{ii}) $ u ( R e^{ i \theta } ) = 0 $ on a set of positive measure.

 In the case (\textit{ii}) one can estimate $ v^\prime $  via the r.h.s. of the obtained inequality,
\be\la{v} v^\prime ( R ) = 2  \int u \frac{\partial u}{\partial r} \, \d \theta \le  \int \left[ \beta  u^2 +  \frac 1\beta \(  \frac{ \partial u }{ \partial r } \)^2 \right]  \d \theta \le \dots \ee
Since $ u $ is smooth, the set of $ \theta $'s for which $  u ( R e^{ i \theta } ) > 0 $ is an open subset in $ [ - \pi , \pi ] $. This subset is proper as we are in the case (\textit{ii}) and one can assume WLog that $ u ( \pm R ) = 0 $. Then for each interval, $ I $, of this set one has
\be\la{Dirichlet} \int_I u^2 ( R e^{ i \theta } ) \d \theta \le \( \frac { |I|}\pi \)^2 \int_I \( \frac{ \partial u }{ \partial \theta } \)^2 \d \theta , \ee
(in fact, for any $ u $ smooth on $ I $ and vanishing at the ends of this interval) as $ \( \pi / |I| \)^2 $ is the first eigenvalue of the Dirichlet problem for this interval. Substituting the integral in the r.h.s. by the integral over $ [-\pi , \pi ] $ and summing these estimates over all the intervals $ I $, we find
\[ \int u^2 ( R e^{ i \theta } ) \d \theta \le 4 \( p ( R) \)^2 \int \( \frac{ \partial u }{ \partial \theta } \)^2 \d \theta . \]
One can now continue the estimate (\ref{v}),
\[  \dots \le 4 \( p ( R) \)^2 \beta \int \( \frac{ \partial u }{ \partial \theta } \)^2 \d \theta +  \frac 1\beta \int \(  \frac{ \partial u }{ \partial r } \)^2 \d \theta \le \dots \]
Pick $ \beta $ so that the $ R \left| \nabla u \right|^2 $ is obtained in the integrand in the r.h.s., that is, let $ \beta = 1/( 2 R p ( R )) $,
\[ \dots \le   2 p ( R) \int \left[ R \(  \frac{ \partial u }{ \partial R } \)^2 +  \frac 1R \(  \frac{ \partial u }{ \partial \theta } \)^2 \right]  \d \theta \stackrel{(\ref{diffineq})}{\le} p ( R ) \( R v^\prime ( R ) \)^\prime . \]
A change of variables $ R = e^\tau $, $ \nu ( \tau ) \colon = v ( e^\tau ) $, applies to the resulting inequality to give
\[ \frac 1{ p ( e^\tau ) } \frac{\partial \nu }{ \partial \tau } \le \frac{\partial^2 \nu }{ \partial \tau^2 }  . \]
This ineqality is established for $ \tau $ such that $ v ( e^\tau e^{ i \theta } ) $ vanishes on a set of positive measure. In the case (\textit{i}) one can only say from (\ref{diffineq}) that $ \frac\d{ \d R } \( R v^\prime ( R ) \) \ge 0 $. Let us combine these by setting 
\[ \eta ( \tau ) = \left\{ \begin{array}{cc} \displaystyle{\frac 1{ p ( e^\tau ) }} , & p ( e^\tau ) < 1 ; \cr \noalign{\vskip5pt}0, & p ( e^\tau ) = 1. \end{array} \right. \]
Then
\[ \eta ( \tau ) \nu^\prime ( \tau ) \le \nu^{ \prime \prime } ( \tau ) \]
holds for all real $ \tau $. To solve this inequality we notice that $ \nu^\prime $ does not vanish. This follows from (\ref{Green}) if we plug $ R^\prime = 0 $ there. Thus, $ \( \ln \nu^\prime ( \tau ) \)^\prime  \ge \eta ( \tau ) $ and for $ \tau \ge\tau_0 $
\be\la{smoothsub} \nu ( \tau ) \ge \nu^\prime ( \tau_0  ) \int_{ \tau_0 }^\tau e^{ \int_{ \tau_0 }^t \eta ( s ) \d s } \d t . \ee
The lemma is proved for smooth $ u $.

Now let $ u $ be an arbitrary subharmonic function satisfying the conditions of the lemma.   Define for $ \von > 0 $
\[ u_\von ( z ) = \frac 1{ \von^2 } \int u ( z +\zeta ) \omega \( \frac \zeta\von \) \d S_\zeta \]
where $ \omega $ is an arbitrary $ C_0^\infty ( \C ) $-function such that $ \omega ( z ) = 0 $ for $ | z| > 1 $, $ 0 < \omega ( z ) \le 1 $ for $ |z|<1 $, and $ \omega (z ) = 1 $ for $ z $ in a vicinity of zero. The function $ u_\von $  is subharmonic (it is a convex combination of subharmonic functions) and smooth for each $ \von > 0 $, hence satisfies (\ref{smoothsub}). It is obvious ($ u \ge 0 $!) that the sets $ {\mathcal M}_\von \colon =\{z \colon u_\von ( z ) = 0 \} $ increase by inclusion in $ \von $ and each of them is contained in $ {\mathcal M} \colon = \{z \colon u ( z ) = 0 \} $. Notice that $ {\mathcal M} \setminus \cup_\von{\mathcal M}_\von $ has zero Lebsgue measure. Indeed if $ z $ is an interior point of $ {\mathcal M} $ then $ z \in {\mathcal M}_\von $ for $ \von $ small enough, hence the complement of $ \cup_\von{\mathcal M}_\von $ in $ {\mathcal M } $ belongs to the boundary of $ {\mathcal M } $ which has measure zero by assumption. It follows that for a.e. $ R $ the set $ \( {\mathcal M} \setminus \cup_\von{\mathcal M}_\von \) \cap \{ z\colon |z| = R \}  $ has zero angular Lebesgue measure, which means that $ p_\von (R ) \downarrow p ( R ) $ for a.e. $ R $ as $ \von \downarrow 0 $. It follows that $ \eta_\von ( \tau ) \uparrow \eta ( \tau ) $ for a.e. $ R $ as $ \von \downarrow 0 $ (on account of the fact that $ \eta_\von (\tau ) = 0 $ whenever $ \eta ( \tau )= 0 $). The function $ \eta $ is locally bounded  because $ p $ is locally separated from zero, as noted at the beginning of the proof. By the dominated convergence, one can pass to the limit $ \von \downarrow 0 $ in the integral in inequality (\ref{smoothsub}). Picking $ \tau_0 $ negative large enough, one can achieve that $ \nu $ is smooth at $ \tau_0 $ in view of the assumed smoothness of $ u $ at zero. Then it is easy to see that $ \nu_\von^\prime ( \tau_0 ) \to \nu^\prime ( \tau_0 ) $, and one can pass to the limit $ \von \downarrow 0$ in the r.h.s. of (\ref{smoothsub}). In the l.h.s. one can also pass to the limit, for $ \limsup u_\von (z ) \le u ( z ) $ by upper semicontinuity of $ u $, hence $ \limsup \nu_\von ( \tau ) \le \nu ( \tau ) $ (apply the dominated convergence to integrals of squares of  $ {\tilde u}_\von ( z ) \colon = \sup_{ \von^\prime > \von } u_{ \von^\prime } ( z ) $), and (\ref{carle}) is proved in full generality. \hfill $ \Box $

\medskip

\noindent \textit{End of proof of theorem \ref{lattice}}. To apply proposition \ref{maximumpr} to $ \xi_X $, $ \xi_Y $ it remains to check that they are of minimal type. Indeed,  $ E_X / E_Y $ being outer in $ \C_+ $ implies that $ \ln | E_X ( z ) / E_Y ( z ) | = o ( |z| ) $ on any ray $ \mathrm{arg} \, z = \theta $, $ 0 < \theta < \pi $ (this is easy to check from the integral representation for an outer function). It now follows from (\ref{xiY1}), (\ref{xiY2}) that $ \xi_Y (z) = o( |z| ) $ on any ray $ \mathrm{arg} \, z = \theta $, $ \theta \ne 0 ,\pi $. An elementary Fragmen-Lindel\"of theorem implies that $ \xi_Y $ is of minimal type, and the same is true of $ \xi_X $ by an analogous argument. Recall that the functions $ \xi_X $, $ \xi_Y $ depend on a single parameter $ \lambda \in \C $, via the functions  $ g $ and $g_1 $ in their definitions which are the reproducing kernels of the respective de Branges spaces.  By proposition \ref{maximumpr} for each $ \lambda \in \C $ either $ \xi_X $, or $ \xi_Y $ vanishes identically. Suppose that $ \xi_Y \equiv 0 $ for all $ \lambda $ from an uncountable set $ \mathcal N  \subset \C_+ $, that is,
\[ \frac {K_\lambda^X  ( z )}{h(z)} \llangle \frac h{\cdot - z } , f_Y \rrangle_{ L^2 ( \R , \d \mu ) } - \llangle \frac{ K_\lambda^X }{ \cdot - z } , f_Y \rrangle_{ L^2 ( \R , \d \mu ) } = 0 \]
for all $ z \in \C \setminus \R $, $ \lambda \in \mathcal N $, $ h \in Y $.

The argument will proceed as follows. We shall show that for any ray $ \operatorname{arg} z = \theta $, $ 0 < \theta \le \pi/2 $
\be\la{exey} \frac {E_Y (z )}{ E_X ( z ) } \longrightarrow 0 ,\; |z| \to +\infty , \ee
this will imply that $ \xi_X \equiv 0 $  \textit{for all} $ \lambda \in \C $, which in turn would mean that the \textit{inverse} of the ratio in (\ref{exey}) vanishes at infinity along the same rays, which is a contradiction.

Fix a $ \lambda \in \mathcal N $ so that $  \llangle K_\lambda^X , f_Y \rrangle \ne 0 $ which is possible since $ f_Y $ is not orthogonal to $ X $, and the linear span of $ K_\lambda^X $, $ \lambda \in \mathcal N $, is the whole of $ X $ for $ \mathcal N $ being uncountable. Consider $ z = R e^{ i \theta } $, $ 0 < \theta \le \pi/2 $. On multiplyng the identity above by $ R $ and passing to the limit $ R \to \infty $ we find
\[  \begin{array}{rcl} \displaystyle{\frac {K_\lambda^X  ( z )}{h(z)}} \big( & \!\!\!\!\!\! \langle h , f_Y \rangle & \!\!\!\!\!\! \left. e^{ -i\theta } + o ( 1 ) \) = \langle K_\lambda^X , f_Y \rangle e^{ - i \theta } + o ( 1 ) . \cr & \!\!\!\!\!\!\!\!\!\| & \cr & \!\!\!\!\!\!\!\!\! 0 & \!\!\!\!\!\!\!\!\!\!\! \mathrm{as }\,\,  h \in Y , \, f_Y \perp Y \end{array} \]
Thus,
\[ \frac{ h( z )  }{K_\lambda^X ( z ) } \longrightarrow 0 , \; R \to +\infty .\]
On choosing $ h $ to be a reproducing kernel of $ Y $ at a point in $ \C_+ $ and taking into account that $ \lambda \in \C_+ $ we find that (\ref{exey}) holds. 
By (\ref{xixest}) this means that for each $ \lambda \in \C $ the function $ \xi_X \( R e^{ i \theta } \) \to 0 $ for all $ \theta \ne 0 $, $ |\theta| < \pi/2 $, and an elementary Fragmen-Lindel\"of theorem implies that $ \xi_X \equiv 0 $, that is,
\[  \frac { K^Y_\lambda ( z ) }{h_1(z)} \llangle \frac{ h_1}{ \cdot - z }  , f_X \rrangle_{ L^2 ( \R , \d \mu ) } - \llangle \frac{K^Y_\lambda }{ z - \cdot }  , f_X \rrangle_{ L^2 ( \R , \d \mu ) } = 0 \]
for all  $ z \in \C \setminus \R $, $ \lambda \in \C $, $ h_1 \in X $. Arguing as above with $ h_1 $ taken to be a reproducing kernel of $ X $ we find that for $ z = R e^{i \theta } $, $ 0 < \theta \le \pi/2 $,
\[ \frac {E_X (z )}{ E_Y ( z ) } \longrightarrow 0 ,\; R \to +\infty \] which contradicts (\ref{exey}). The contradiction obtained means that either $ \xi_Y $ vanishes for at most countably many $ \lambda \in \C_+ $, or the basic assumption of this section about existence of $ f_X $ and $ f_Y $  fails. In the latter case the theorem is proved, in the former $ \xi_X $ vanishes for uncountably many $ \lambda \in \C_+ $, and one can notice that the argument above is symmetric with respect to interchange of $ X $ and $ Y $. Hence the basic assumption fails in any case, and the proof is complete. \hfill $ \Box $

\section{Necessary facts about de Branges spaces}\la{facts}

\begin{remark} Let $ F $ and $ G $ be real entire functions. Then the following are equivalent,

(i) $ E = F + i G $ is an HB function, 

(ii) $ F/G $ is a non-constant Herglotz function, that is, $ \Im \( F / G \) > 0 $ in the upper half plane. \end{remark}

Another way of writing the condition (ii) is $ \frac 1i \llangle J \Theta ( z ) , \Theta ( z ) \rrangle > 0 $ for $ z \in \C_+ $ where $ \Theta  = \( \begin{array}{c} F \cr G \end{array} \) $. 

\begin{definition}
Let $ E $ be an HB function. The de Branges space, $ \HE $, is the linear set of entire functions $ f $ satisfying $ f/E \in H^2 $, $ f^* / E \in H^2 $ endowed with the metric $ \len f \rin = \len f/E \rin_{ L^2 ( \R ) } $.
\end{definition}

Elements of a de Branges space can be characterized by a "pointwise" condition given by  lemma \ref{deBrchar}. 

\medskip 

\noindent \textit{Proof of lemma \ref{deBrchar}.} Only the implication $ (i) + (ii) \Rightarrow f \in \HW $ is nontrivial. We first show that $ f/W \in H^2 $. Indeed, define 
\[ h_\tau ( z ) = \frac{ i \tau }{ z + i \tau }\frac{ f ( z ) }{ W  ( z ) } . \]
Then $ h_\tau \in L^2 ( \R ) $ for all $ \tau > 0 $, and $ | h_\tau ( z ) | \le C \left| z \right|^{ -1 } \( \Im z \)^{ -1/2 } $ by assumption (ii). By elementary contour integration it follows that $ \frac 1{2 \pi i } \int \frac { h_\tau ( t ) }{ t - z } \mathrm{d} t = h_\tau ( z) $, that is, $ h_\tau \in H^2 $ for any $ \tau > 0 $. Since $ h_\tau \to f/W $ in $ L^2 ( \R ) $ as $ \tau \to + \infty $ we infer that $  f/W \in H^2 $. The inclusion $ f^*/W \in H^2 $ is checked similarly. \hfill $ \Box $

\medskip

\begin{theorem}\la{axiomatic}
Let $ X $ be a Hilbert space of entire functions such that (i) $ f \mapsto f ( w ) $ is a bounded functional on $ X $ for any $ w \in \C $; (ii) $ f^* \in X $ whenever $ f \in X $, and $ \len f^* \rin = \| f \| $; (iii) if $ f \in X $ and $ f ( w ) =0 $ for a $ w \notin \R $, then the function $ f ( z ) / ( z - w ) $ belongs to $ X $, and 
\[ \len \frac{ z - \overline w}{ z - w } f ( z ) \rin_X = \len f \rin_X . \]
Then for any $ a \in \C_-$ the function $ ( z - a ) K^X_{ \overline a } ( z ) $, $ K^X_{ \overline a} $ being the reproducing kernel for the space $ X $ at the point $ \overline a $, is HB, and there exists a number $ c_a $ such that $ X = {\mathcal H}{ (  E_a ) } $ (as Hilbert spaces) for $ E_a ( z ) = c_a ( z - a ) K^X_{ \overline a } ( z ) $. 
\end{theorem}

Once the statement is known, the proof (\cite[Theorem 23]{deBr} or \cite[Chapter 6.1]{DymMcKean}) of this theorem is a rather straightforward algebra.

\begin{remark}
Let $ E $ be an HB polynomial of degree $ n > 1 $, and let $ \alpha \in \R $ be such that $ e = E + e^{ i\alpha } E^* $ is a polynomial of degree less than $ n $ (such an $ \alpha $ obviously exists). Then $ e $ is orthogonal in $ \HE $ to all polynomials of degree $ \le n-2 $.  
\end{remark}

Indeed, for any polynomial $ p $ of degree $ \le n-2 $
\[ \langle p , e \rangle = \int_\R \frac{ p ( k ) }{ E ( k ) } \d k  + e^{ -i\alpha } \int_\R \frac{ p ( k ) }{ E^* ( k )} \d k = 0. \] 

\begin{definition}
An HB function $ E $ and the corresponding de Branges space $ \HE $ are called regular if any of the following equivalent conditions holds

(i) \[ \frac 1{ ( z + i ) E ( z ) } \in H^2 , \]

(ii) for any $ f \in \HE $ and some $ z_0 \in \C $ \[ \frac{ f( z ) - f ( z_0 )}{ z - z_0 } \in \HE . \]  \end{definition}

Notice that if \textit{(ii)} holds for some $ z_0 \in \C $ then it does for any.  

The following theorem is a de Branges version of the spectral theorem for the operator $ D_\alpha $ (theorem \ref{Dalpha}). In fact, \textit{if we know} that the function $ E$ corresponds to a regular canonical system, it follows immediately from theorem \ref{Dalpha} by applying the Fourier transform $ \cU $. We use it in the solution of the inverse problem when reconstructing  the second column of the monodromy matrix from the first one  (theorem \ref{restore}). 

\begin{theorem} \la{orthog}\cite[Theorem 22]{deBr}
Let $ G $ be an HB function without real zeroes, and let $ \alpha \in \R $. Define $ \mathcal N = \{ t \in \R \colon \mbox{arg } G ( t ) = \alpha (\mbox{mod } \pi ) \}$ and let $ K_w $ be the reproducing kernel of the space $ \HG $. Then either 

\textit{(i)} The set $ \{ K_t \}_{ t \in \mathcal N} $ is an orthogonal basis in $ \HG $, 

or 

\textit{(ii)} The function $ e^{ i\alpha } G - e^{ -i\alpha } G^* $ belongs to $ \HG $. The set $ \{ K_t \}_{ t \in \mathcal N} $  has one-dimensional orthogonal complement spanned by this function.
\end{theorem} 

\begin{proof}
It suffices to consider $ \alpha = 0 $, the general case reduces to this one by considering the function $ e^{- i \alpha } G $. Then $ K_t ( z ) = (\mbox{const}) G_- (  z ) / ( z - t ) $ for $ t \in \mathcal N $. It is obvious that $ K_t \perp K_\tau $ for different $ t , \tau \in \mathcal N $.  Let $ Y = \bigvee_{  t \in \mathcal N} K_t $ and $ P_Y $ be the orthogonal projection on $ Y $. Then for an arbitrary $ w \notin \mathcal N $
\begin{eqnarray} ( P_Y K_w ) ( z ) = \sum_{ t \in \mathcal N } \frac{ ( K_w , K_t ) }{ \len K_t \rin^2 } K_t ( z ) =\sum_{ t \in \mathcal N } \frac{ \overline{ K_t (w)} }{ - \frac 1\pi G_+ ( t ) \dot{G}_- ( t ) } K_t ( z ) = \nonumber \\ - \frac 1\pi \sum_{ t \in \mathcal N } \frac{ G_+ (t)  }{ \dot{G}_- ( t ) } \frac 1{ \( \overline w - t \) ( z - t ) } \overline{G_- ( w )} G_- ( z )  . \la{PYrepro} \end{eqnarray}
On the other hand $ G_+ / G_- $ is a meromorphic Herglotz function whose set of poles coincides with $ \mathcal N $, hence for $ z \in \C_+ $
\[ \frac{ G_+ ( z ) }{ G_- ( z ) } = a + bz + \sum_{ t \in \mathcal N} \frac{ G_+ ( t ) }{ \dot{G}_- ( t ) } \( \frac 1{ z - t} + \frac t{ 1 + t^2 } \) . \]
Subtracting  the same decomposition at a point $ \overline w $ for $ w \in \C_- $ we find
\[ \frac { G_+ ( z ) G_- ( \overline w ) - G_- ( z ) G_+ ( \overline w ) }{ G_- ( z ) G_- ( \overline w ) } = b ( z - \overline w ) +  \sum_{t \in \mathcal N} \frac{ G_+ ( t ) }{ \dot{G}_- ( t ) } \( \frac 1{ z - t} - \frac 1{ \overline w - t} \) , \] 
that is,
\[ K_w ( z ) = \frac b\pi  G_- ( z ) G_- ( \overline w ) - \frac 1\pi \sum_{ t \in \mathcal N } \frac{ G_+ (t)  }{ \dot{G}_- ( t ) } \frac 1{ \( \overline w - t \) ( z - t ) } G_- ( \overline w ) G_- ( z ) . \]
Comparing this and (\ref{PYrepro}) we find that 
\[ K_w = \frac b\pi G_- ( \overline w ) G_- + P_Y K_w \]
for arbitrary $ w \in \C_- $. The set of $ K_w $'s  with $  w \in \C_- $ is total in $ \HG $, hence \textit{(i)} holds if $ b = 0 $, \textit{(ii)} in the opposite case.
\end{proof}

\section{Solutions}

{\small\textbf{Exercise \ref{counterex}}. 

\[ \begin{pmatrix} \frac 13 + 3i & -1 - 2i \cr 1 + \frac 43 i & -1-i \end{pmatrix} . \]  An arbitrary fractional linear mapping, $ S $, of $ \C_+ $ to a disc, $ D $, has the form $ S = T \circ \xi $ where $ T = z_0 - \frac 1{ z + i \tau } $, $ z_0 \in \C $, $ \tau > 0 $, and $ \xi $ is an $ SL ( 2 , \R ) $--transform. The disc $ D $ lies in $ \C_+ $ if and only if $ \Im z_0 > 0 $. Let $ \Im z_0 = -1 $, so $ T $ is not $ J $--contractive. The last two conditions in the exercise mean that $ T^{ -1 } ( 0 ) , T^{ -1 } ( \infty ) \in \C_- $. The latter is satisfied automatically because $ \tau > 0 $, and the former is equivalent to $ \Im \frac 1{z_0} - \tau < 0 $. This can be achieved for arbitrarily small $ \tau $ by choosing $ \Re z_0 $ large enough. The remaining two conditions mean that 
\[ \Im S ( 0 ) = -1+ \frac \tau{ \xi^2 ( 0 ) + \tau^2 } > 0 , \; \Im S ( \infty ) = -1 + \frac \tau{ \xi^2 ( \infty ) + \tau^2 } > 0 , \]
and these are achieved by choosing $ \xi ( \infty ), \xi ( 0 ) $ small enough and then optimizing the choice of $ \tau $. }

\section{Comments on literature}

\textbf{Section \ref{examples}}. 2. More on relations of de Branges spaces and canonical systems to the inverse spectral theory for the Schr\"odinger operator can be found in \cite{Remling}. 

3. In fact, a canonical system can be defined for any string. Let us sketch the construction, the details can be found in \cite{KWW}. Let $ m $ be a finite measure on $ ( 0 , l ) $, $ m ( \{ l \} ) = 0 $, then the string equation is 
\[ - \frac{ d^2 y }{ dm dy } = \lambda y \]
with an appropriately understood derivative in the l.h.s. Define the measure $ d\xi ( t ) = dt + d m ( t ) $ and the function $ \xi ( t ) = t +  m ( t ) $, $ m ( t ) = m ( [ 0 , t )) $. Obviously, both $ m $ and the Lebesgue measure are absolutely continuous with respect to $ d \xi $, and the relations $ h_1 ( \xi ( t ) ) = \frac{ dt}{ d\xi ( t ) } $, $  h_2 ( \xi ( t ) ) = \frac{ d m ( t ) }{ d \xi ( t ) } $ define functions $ h_{ 1,2 } $ on $ \xi ( [ 0 , l ) ) $ a.e. with respect to the Lebesgue measure. If we set $ h_1 ( x ) = 0 $, $ h_2 ( x ) = 1 $ for $ x \in [ 0 , l + m ( l ) ] \setminus \xi ( [ 0 , l ) ) $, then $ h_{  1, 2 }$ are defined a.e. on $ [ 0 , l + m ( l ) ] $. Then the canonical system corresponding to the string is $ ( \cH , L ) $ with $ \cH = \operatorname{diag} ( h_1 , h_2 ) $, $ L = l + m ( l ) $. 
When $ m $ is an absolutely continuous measure, this construction reduces to the one given in the main text. 

\textbf{Section \ref{operator}}. The material of this section is not in the book but can be considered standard and is beyond attributing to specific authors. We believe most of it can be found in one form or another in \cite{HassiSnooWinkler} and references therein. For instance, the argument we used to prove the symmetricity of $ D $ in theorem \ref{operatorLinf} appears to be similar to the one in their Lemma 7.8, but it turned out to be easier to work out the necessary assertions on our own rather than extract them from those papers. In part this is because they prefer to deal with non-densely defined operators and use the clumsy terminology of selfadjoint relations, which we feel superfluous, and in part because the required arguments are rather straightforward. 

Formulae (\ref{rho}) and (\ref{q}) in Example 4 expressing the Jacobi parameters in terms of the corresponding Hamiltonian can be found, for instance, in \cite{Katz}.  

\textbf{Section \ref{direct}}. The argument in Exercise \ref{discspect} is very common in spectral theory of ODE's, see e. g. \cite{Naimark}.

\textbf{Section \ref{finreg}}. The proof of the lower estimate for the type in Theorem \ref{exptypeM} is virtually the one in the book. The proof of the upper estimate follows the one in \cite[Theorem X]{deBrII} except for he uses an explicit reduction of the matrix $ J \cH ( x) $ instead of the assertion that constitutes our lemma \ref{2by2} (ii). The proof of the upper estimate in the book is different and considerably more involved. Essentially, de Branges establishes there that, in our notation, $ p ( x ) \le \sqrt{ \det \int_0^x \cH ( s ) \d s } $ which appears weaker than the required estimate but in fact implies it by the multiplicative property of the monodromy matrix. The said inequality follows in turn from another one,
\be\la{deBrproof}  i^{ -1 } M^* ( a , \lambda ) J M (a, \lambda ) \ge e^{ 2 p ( a ) \Im \lambda } i^{ -1} J \ee
holding for all $ \lambda \in \C_+ $,  by developing it at $ \lambda = 0 $ (\cite[Lemma 9]{deBr}). The proof of (\ref{deBrproof}) is based on the fact that $ e^{ \pm i p(a) z } $ is an associated function for $ \HE $ and then uses an extended version  \cite[Theorem 27]{deBr} of the assertion that forms Theorem \ref{restore} in the present text. The analog of (ii) in Theorem \ref{restore} in this extended version is precisely (\ref{deBrproof}). 

\textbf{Section \ref{borgtheorem}}. There is an assertion (Theorem 24) in \cite{deBr} which can be considered  a "Borg theorem for abstract de Branges spaces". It says that 

\begin{theorem} \la{AbstrB}
For any HB functions,  $ G $, $ \tilde G $ having no real zeroes and such that
\[ \{ t \colon \Re  G ( t ) = 0 \} = \{ t \colon \Re  \tilde G ( t ) = 0 \} ,  \{ t \colon \Im  G ( t ) = 0 \} = \{ t \colon \Im  \tilde G ( t ) = 0 \} \]
there exists a real entire function $ S $ such that the mapping $ f \mapsto S f $ is an isomorphism of $ \HG $ onto $ {\mathcal H} (\tilde G) $.
\end{theorem}

This theorem could have been used to derive the coincidence of the functions $ E $ and $ \tilde E $ in our argument. We have prefered another path here because the direct argument is simpler (it does not exploit the Krein theorem). Notice that the actual formulation of theorem \ref{AbstrB}  in the book looks a little more general in that it allows for real zeroes of $ G $.  This generalization is irrelevant in the context of the Schr\"odinger operator because the corresponding HB functions do not have real zeroes. 

\textbf{Appendix I}. The argument  in the proof of theorem \ref{restore} showing that $ \Phi_- / \Theta_- $ coincides with the regularized sum of its principal parts could have been shortened by a reference to general properties of functions of the Cartwright class. As mentioned in the Introduction, theorem \ref{restore} is essentially \cite[Theorem 27]{deBr} in the regular case. 

The identity (\ref{sysphi}) can be easily generalized by considering the scalar product of $ \Phi_\lambda $ and $ \Phi_z $  for arbitrary $ z $. The resulting formula will contain two spectral parameters $ \lambda , z $ so that (\ref{sysphi}) corresponds to $ z = 0 $. A scary-looking bilinear identity in \cite[Theorem 27]{deBr} when applied to reproducing kernels,  is precisely this formula.

Instead of using the identity (\ref{Jcontr}) to prove (\textit{ii}) in theorem \ref{restore}, one could refer to a theorem of Sodin \cite{Sodin} which says, in the setup of proof of theorem \ref{restore}, that if any three of the quotients, $ \Theta_+ / \Theta_- $, $ \Phi_+ / \Phi_- $, $ \Phi_\pm / \Theta_\pm $ are Herglotz functions then the matrix-function $ N ( \lambda) $ satisfies (ii). Since $ \Theta_+ / \Theta_- $ is Herglotz because $ E $ is HB, $ \Phi_+ / \Theta_+ $ is Herglotz by (\ref{thplu}), it remains to show that $ \Phi_- / \Theta_- $ is Herglotz. That can be done by observing that the residues of this meromorphic function at zeroes, $ t $, of $ \Theta_- $ are negative, $ \( \Theta_+ ( t ) \dot\Theta_- ( t )  \)^{ -1 } < 0 $, and using an appropriate growth argument to establish that the difference between $ \Phi_- / \Theta_- $ and the regularized sum of its main parts at poles is at most linear, but overall it would probably not be simpler than the argument in the text.   

\textbf{Appendix II}. Theorem \ref{isometry} is Theorems V.A and V.B in \cite{deBr0} and loosely corresponds to Theorem 31 in the form given by Problem 90 (the necessity - if embedding holds then there exists the corresponding contractive function $ A $) and Theorem 32 (sufficiency) in the book. 

\textbf{Appendix III}. The original proof of the lattice property in \cite[Section 35]{deBr} contains an error. The error is in lemma 7 which claims, in our notation, the convexity of the function $ v ( R ) $ for $ u = \log_+ | f | $, $ f $ being entire, with respect to the variable (see (\ref{carle}))
\[ m( R ) =  \int_1^R \exp \( \int_1^r \frac { \d s}{ s p (s) } \) \frac{ \d r }r  . \]
The counterexample is $ f ( z ) = z $ (in fact, any polynomial $f$). The only other exposition of the proof we are aware of - \cite[Section 6.5]{DymMcKean} - reproduces the mistake. The reason for it is that de Branges and Dym-McKean tacitly assume that $ p ( R ) < 1 $ which allows to estimate the $ L^2 $-norm of the function $ u $ from above by that of $ \frac{\partial u}{ \partial \theta } $ (see (\ref{Dirichlet})). The latter is obviously impossible if $ u $ does not vanish on the circle. The case $ p ( R ) = 1 $ is, of course, conceptually simpler, for it corresponds to the situation when one of the functions is not small on the whole of the circle. Nevertheless, the gap appears not to have been noticed and mended so far, although the way to correct the mistake is not difficult and is apparently known to some experts in the field.


\begin{thebibliography}{15}

\bibitem{deBr} Louis de Branges, \textit{Hilbert spaces of entire functions}, Prentice-Hall, NJ (1968).

\bibitem{Marchenko} V. A. Marchenko, \textit{Sturm-Liouville operators and applications}, AMS Chelsea Publishing, Providence, RI (2011).

\bibitem{DymMcKean} H. Dym, H. P. McKean, \textit{Gaussian processes, function theory, and the inverse spectral problem}, Academic Press, New York (1976).
  
\bibitem{Rudin} W. Rudin, \textit{Real and complex analysis}, McGraw-Hill Series in Higher Mathematics. McGraw-Hill Book Company, New York etc. (1966).

\bibitem{Gofman} K. Hoffman, \textit{Banach spaces of analytic functions},  Prentice-Hall, NJ (1962). 

\bibitem{Remling} C. Reming, \textit{Schr\"odinger operators and de Branges spaces}, J. Funct. Anal. \textbf{196}(2), 323 -- 394 (2002).

\bibitem{HassiSnooWinkler} S. Hassi, H. de Snoo and H. Winkler, \textit{Boundary-value problems for two-dimensional
canonical systems}, Integral Eq. Oper. Theory \textbf{36}(4), 445--479 (2000).  

\bibitem{KWW} M. Kaltenb\"ack, H. Winkler, H. Woracek, \textit {Strings, dual strings and related canonical systems}, Math. Nachr. \textbf{280} (13-14), 1518--1536 (2007).

\bibitem{Katz} I. S. Kac, \textit{Inclusion of Hamburger's power moment problem in the spectral theory of canonical systems}, Journal of Mathematical Sciences (New York) \textbf{110}:5, 2991--3004 (2002).

\bibitem{Naimark} M. A. Naimark, \textit{Linear differential operators}, Parts I and II, Frederick Ungar, New York (1967).

\bibitem{Sodin} M. Sodin, \textit{A remark to the definition of Nevanlinna matrices}, Matematicheskaya fizika, analiz, geometriya \textbf{3}, vol. 3, 412--422 (1996).

\bibitem{deBr0} Louis de Branges, \textit{Some Hilbert spaces of entire functions}, Transactions AMS  \textbf{96}, 259--295 (1960).

\bibitem{deBrII}  Louis de Branges, \textit{Some Hilbert spaces of entire functions. II}, Transactions AMS  \textbf{99}, 118--152 (1961).


\end{thebibliography}
\end{document}